\documentclass[11pt,fleqn]{amsart} 

\usepackage[usenames,dvipsnames,condensed]{xcolor}

\usepackage{amsthm}
\usepackage{amsfonts}
\usepackage[english]{babel}
\usepackage[usenames]{xcolor}
\usepackage{graphicx}
\usepackage{soul}
\usepackage{stfloats}
\usepackage{morefloats}
\usepackage{cite}
\usepackage{lscape}
\usepackage{epstopdf}
\usepackage{braket}
\usepackage[lite]{amsrefs}
\usepackage{mathbbol}
\usepackage{tikz}
\usepackage{tikz-cd}
\usepackage{mathtools}

\usepackage{amsmath,amstext,amsopn,amsfonts,eucal,amssymb}
\usepackage{graphicx,wrapfig,url}

\newcommand\Z{{\mathbb Z}}

\newtheorem{theorem}{Theorem}[section]
\newtheorem{corollary}[theorem]{Corollary}
\newtheorem{lemma}[theorem]{Lemma}
\newtheorem{proposition}[theorem]{Proposition}
\newtheorem{definition}[theorem]{Definition}

\newtheorem{example}[theorem]{Example}

\newtheorem{remark}[theorem]{Remark}

\usepackage{soul}

\begin{document}

\title{Deformation Cohomology for Braided Commutativity} 

\author{Masahico Saito} 
\address{Department of Mathematics and Statistics, 
	University of South Florida, Tampa, FL 33620} 
\email{saito@usf.edu} 

\author{Emanuele Zappala} 
\address{Department of Mathematics and Statistics, Idaho State University\\
	Physical Science Complex |  921 S. 8th Ave., Stop 8085 | Pocatello, ID 83209} 
\email{emanuelezappala@isu.edu}

\maketitle

\begin{abstract}
		Braided algebras are algebraic structures consisting of an algebra endowed with a Yang-Baxter operator, satisfying some compatibility conditions.
		Yang-Baxter Hochschild cohomology was introduced by the authors to classify infinitesimal deformations of braided algebras, and determine obstructions to higher order deformations. Several examples of braided algebras satisfy a weaker version of commutativity, 
		which is called braided commutativity
		and involves the Yang-Baxter operator of the algebra. 
		We extend the theory of Yang-Baxter Hochschild cohomology to study braided commutative deformations of braided algebras. The resulting cohomology theory classifies infinitesimal deformations of braided algebras that are braided commutative, and provides obstructions for braided commutative higher order deformations. We consider braided commutativity for Hopf algebras in detail, and obtain some classes of nontrivial examples. 
\end{abstract}

\date{\empty}

\tableofcontents

\section{Introduction}

The Yang-Baxter equation (YBE) has been studied extensively in physics and knot theory.
In the set-theoretic case, the YBE is defined for an invertible map,  called an R-matrix or braiding, 
 $R: X \times X \rightarrow X \times X$ 
on a set $X$ by 
$(R \times \mathbb 1) ( \mathbb 1 \times R) (R \times \mathbb 1)= ( \mathbb 1 \times R) (R \times \mathbb 1) ( \mathbb 1 \times R)$,
and corresponds to the Reidemeister type III move in knot theory. 
A typical solution is defined by conjugation $R(x, y) = (y, y^{-1} xy)$
 in a group with multiplication $\mu(x,y)=xy$.
%

\begin{figure}[htb]
\begin{center}
\includegraphics[width=6in]{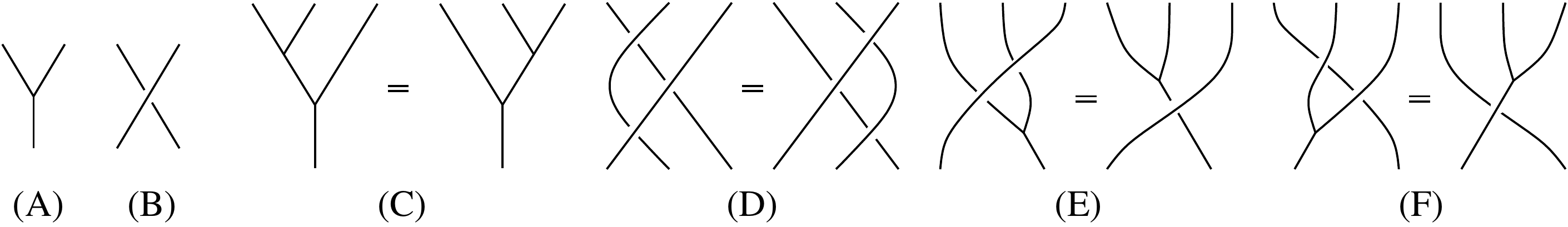}
\end{center}
\caption{
}
\label{BA}
\end{figure}

{ Figure~\ref{BA} shows a diagrammatic representation of a multiplication $\mu$ and a braiding $R$ in (A) and (B),
respectively. Diagrams are read top to bottom, so that in (A), the top two end points  receive a pair of elements $(x,y)$,
and the bottom end point receives $\mu(x,y)=xy$.
In (C) and (D), the associativity of the group multiplication $\mu$ and the YBE are depicted, respectively.
The multiplication and the braiding also satisfy the {\it compatibility} conditions represented in (E) and (F),
called the YI and IY relations.

The above example $R(x, y) = (y, y^{-1} xy)$ satisfies 
also the relation   $\mu R=\mu$  with group 
operation $\mu$, called {\it the braided commutativity}, as checked by comptation
$( \mu R ) (x, y) =\mu ( y, y^{-1} xy) = xy = \mu (x,y)$.
This relation is depicted in Figure~\ref{BCeq}. 

Figure~\ref{BA} (C)--(F) and Figure~\ref{BCeq} are some of the diagrammatic Reidemeister type moves for handlebody links \cite{Ishii08}, moves for planar  trivalent vertex graphs with crossings.
Discrete algebraic structures called multiple conjugation quandles (MCQ) that have both multiplication and braiding
have been used \cite{IIJO} to define invariants of handlebody links, and their homology theory was developed 
\cite{CIST} to define cocycle invariants \cite{CJKLS}.
See also \cite{Lebed}. For the move in Figure~\ref{BCeq},
a subcomplex of the homology theory corresponding to the braided commutativity 
was defined so that the corresponding cocycles define their invariants.

\begin{figure}[htb]
\begin{center}
\includegraphics[width=1in]{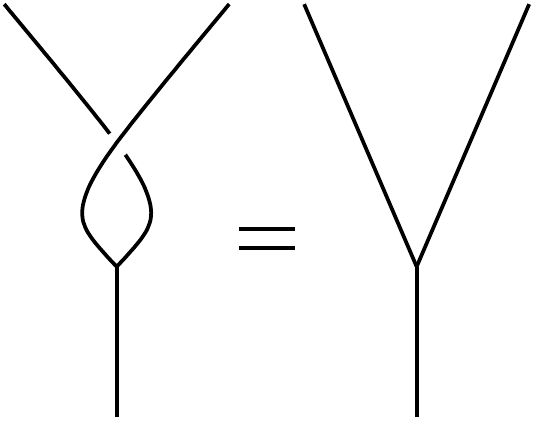}
\end{center}
\caption{
}
\label{BCeq}
\end{figure}

For an associative algebra, the YBE is defined on a tensor product, as well as the multiplication, 
and the same formula $\mu R=\mu$ defines braided commutativity.
In particular, if $X$ is a Hopf algebra, the map corresponding to conjugation
 ${\rm ad}: X \otimes X \rightarrow X$ is defined by
${\rm ad} (x \otimes y)=S(y^{(1)}) x y^{(2)}$ for simple tensors, where
$\Delta(y)=y^{(1)} \otimes y^{(2)}$ is Sweedler's notation for the comultiplication, 
and $S$ denotes the antipode.
The corresponding $R$-matrix (a solution to the YBE) is defined by
$R(x \otimes y)=y^{(1)} \otimes S(y^{(2)}) x y^{(3)}$, and it is checked that 
$R$ is indeed a solution to YBE (for example, \cite{CCES-adjoint}) and together with the multiplication, it satisfies the braided commutativity.

When an algebra has an $R$-matrix and satisfies certain compatibility relations called 
YI and IY conditions (next section), it is called a {\it braided algebra} \cite{Baez}. 
Cohomology theory for braided algebras was developed in \cite{SZ-YBH} from the
point of view of deformation theory.
In particular, the following properties were given: 
(1) the second cohomology group classifies equivalence classes of  infinitesimal deformations,
(2) a cochain map from Hopf algebra second cohomology  groups to braided algebra cohomology group was defined, and (3)  the third cohomology group is shown to be an obstruction to quadratic deformation.
It is the purpose of this paper to establish cohomology theory for braided commutativity of braided algebras and provide these three properties, including a stronger version of (3) for all deformation degrees.

Specifically, we define a cochain complex that encodes Yang-Baxter Hochschild cohomology and braided commutativity up to dimension 4 such that 
(1) the second cohomology group classifies equivalence classes of  infinitesimal deformations of 
braided commutative braided algebras,
(2) a cochain map from Hopf algebra second cohomology  groups to braided commutative algebra cohomology group is defined for 
the adjoint R-matrix, and (3)  the third cohomology group is shown to contain the obstruction to 
 higher order
deformation
of braided commutative braided algebras.

This article is organized as follows. In Section~\ref{sec:braidA_BC} we recall the definitions of braided algebra and braided commutativity, and we provide examples of braided algebras that satisfy the braided commutativity property. In Section~\ref{sec:YBH} we recall the definition of Yang-Baxter Hochschild (YBH) cohomology of a braided algebra. Section~\ref{sec:2_cocy_BC} is devoted to defining braided commutative (BC) second cohomology of braided algebras, and the classification of infinitesimal BC deformation by the second BC cohomology group. In Section~\ref{sec:Hopf} we consider Hopf algebras, and show that the second cohomology group of a Hopf algebra admits a homomorphism in the BC second cohomology group. Moreover, we show that cohomology classes that are not cobounded in Hochschild cohomology have nontrivial image. We use this fact to produce families of examples of nontrivial BC cohomology groups. The third BC cohomology group, and the study of obstructions of higher order deformations are found in Section~\ref{sec:3_cocy_BC}. Section~\ref{sec:complex} is devoted to the study of a multicomplex whose low-degree differentials coincide with BC cohomology.

\section{Braided algebras and braided commutativity}\label{sec:braidA_BC}

In \cite{SZ-YBH} deformation cohomology of braided algebras \cite{Baez} 
was studied. As constructions of braided algebras, two examples were presented:
(1) tensoring set-theoretic solutions of the YBE,  and (2) 
using the adjoint R-matrix in Hopf algebras \cite{CCES-adjoint}.

\subsection{Braided algebras}

In this section we define YI, IY, and braided algebras for which  Yang-Baxter Hochschild (co)homology
is defined, and give examples of such algebras.
These are algebras with both multiplication and braiding that satisfy certain compatibility conditions. Throughout this section, and the rest of the article, we assume that all algebras are over a unital ring $\mathbb k$,
and we use the symbol $\mathbb 1$ to indicate the identity map.
 The set of ${\mathbb k}$-algebra homomorphisms ${\rm Hom}_{\mathbb k}( - , -)$ is simply denoted by ${\rm Hom}( - , -)$. 
The review material of this section follows that of \cite{SZ-YBH}.

Recall that for an invertible map $R: V \otimes V \rightarrow V \otimes V$
of a ${\mathbb k}$-module $V$,
the equation
$$(R \otimes {\mathbb 1}) ({\mathbb 1} \otimes R) (R \otimes {\mathbb 1}) 
=({\mathbb 1} \otimes R) (R \otimes {\mathbb 1}) ({\mathbb 1} \otimes R) $$
 is called the {\it Yang-Baxter} equation (YBE),  and a solution $R$ of it is called a Yang-Baxter (YB) operator,
 or R-matrix.

\begin{definition}[\cite{SZ-YBH}]
{\rm
Let $(V, \mu)$ be a ${\mathbb k}$-algebra 
for a unital ring $\mathbb k$, with a YB operator $R: V \otimes V \rightarrow V \otimes V$.
We say that $\mu$ and $R$ satisfy the ${\rm YI}$  (resp. ${\rm IY}$) condition if they satisfy 
\begin{eqnarray}
(\mathbb 1\otimes \mu) ( R \otimes {\mathbb 1}) ({\mathbb 1} \otimes  R  ) &=& R \circ (\mu\otimes {\mathbb 1})\label{eqn:YI}, \\
(\mu\otimes \mathbb 1) ({\mathbb 1} \otimes R )(R \otimes {\mathbb 1}) &=& R\circ  (\mathbb 1\otimes \mu) ,\label{eqn:IY} 
\end{eqnarray}
respectively.
We call $(V, \mu, R)$ a YI algebra (resp. IY algebra) if Equation~\eqref{eqn:YI} (resp. Equation~\eqref{eqn:IY}) holds. If $(V, \mu, R)$ satisfies both Equation~\eqref{eqn:YI} and Equation~\eqref{eqn:IY}, we call it a braided algebra.  
}
\end{definition}

	\begin{definition}\label{def:braided_hom}
		{\rm 
			If $(V_1,\mu_1, R_1)$ and $(V_2,\mu_2, R_2)$ are braided algebras, we say that a linear map $f : V_1 \rightarrow V_2$ is a homomorphism (or morphism) of braided algebras if $f$ is an algebra homomorphism such that the following diagram commutes
			\begin{center}
				\begin{tikzcd}
					V_1\otimes V_1\arrow[rr,"f\otimes f"]\arrow[d,"R_1"] & & V_2\otimes V_2\arrow[d,"R_2"]\\
					V_1\otimes V_1\arrow[rr,"f\otimes f"]& &V_2\otimes V_2 
				\end{tikzcd} 
			\end{center}
						If $f$ admits an inverse that is itself a braided algebra homomorphism, we say that $f$ (and its inverse) is a braided isomorphism. 
		}
	\end{definition}

\begin{definition}[\cite{Baez}]
{\rm

A braided algebra $(V, \mu, R)$ is said to have, or satisfy, {\it braided commutativity} if 
$\mu R = \mu$ holds. 
}
\end{definition}

\begin{definition}
{\rm
Two braided algebras that satisfy braided commutativity are {\it equivalent} if they are isomorphic as braided algebras.
}
\end{definition}

For the rest of the section, we provide examples of braided algebras that satisfy braided commutativity.

\subsection{Tensorized braid group representations}

The following construction is found in \cite{SZ-YBH},
 and it  is a tensorized version of 
 multiple conjugation quandles (MCQs) \cite{CIST}.

First, recall that a \textit{quandle} 
is a non-empty set $X$ with a binary operation $*:X\times X\to X$ satisfying the following axioms.
\begin{itemize}
\item[(1)] For any $a\in X$, we have $a*a=a$.
\item[(2)] For any $a\in X$, the map $S_a:X\to X$ defined by $S_a(x)=x*a$ is a bijection.
\item[(3)] For any $a,b,c\in X$, we have $(a*b)*c=(a*c)*(b*c)$.
\end{itemize}
A {\it rack} is a set with an operation that satisfies (2) and (3).

\begin{definition}[\cite{Ishii08}]
A \textit{multiple conjugation quandle (MCQ)} $X$ is the disjoint union of groups $G_\lambda$,
where $\lambda$ is an element of an index set $\Lambda$, 
with a binary operation $*:X\times X\to X$ satisfying the following axioms.
\begin{itemize}
\item[(1)] For any $a,b\in G_\lambda$, we have $a*b=b^{-1}ab$.
\item[(2)] For any $x\in X$, $a,b\in G_\lambda$, we have $x*e_\lambda=x$ and $x*(ab)=(x*a)*b$, where $e_\lambda$ is the identity element of $G_\lambda$.
\item[(3)] For any $x,y,z\in X$, we have $(x*y)*z=(x*z)*(y*z)$.
\item[(4)] For any $x\in X$, $a,b\in G_\lambda$, 
we have $(ab)*x=(a*x)(b*x)$ in some group $G_\mu$.
\end{itemize}
\end{definition}

We call the group $G_\lambda$ a \textit{component} of the MCQ.
An MCQ is a type of quandle that can be decomposed as a union of groups, and the quandle operation in each component is given by conjugation.
Moreover, there are compatibilities, (2) and (4),  between the group and quandle operations.
In \cite{Ishii08}, concrete examples of MCQs are presented.

\begin{example}[\cite{SZ-YBH}] \label{ex:MCQ}
{\rm 
Let $X=\sqcup_{\lambda \in \Lambda} G_\lambda$ be an MCQ with a quandle operation $*$. 
Let ${\mathbb k}[X]$ be the free ${\mathbb k}$-module generated by $X$.

Define a multiplication $\mu : {\mathbb k}[X] \otimes {\mathbb k}[X] \rightarrow {\mathbb k}[X]$
on generators by $\mu (x\otimes y):= x y $ if $x, y \in G_\lambda$ for some $\lambda \in \Lambda$, 
and 
$\mu (x\otimes y) :=0$ otherwise, and extended linearly.
Define $R: {\mathbb k}[X]\otimes {\mathbb k}[X] \rightarrow {\mathbb k}[X] \otimes {\mathbb k}[X]$
on generators $x, y \in X$ by $R(x \otimes y):=  y \otimes (x*y)$ and extended linearly.
Then ${\mathbb k}[X]$ is a braided algebra,  as we now proceed to show.

If two of $x, y, z \in X$ belong to  distinct $G_\lambda$'s, then the triple product is $0$, 
and if all belong to the same $G_\lambda$ then it is associative, hence 
$\mu$ is associative.
The YI and IY conditions, respectively,  follow from the conditions 
 $(ab)*x=(a*x)(b*x)$ and $x*(ab)=(x*a)*b$ 
 for  any $x\in X$, $a,b\in G_\lambda$ on generators, and if $a, b $ belong to distinct $G_\lambda$ then
 the values are $0$ for both sides of the equations, hence ${\mathbb k}[X]$ is a braided  algebra.
}
\end{example}

\begin{lemma}
The preceding  example satisfies  braided commutativity.
\end{lemma}
\begin{proof}

The multiplication $\mu(x\otimes y)$ for $x, y \in X$ vanishes unless $x, y \in G_\lambda$, and 
in the case $x, y \in G_\lambda$, $\mu(x \otimes y)=xy$ is the group multiplication.
Since the braiding operation $(x, y) \mapsto (y, y^{-1} xy)$ satisfies braided commutativity with the group multiplication, the claim follows.
\end{proof}

\begin{remark}	
{\rm 
A natural analogue of braided commutativity is formulated as follows.
Let $X$ be a set with an associative monoid structure $\mu: X \times X \rightarrow X$ 
and a solution $R$  to the {\it set  theoretic Yang-Baxter equation} (sYBE) 
$$ (R \times {\mathbb 1}) ({\mathbb 1} \times R)  (R \times {\mathbb 1}) =
({\mathbb 1} \times R) (R \times {\mathbb 1}) ({\mathbb 1} \times R) .$$
Then $(X, \mu, R)$ is said to satisfy the braided commutativity if $\mu R=\mu$ holds.

These are examples in set-theoretic case $(X, \mu, R)$ coming from 
Wada's representations of braid groups in free group automorphism groups \cite{Wada}. 
These  representations are described by 
 $(x,y) \mapsto (u(x,y), v(x,y))$ for two letters $x,y$, where 
$u$ and $v$ are words in $x$ and $y$.
The representation of $n$-braids $B_n$ in ${\rm Aut}(F_n)$, the automorphism group
of the free group on $n$ letters, $\{ x_1, \ldots, x_n\}$, is given by
$$(x_1, \ldots, x_n) \mapsto (x_1, \ldots, x_{i-1} , u(x_i, x_{i+1}), v(x_i, x_{i+1}), x_{i+1}, \ldots, x_n)$$
for a standard braid generator $\sigma_i \in B_n$. The words $u$, $v$ have to satisfy certain equations for this to be well-defined braid group representation, and Wada provided a list of such words. In particular, these representations provide
solutions to the sYBE.
These solutions were further extended via Fox calculus in \cite{CSsYBE} and birack cocycles in  \cite{CSSES}. 
Wada's conjecture that  the list was complete was proved by Ito~\cite{Ito}. 
Here we provide those in Wada's list that satisfy the braided commutativity, as checked by simple calculations:
$$ (x,y) \mapsto (u(x,y), v(x,y)) =(y, y^{-1} xy), \ (y^{-1}, yxy), \ (xy^{-1} x^{-1}, xy^2).$$

Following the standard procedure of producing examples in tensors of modules
from set-theoretic case using Hopf algebras, we 
obtain the following list, 
as verified by Hopf algebra axioms:
$$R(x \otimes y)=y^{(1)} \otimes S( y^{(2)}) x y^{(3)}, \ 
S( y^{(1)} ) \otimes y^{(2)} x y^{(3)}, \ 
x^{(1)} S( y^{(1)} ) S( x^{(2)} ) \otimes x^{(3)} y^{(2)} y^{(3)} .$$
Although these Wada's braiding maps satisfy braided commutativity,  only the conjugation $(x, y) \mapsto (y, y^{-1}xy)$ satisfies
the discrete versions of the YI or IY  relations.
Therefore, although they provide examples of braided commutativity,
the deformation cohomology of this paper cannot be applied to the other two braidings, since they do not provide
braided algebras to which our theory applies.
}
\end{remark}


\subsection{Hopf algebra with R-matrix associated to adjoint map}\label{subsec:adR}

In this section we briefly recall the notion of Hopf algebras and establish  the diagrammatic interpretation of their defining axioms.

A {\it Hopf algebra} $(X, \mu,  \eta,  \Delta, \epsilon, S)$ (a 
module over a unital ring 
$\mathbb k$,
multiplication, unit, comultiplication, counit, antipode,  respectively), is
defined as follows. 
First, recall that a bialgebra 
$X$  
is a module endowed with an algebra structure with
multiplication $\mu: X\otimes X\rightarrow X$ and unit $\eta$,  and a coalgebra structure with comultiplication $\Delta: X\rightarrow X\otimes X$ and counit $\epsilon$, such that the compatibility condition 
$$\Delta \circ  \mu = (\mu\otimes \mu)\circ ( {\mathbb 1}\otimes  \tau\otimes  {\mathbb 1}) \circ (\Delta\otimes \Delta)$$
holds,
where $\tau$ denotes the transposition $\tau(x \otimes y) = y\otimes x$ for simple tensors.  A Hopf algebra is a bialgebra endowed with a map $S: X\rightarrow X$, called {\it antipode}, satisfying the equations 
$$\mu \circ (\mathbb 1\otimes S)\circ \Delta = \eta \circ \epsilon = \mu\circ (S\otimes \mathbb 1)\circ\Delta,$$
 called the {\it antipode condition}.

\begin{figure}[htb]
\begin{center}
\includegraphics[width=2in]{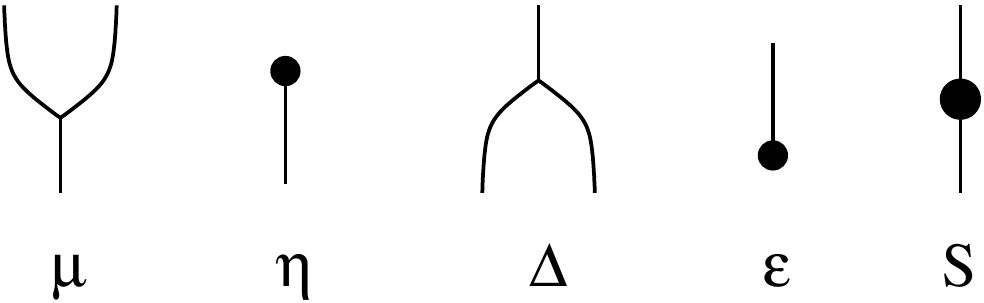}
\end{center}
\caption{ 
}
\label{hopfopB}
\end{figure}

In diagrams, an  array of 
 vertical $n$ edges represents $X^{\otimes n}$, and the diagrams are read from top to bottom, in the direction of  homomorphisms.
In Figure~\ref{hopfopB}, diagrammatic representations of multiplication $\mu$, unit $\eta$, comultiplication $\Delta$, counit $\epsilon$, antipode $S$ are depicted. 
In Figure~\ref{hopfaxiomB}, the conditions in the definition of a Hopf algebra are depicted;
(A) is the associativity, (B) is the unit condition, (C) is the compatibility between multiplication and comultiplication, 
and (D) is the antipode condition.

\begin{figure}[htb]
\begin{center}
\includegraphics[width=4.5in]{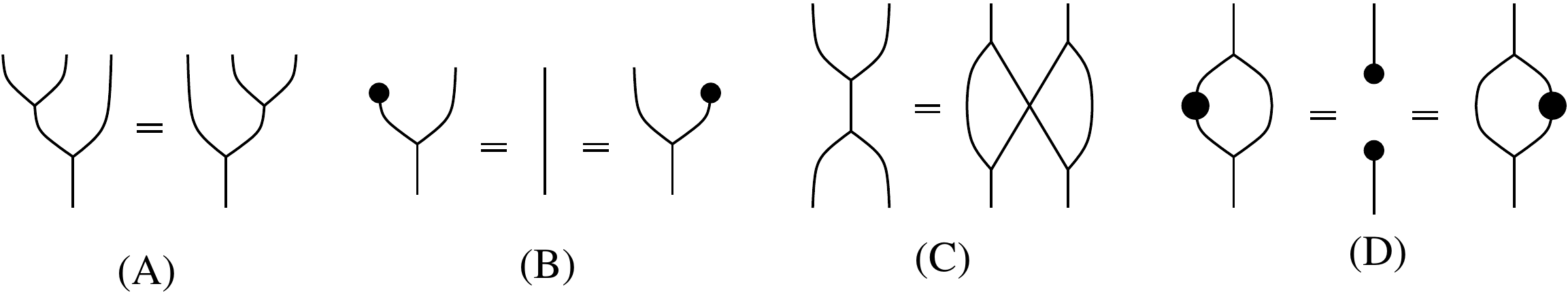}
\end{center}
\caption{  
}
\label{hopfaxiomB}
\end{figure}


For the comultiplication, we use Sweedler's notation $\Delta(x)=x^{(1)}\otimes x^{(2)}$ suppressing the summation symbol. Further, we use
$$( \Delta \otimes {\mathbb 1} ) \Delta (x) = ( x^{(11)} \otimes x^{(12)} ) \otimes x^{(2)}\quad {\rm and} \quad
( {\mathbb 1}  \otimes \Delta ) \Delta (x) =  x^{(1)} \otimes ( x^{(21)} \otimes x^{(22)} ) , $$
both of which are also written as 
$ x^{(1)} \otimes x^{(2)}  \otimes x^{(3)}$,  which is well defined due to coassociativity.

The example $R(x \otimes y)=y^{(1)} \otimes S( y^{(2)}) x y^{(3)}$ is called
 the {\it adjoint R-matrix}, induced by the {\it adjoint map}
 $x \otimes y \mapsto S( y^{(1)} ) x y^{(2)} $,
that is a Hopf algebra version of group conjugation.
A diagrammatic representation of the adjoint R-matrix is given in Figure~\ref{adR}. The cohomology of these R-matrices was studied in \cite{CCES-adjoint}.
 Further generalization of this construction to Hopf monoids in symmetric monoidal categories was studied in \cite{EZ}.

\begin{figure}[htb]
\begin{center}
\includegraphics[width=.5in]{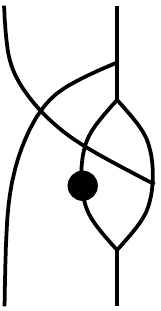}
\end{center}
\caption{ 
}
\label{adR}
\end{figure}

The following is found in \cite{SZ-YBH}.

\begin{lemma}[\cite{SZ-YBH}] \label{lem:BA}
	Let $(H,\mu,\Delta,\eta,\epsilon)$ be a Hopf algebra and let $R_H$ denote the  adjoint R-matrix defined above. Then $(H,R_H,\mu)$ is a braided algebra. 
\end{lemma}

\begin{lemma}
Let $R$ be the adjoint R-matrix defined on a Hopf algebra $X$ as  above.
Then $(X, \mu, R)$ satisfies braided commutativity.
\end{lemma}

\begin{proof}
The proof is represented by a sequence of diagrams in Figure~\ref{adRbc}.
Each equality from the left to the right follows from:
(1) definition, (2) (co)associativity, (3) commutation with transposition map, (4) antipode condition, and (5) (co)unit condition.
\end{proof}

\begin{figure}[htb]
\begin{center}
\includegraphics[width=3in]{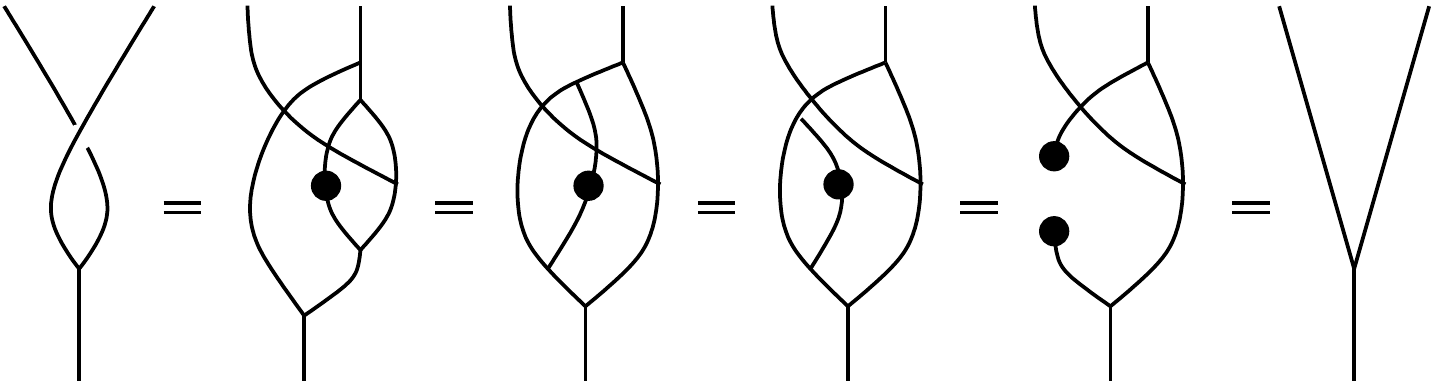}
\end{center}
\caption{ 
 }
\label{adRbc}
\end{figure}

In this paper the adjoint R-matrices are the main examples to which the deformation theory of braided commutativity is applied.

\section{Overview of Yang-Baxter Hochschild cohomology and deformation}\label{sec:YBH}

In this section, we first give a brief overview of Hochschild and  Yang-Baxter cohomologies, 
then  Yang-Baxter Hochschild cohomology that unifies them \cite{SZ-YBH}. 
Descriptions closely follow those in   \cite{SZ-YBH}.

\subsection{Hochschild cohomology and deformations}

In this section we briefly review deformation theoretic aspects of low dimensional 
Hochschild cohomology for a simplest case.
Let $(V, \mu)$ be an associative algebra with coefficient unital ring ${\mathbb k}$.
The cocahin groups of Hochschild cohomology are defined to be  
$C_{\rm H}^n(V,V)={\rm Hom}(V^{\otimes n},V)$ for $n\geq 1$.
Although Hochschild cohomology is defined for all dimensions, for defining the Yang-Baxter Hochschild cohomology, we focus on descriptions in low dimensions and establishing diagrammatic conventions.

The differentials are defined for $f \in C_{\rm H}^1(V,V)$ and $\psi \in C_{\rm H}^2(V,V)$ by 
\begin{eqnarray*}
\delta^1_{\rm H}(f)&=&
  \mu ( f \otimes {\mathbb 1}) + \mu  ( {\mathbb 1} \otimes  f ) 
  - f \mu , \\
\delta^2_{\rm H}(\psi)&=&  
\mu (\psi \otimes {\mathbb 1}) + \psi (\mu \otimes {\mathbb 1})
- \mu ( {\mathbb 1}\otimes  \psi ) - \psi ({\mathbb 1}\otimes   \mu ) .
\end{eqnarray*}
Though the signs of these maps are different from traditional definition, it does not affect forming a cochain complex, 
and this convention is convenient for diagrammatic computations, and is adopted in \cite{CCES-coalgebra}. 
Diagrammatic presentations of the differentials are depicted in Figures~\ref{Hochdiff1} and \ref{Hochdiff2}. 
Vertical $n$ edges represent $V^{\otimes n}$, and a  circle on an edge represent $f: V \rightarrow V$.
A trivalent vertex  represents a multiplication $\mu: V ^{\otimes 2} \rightarrow V$, 
and a circled trivalent vertex represent $\psi: V ^{\otimes 2} \rightarrow V$.

\begin{figure}[htb]
\begin{center}
\includegraphics[width=1.5in]{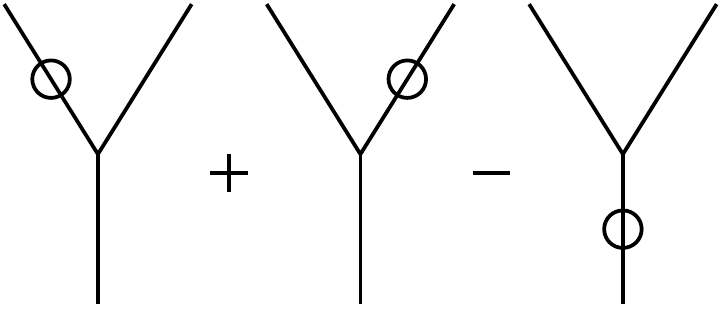}
\end{center}
\caption{
 }
\label{Hochdiff1}
\end{figure}

\begin{figure}[htb]
\begin{center}
\includegraphics[width=2.5in]{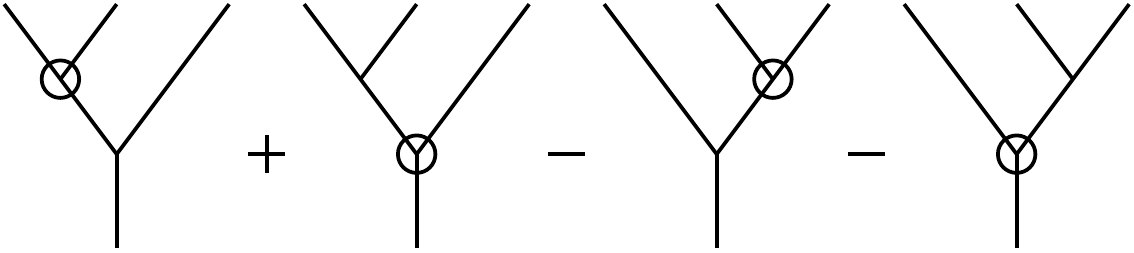}
\end{center}
\caption{ 
}
\label{Hochdiff2}
\end{figure}

	Let $(V,\mu)$ be an algebra over $\mathbb k$.  
We say that an  algebra $(V',\mu')$ over the power series $\mathbb k[[\hbar]]$
 is a deformation of $V$ if the quotient algebra $(V' / (\hbar) V' , \mu')$ coincides with 
 $(V, \mu)$. 

	Let $(V,\mu)$ be an associative algebra, 
	$\tilde V=(V \otimes {\mathbb k} [[\hbar]])/(\hbar^2 ) \cong V \oplus \hbar V$,
	and  $\psi \in Z^2_{\rm H}(V,V)$. 
 Then, setting $\tilde \mu = \mu + \hbar \psi$,
  $(\tilde V,\tilde \mu)$ is an algebra  if and only if the equations
	$\delta^2_{\rm H}( \psi)=0$ hold. 
 This  classical result \cite{Gerst} is seen by computing the associativity  in $\tilde V$,
\begin{eqnarray*}
\tilde \mu (\tilde \mu \otimes {\mathbb 1} )& =&  
( \mu + \hbar \psi ) ( (\mu  + \hbar \psi  ) \otimes {\mathbb 1} ) \\
&=&
 \mu ( \mu \otimes {\mathbb 1} ) + \hbar [ \mu ( \psi   \otimes {\mathbb 1} )+
 \psi ( \mu  \otimes {\mathbb 1} ) ] , \\ 
\tilde \mu ({\mathbb 1}  \otimes \tilde \mu ) &=& 
 ( \mu + \hbar \psi ) ( {\mathbb 1}  \otimes  (\mu  + \hbar \psi  ) \\
 &=&
 \mu ({\mathbb 1}  \otimes  \mu ) + \hbar[ 
 \mu (  {\mathbb 1}   \otimes  \psi)+
  \psi (  {\mathbb 1}  \otimes \mu ) ] . 
  \end{eqnarray*}

In this case, we also have $\tilde \mu \equiv_{(\hbar)} \mu$, meaning $\tilde \mu$ coincides with $\mu$ modulo $(\hbar)$, and we say that $(\tilde V, \tilde \mu)$ is an infinitesimal, or first, deformation   
of $(V, \mu)$. 
 Thus we say  that the primary obstruction to the first deformation vanishes if and only if $\delta^2_{\rm H}(\psi)=0$.

\subsection{Yang-Baxter cohomology and deformations}

In this section we review  Yang-Baxter (deformation) cohomology.
Let $(V, R)$ be a ${\mathbb k}$-module with the R-matrix $R : V^{\otimes 2} \rightarrow V^{\otimes 2}$.
The cochain groups are defined by 
$C^0_{\rm YB}(V,V)=0$ and 
$C^n_{\rm YB}(V,V)={\rm Hom} (V^{\otimes n},V^{\otimes n})$ for $n>0$.
We define the differentials for $f \in C^1_{\rm YB}(V,V) $
and $\phi \in C^2_{\rm YB}(V,V)$ by 
  \begin{eqnarray*}
 \delta^1_{\rm YB} (f) 
 &=&
 R ( f \otimes {\mathbb 1}) + R ( {\mathbb 1} \otimes  f ) 
 -  ( f \otimes {\mathbb 1}) R - ( {\mathbb 1} \otimes  f )  R ,\\
\delta^2_{\rm YB} (\phi) &=&
(R \otimes {\mathbb 1} ) ( {\mathbb 1}  \otimes R ) ( \phi \otimes  {\mathbb 1} )
+ (R \otimes {\mathbb 1} ) ( {\mathbb 1}  \otimes \phi ) ( R \otimes  {\mathbb 1} )
+  (\phi \otimes {\mathbb 1} ) ( {\mathbb 1}  \otimes R ) ( R \otimes  {\mathbb 1} ) \\
& & -  ( {\mathbb 1}  \otimes R ) ( R \otimes  {\mathbb 1} ) ( {\mathbb 1}  \otimes \phi ) 
- ( {\mathbb 1}  \otimes R ) ( \phi \otimes  {\mathbb 1} ) ( {\mathbb 1}  \otimes  R ) 
- ( {\mathbb 1}  \otimes \phi ) ( R \otimes  {\mathbb 1} ) ( {\mathbb 1}  \otimes  R ) .
\end{eqnarray*}
The differentials are depicted in Figures~\ref{YBdiff1} and \ref{YBdiff2}.
The cochains $f$ and $\phi$ are represented by circles on an edge and a crossing, 
respectively. 

\begin{figure}[htb]
\begin{center}
\includegraphics[width=2.2in]{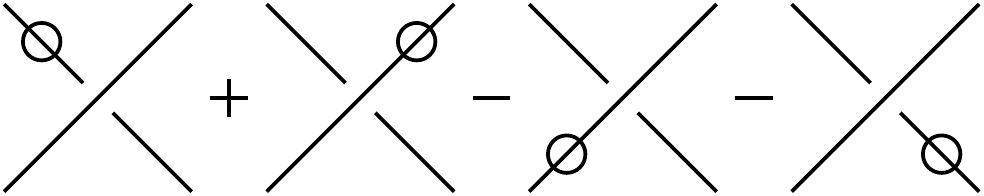}
\end{center}
\caption{ }
\label{YBdiff1}
\end{figure}

\begin{figure}[htb]
\begin{center}
\includegraphics[width=3.5in]{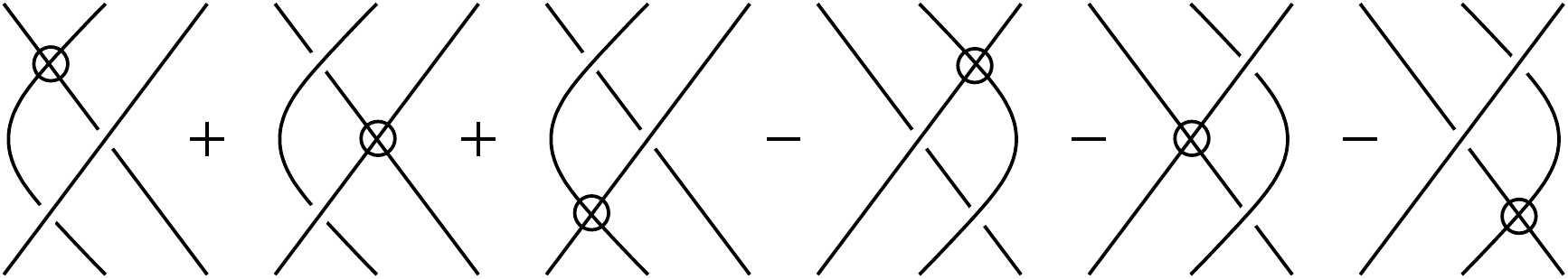}
\end{center}
\caption{}
\label{YBdiff2}
\end{figure}

\begin{lemma}[\cite{SZ-YBH}] \label{lem:R}
The sequence 
$C^0_{\rm YB}(V,V) \rightarrow 
C^1_{\rm YB}(V,V) \stackrel{\delta^1_{\rm YB}}{\longrightarrow}
C^2_{\rm YB}(V,V) \stackrel{\delta^2_{\rm YB}}{\longrightarrow}
C^2_{\rm YB}(V,V)
$
defines a cochain complex.
\end{lemma}

A variant of deformation cohomology was defined in all dimensions in \cite{Eisermann1}.

\subsection{Yang-Baxter Hochschild cohomology}

In this section we review Yang-Baxter Hochschild cohomology up to dimension 3 defined in \cite{SZ-YBH}.
Let $V$ be a braided algebra with coefficient unital ring ${\mathbb k}$.
The cochain groups for a braided algebra $V$ with coefficients in itself up to degree $3$ are defined as follows. 
 We set $C_{\rm YBH}^0(V,V) =0$,
and  
$C^{n,k}_{\rm YBH}(V,V) = {\rm Hom}(V^{\otimes n},V^{\otimes k}) $
for $n , k >0$.
We also use different subscripts 
$$C^{n,k}_{\rm YI}(V,V) =  {\rm Hom}(V^{\otimes n},V^{\otimes k})
= C^{n,k}_{\rm IY}(V,V)$$
to distinguish different isomorphic direct summands. 
Define  
\begin{eqnarray*}
C_{\rm YBH}^1(V,V) &=& C^{1,1}_{\rm YBH}(V,V) ={\rm Hom}(V,V) , \\
  C^{2}_{\rm YBH}(V,V) &=& C^{2,2}_{\rm YBH}(V,V) \oplus C^{2,1}_{\rm YBH}(V,V), \quad {\rm and} \\
C^{3}_{\rm YBH}(V,V)  &=& C^{3,3}_{\rm YBH}(V,V)  \oplus
 C^{3,2}_{\rm YI}(V,V)  \oplus C^{3,2}_{\rm IY}(V,V) 
 \oplus  
 C^{3,1}_{\rm YBH}(V,V) . 
 \end{eqnarray*}

 The differentials are defined as follows.
 We set $\delta^1_{\rm YBH}$ to be the direct sum $\delta^1_{\rm YB}\oplus \delta^1_{\rm H}$.

   The second differential $\delta^2_{\rm YBH}$ is defined to be the direct sum of four terms
as 
   	 \begin{eqnarray*}
   	 		\delta^2_{\rm YBH} = \delta^2_{\rm YB}\oplus \delta^2_{\rm YI}\oplus \delta^2_{\rm IY}\oplus \delta^2_{\rm H} 
   	 \end{eqnarray*}  
  where $\delta^2_{\rm YB}$ and $\delta^2_{\rm H}$ map in the first ($C^{3,3}_{\rm YBH}(V,V) ={\rm Hom}(V^{\otimes 3},V^{\otimes 3})$) 
  and last $(C^{3,1}_{\rm YBH}(V,V) = {\rm Hom}(V^{\otimes 3},V)$)  direct summands of $C^3_{\rm YBH}(V,V)$, respectively,
  while $\delta^2_{\rm YI}$ and $\delta^2_{\rm IY}$ map to the middle  
two factors $C^{3,2}_{\rm YI}(V,V)$ and $C^{3,2}_{\rm IY}(V,V)$, respectively.
  Each differential is defined as follows.
  \begin{eqnarray*}
\delta^2_{\rm YI} (\phi\oplus \psi) &=&
 (\mathbb 1\otimes \psi)(R\otimes \mathbb 1)(\mathbb 1 \otimes R) + (\mathbb 1 \otimes \mu)(\phi\otimes \mathbb 1)(\mathbb 1\otimes R) \\
 & &
 + (\mathbb 1 \otimes \mu)(R\otimes \mathbb 1 )(\mathbb 1 \otimes \phi) 
-  R (\psi\otimes \mathbb 1 ) - \phi (\mu\otimes \mathbb 1). \\
\delta^2_{\rm IY}(\phi\oplus \psi) &=& 
(\psi\otimes \mathbb 1)(\mathbb 1\otimes R)(R\otimes \mathbb 1) + (\mu\otimes \mathbb 1)(\mathbb 1\otimes \phi)(R\otimes \mathbb 1) \\
& &  
+ (\mu\otimes \mathbb 1)(\mathbb 1 \otimes R)(\phi\otimes \mathbb 1)
- R(\mathbb 1 \otimes \psi)  - \phi (\mathbb 1\otimes \mu) .
\end{eqnarray*}
The differential $\delta^2_{\rm IY}(\phi\oplus \psi)$ is represented diagrammatically in Figure~\ref{IY2cocy}, where 2-cochains $\phi \in C^{2,2}_{\rm YBH}(V,V)$ and $\psi\in C^{3,1}_{\rm YBH}(V,V)$ are represented by 4-valent (resp. 3-valent) vertices 
with circles. The figure is from \cite{SZ-YBH}.  Similar diagrammatics hold for the differential $\delta^2_{\rm YI}(\phi\oplus \psi)$ as well.
 The YI and IY components of the cochain complex above are included to enforce the coherence axioms between deformed algebra structure, and deformed YB operator. In other words, they ensure that YBH $2$-cocycles satisfy Equation~\eqref{eqn:YI} and Equation~\eqref{eqn:IY}, respectively.

\begin{figure}[htb]
\begin{center}
\includegraphics[width=3in]{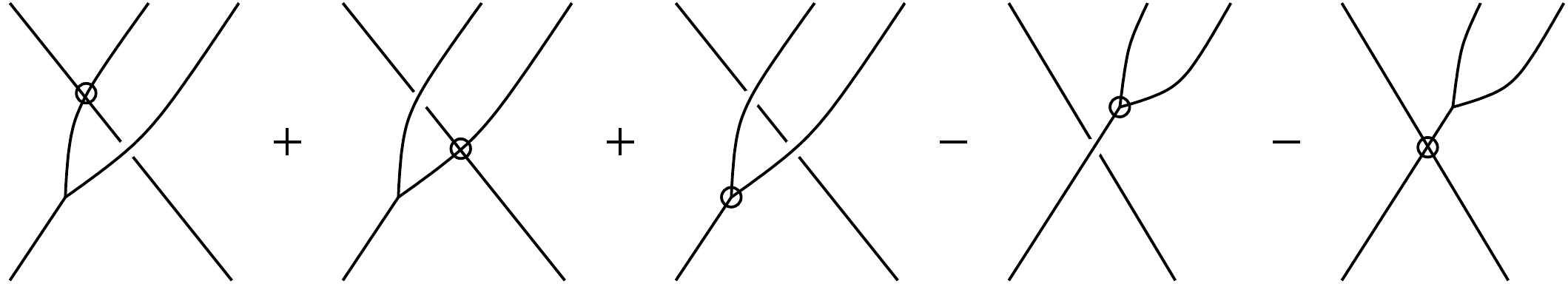}
\end{center}
\caption{}
\label{IY2cocy}
\end{figure}

\begin{proposition}[\cite{SZ-YBH}] \label{pro:cochain}
The sequence
$$ 
C_{\rm YBH}^1(V,V) 
\stackrel{\delta^1_{\rm YBH}}{\longrightarrow}
C_{\rm YBH}^2(V,V) 
\stackrel{\delta^2_{\rm YBH}}{\longrightarrow}
 C_{\rm YBH}^3(V,V) 
 $$
 defines a cochain complex.
\end{proposition}

In \cite{SZ-YBH}, a bijective relationship between the second Yang-Baxter Hochschild cohomology group 
and the equivalence classes of infinitesimal deformations was established as follows.

	Let $(V,\mu,R)$ be a braided algebra over $\mathbb k$.  Let us consider the power series ring $\mathbb k[[\hbar]]$ over the formal variable $\hbar$, and let $(\hbar )$ denote the 
	ideal of $\mathbb k[[\hbar]]$ generated by $\hbar$.

\begin{definition}[\cite{SZ-YBH}] \label{def:inf_deform} 
	{\rm 

	Let $(V, \mu, R)$ be a braided algebra over ${\mathbb k}$. Extend $\mu$ and $R$ 
	to $\tilde V=( V \otimes_{\mathbb k} {\mathbb k}[[\hbar]] ) / (\hbar^2 ) \cong V \oplus \hbar V$ by linearly extending on $\hbar V$  and use the same notation $\mu$ and $R$ on $\tilde V$. 	We say that $(\tilde V, \tilde \mu , \tilde R)$ is an {\it infinitesimal deformation} of $(V, \mu, R)$ 
	if ${\rm Im }(\tilde \mu - \mu) \subset \hbar  V$ and 
	${\rm Im }(\tilde R - R ) \subset \hbar  ( V \otimes  V)$.
Similar definitions hold for YI algebras and IY algebras. 

	Let $(V, \mu, R)$ be a braided algebra, and let $(\tilde V, \tilde \mu, \tilde R)$ and $(\hat V, \hat \mu, \hat R)$
	be two infinitesimal deformations of $V$. Let $\tilde f: \tilde V \rightarrow \hat V$ be a homomorphism of braided
	algebras (Definition~\ref{def:braided_hom}). Then, we say that $f$ is a {\it homomorphism of infinitesimal deformations} if $\tilde f |_{V}  : V \rightarrow V$
	is the identity map $\mathbb 1_V$. A homomorphism of infinitesimal deformations that is invertible through a
	homomorphism of infinitesimal deformations is said to be an {\it isomorphism of infinitesimal deformations}. 

}
\end{definition}

If $(\tilde V, \tilde \mu , \tilde R)$ is an infinitesimal deformation of $(V, \mu, R)$, then we can write
\begin{eqnarray*}
			\tilde \mu &=& \mu + \hbar \psi\\
			\tilde R  &=& R + \hbar \phi,
		\end{eqnarray*}
		where $\psi : V\otimes V\rightarrow V$ and $\phi: V\otimes V \rightarrow V\otimes V$.

\begin{theorem}[\cite{SZ-YBH}] \label{thm:classification}
	Let $(V,\mu,R)$ be a braided algebra. Then the  Yang-Baxter Hochschild second cohomology group classifies the infinitesimal deformations of $(V,\mu,R)$. 
\end{theorem}

\section{Deformation 2-cocycles for braided commutativity}\label{sec:2_cocy_BC}

In this section we define cohomology theory for braided commutativity up to dimension 2.
Let $(V, \mu, R)$ be a braided algebra.

\begin{definition}
	{\rm 
				Let $(V, \mu, R)$ be a braided commutative algebra, and let $(\hat V, \hat \mu, \hat R)$ be a braided algebra infinitesimal deformation of $V$, in the sense of Definition~\ref{def:inf_deform}. Then, we say that $\hat V$ is a braided commutative infinitesimal deformation, or BC infinitesimal deformation for short, if  $(\hat V, \hat \mu, \hat R)$ is braided commutative. 
				
				If $\hat V_1$ and $\hat V_2$ are two BC infinitesimal deformations, we say that a braided algebra homomorphism is a braided commutative (BC) homomorphism. Two BC infinitesimal deformations are said to be isomorphic, if there exists a braided algebra isomorphism between them.
	}
\end{definition}

\begin{definition}
	{\rm  		We let $D(V, \mu, R)$ indicate the abelian group of equivalence classes of of infinitesimal braided commutative deformations of the braided algebra $(V, \mu, R)$.
		The abelian structure is defined as follows. 
		For two BC deformations $(\hat V_i, \hat \mu_i, \hat R_i)$, $i=1,2$,  of  $(V, \mu, R)$,
		where $\hat \mu_i = \mu + \hbar \psi_i$ and $\hat R_i = R + \hbar \phi_i$, the addition is defined as 
		the deformation of $V$ by 
		$\mu +  \hbar (\psi_1 + \psi_2)$ and $R + \hbar (\phi_1 + \phi_2)$. A direct computation shows that this is well defined.
	}
\end{definition}

\begin{definition}
{\rm
The cochain groups of braided  commutative cohomology are defined as follows:
\begin{eqnarray*}
C^1_{\rm BC}(X,X) &=& {\rm Hom}(X,X),  \\
C^2_{\rm BC}(X,X) & =&
C^2_{\rm YB }(X,X)
\oplus  C^2_{\rm  H}(X,X) , \\
C^3_{\rm BC}(X,X) &=&
C^3_{\rm YB}(X,X)  \oplus 
C^{3,2}_{\rm IY}(X,X)  \oplus  
C^{3,2}_{\rm YI}(X,X)    \oplus
C^3_{\rm H}(X,X) 
 \oplus  
C^{2,1}_{\rm BC} (X,X),
 \end{eqnarray*}
where we set $C^{n,m}_{\rm BC} (X,X) = C^{n,m}_{\rm YBH} (X,X)$, and introduce the suffix BC to explicitly keep track of the direct summands that refer to the braided commutative constraint, as it will be clear in the proofs below.
 We note that $C^3_{\rm YBH} (X,X)=C^3_{\rm YB}(X,X)\oplus C^3_{\rm YI}(X,X) 
\oplus  C^3_{\rm IY}(X,X)\oplus C^3_{\rm H}(X,X) $. 
The differentials are defined as follows:
\begin{eqnarray*}
\delta^1_{\rm BC} (f) &=& \delta^1_{\rm YB} (f) + \delta^1_{\rm Hf} (f)  \ =\ \delta^1_{\rm YBH} 
(f), \\
\delta^2_{\rm BC} (\phi \oplus  \psi ) &=&  \delta^2_{\rm YBH}( \phi \oplus  \psi )+  (\mu  \phi +  R \psi -  \psi ).
\end{eqnarray*}

}
\end{definition}

\begin{figure}[htb]
\begin{center}
\includegraphics[width=1.5in]{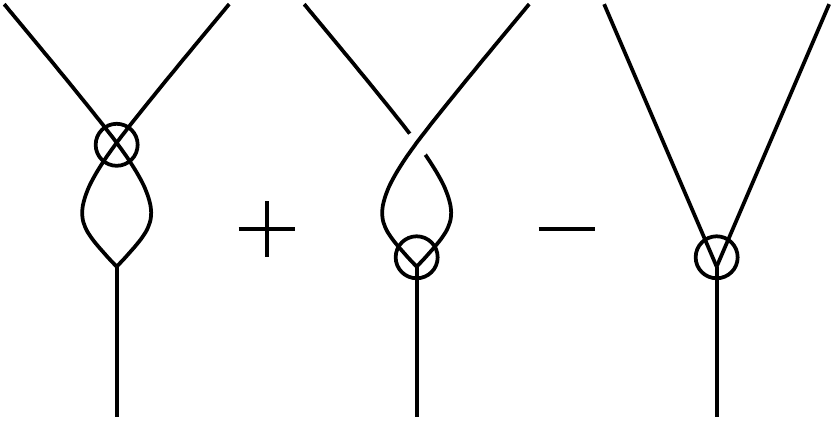}
\end{center}
\caption{}
\label{BCdiff2}
\end{figure}
 
 The map  $\mu  \phi +  R \psi -  \psi $ in $\delta^2_{\rm BC}$   is depicted in Figure~\ref{BCdiff2}. 
The crossing with a circle represents $\phi \in C^2_{\rm YB}(V,V)$, and trivalent vertices with circles 
represent $\psi \in C^2_{\rm H}(V,V)$.

\begin{figure}[htb]
\begin{center}
\includegraphics[width=3in]{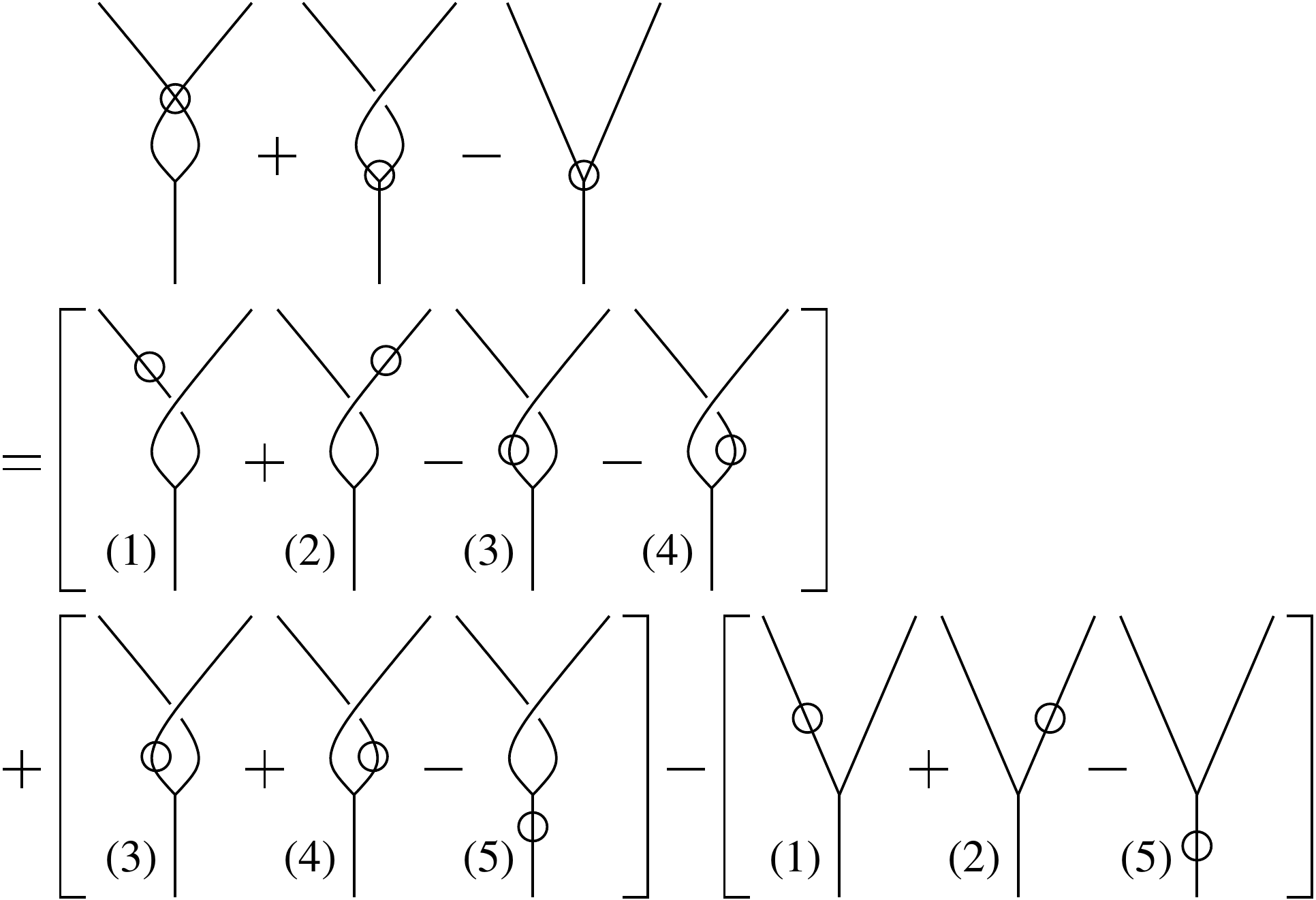}
\end{center}
\caption{}
\label{BCdiff1diff2}
\end{figure}

\begin{lemma}\label{lem:sub}
The following sequence forms a chain complex:
$$
C^1_{\rm BC} (V,V) \stackrel{\delta^1_{\rm BC}}{\xrightarrow{\hspace*{1cm}}} 
C^2_{\rm BC} (V,V) \stackrel{\delta^2_{\rm BC}}{\xrightarrow{\hspace*{1cm}}} 
C^3_{\rm BC} (V,V) .
$$
\end{lemma}

\begin{proof}
Set $\phi=\delta^1_{\rm YB} (f)$ and $\psi =\delta^1_{\rm H} (f)$. 
Diagrammatically the map $\mu \phi + \psi R - \psi$  is depicted in Figure~\ref{BCdiff2} in the top line with circle vertices representing $\phi$ and $\psi$.
The substitution  $\phi=\delta^1_{\rm YB} (f)$ and $\psi =\delta^1_{\rm H} (f)$ are depicted in the bottom two lines,
with circles on edges representing $f$. The cancelling terms are indicated by numbers in the figure, showing that 
the equation is satisfied with the substitution.
This proves $\delta^2_{\rm BC} \delta^1_{\rm BC}=0$.
\end{proof}

\begin{definition}
{\rm
The second braided commutative cohomology group $H^2_{\rm BC}(V,V)$ is defined 
as usual by ${\rm Ker}(\delta^2_{\rm BC}) / {\rm Im} ( \delta^1_{\rm BC} )$, which is well posed following Lemma~\ref{lem:sub}.
}
\end{definition}

\begin{lemma}\label{lem:braid_comm_deform}
Let $(V,\mu, R)$ be a braided algebra. 
Let $(\phi, \psi)$ be a  Yang-Baxter Hochschild $2$-cochain. Then, setting $\tilde \mu = \mu + \hbar \psi$, $\tilde R = R + \hbar \phi $, the pair $(\tilde \mu, \tilde R)$ defines a braided commutative infinitesimal deformation if and only if 
$(\phi, \psi) \in Z_{\rm BC}^2(V,V)$, i.e. it is a braided $2$-cocycle. 
\end{lemma}
\begin{proof}
		The pair $(\tilde \mu, \tilde R)$ is a braided commutative infinitesimal deformation if it defines a braided algebra infinitesmial deformation of the braided algebra $V$, and this deformation is braided commutative. Using Lemma~3.7 in \cite{SZ-YBH} $(\tilde \mu, \tilde R)$ defines a braided algebra deformation if and only if $(\phi, \psi)$ is a YBH $2$-cocycle. Since the differentials for YBH and BC differ only by a direct summand, it means that to prove the statement we only need to show that a braided algebra infinitesimal deformation defined by a YBH $2$-cocycle $(\phi, \psi)$ is a braided commutative if and only if $\mu\phi + R\psi - \psi = 0$. 
		
		The braided algebra infinitesimal deformation defined by $(\phi, \psi)$ is braided commutative, by definition, if and only if
		$$
				\hat \mu\hat R = \hat\mu.
		$$
		Taking the component of degree $1$ in $\hbar$ of the previous equation we find that this is equivalent to
		$$
				\psi R + \mu \phi = \psi.
		$$
		Since terms of higher order in $\hbar$ vanish for infinitesimal deformations, and the degree $0$ part of the braided commutativity equation is automatically satisfied, this completes the proof. 
\end{proof}

\begin{theorem}
Let $(V, \mu, R)$ be a braided commutative algebra. 
There is a bijection between the equivalence classes $D(V, \mu, R)$
of infinitesimal  braided commutative deformations and the second braided commutativity cohomology, $H^2_{\rm BC}(V,V)$.
\end{theorem}
\begin{proof}
Using Lemma~\ref{lem:braid_comm_deform}, and noting that the abelian structure defined on $D(V, \mu, R)$ is compatible with the correspondence of the lemma,
there is a homomorphism from $Z^2_{\rm BC}(V,V)$ to the abelian group of deformations.
We show that this descends to an isomomorphism $H^2_{\rm BC} (V,V) \rightarrow D(V, \mu, R)$.
Specifically, we show the following.

		We first show that BC $2$-cocycles define BC deformations, and that a cobounded $2$-cocycle defines a trivial braided commutative infinitesimal deformation. Then, we prove the converse of this statement, showing that BC deformations arise through BC $2$-cocycles, and trivial braided commutative infinitesimal deformations are cobounded. 
		
		From Lemma~\ref{lem:braid_comm_deform}, given a $2$-cocycle $(\phi, \psi)$ we obtain a BC deformation by setting $\hat \mu = \mu + \hbar \psi$ and $\hat R = R + \hbar \phi$. This correspondence defines a map from $2$-cocycles to infinitesimal braided commutative deformations. This is independent of the choice of representative in the class of $H^2_{\rm BC}(V,V)$. In fact, suppose that $(\phi_1, \psi_1)$ and $(\phi_2, \psi_2)$ are in the same equivalence class. Since $\delta^1_{\rm BC} = \delta^1_{\rm YBH}$, it follows that $(\phi_1, \psi_2) =  (\phi_2, \psi_2) + \delta^1_{\rm YBH}(f)$ for some $f$. Applying Theorem~3.8  in \cite{SZ-YBH},
		which states that
		 the Yang-Baxter Hochschild second cohomology group classifies the infinitesimal deformations of 
		 $(V, \mu, R)$,
		 it follows that $\hat \mu_i = \mu + \hbar \psi_i$ and $\hat R_i = R + \hbar \phi_i$, for $i=1,2$, define equivalent braided algebra deformations. Since they are both braided commutative, they are equivalent BC infinitesimal deformations. It follows that to show well-definedness of the correspondence, we can limit ourselves to showing that a cobounded $2$-cocycle defines a trivial braided commutative infinitesimal deformation, as stated above. 
		
		Suppose now that we have a $2$-cocycle $(\phi,\psi)$ such that $(\phi,\psi) = \delta^1_{\rm BC}(f)$, for some BC $1$-cochain $f$. The corresponding braided algebra infinitesimal deformation is trivial (as a braided algebra), by virtue of Theorem~3.8 in \cite{SZ-YBH}, considering the fact that $\delta^2_{\rm BC}$ contains a direct summand that coincides with $\delta^2_{\rm YBH}$. Therefore, there exists a braided algebra isomorphism $\tilde f$ between the deformation induced by $(\phi,\psi)$ and the undeformed algebra $V$, which reduces to the identity on $V$. This means that the braided commutative infinitesimal deformation is trivial.
		
		Before showing the converse, consider two  equivalent infinitesimal deformations $\hat \mu_i = \mu + \hbar \psi_i$ and $\hat R_i = R + \hbar \phi_i$, for $i=1,2$, where $\hat \mu_2$ and $\hat R_2$ define a braided commutative structure. Then, applying Theorem~3.8  in \cite{SZ-YBH}, it follows that $\hat \mu_1 = \hat \mu_2 + \hbar \delta^1_H(f)$ and $\hat R_1 = \hat R_2 + \hbar \delta^1_{\rm YB}(f)$ for some $f: V\rightarrow V$. Lemma~\ref{lem:sub} implies that $(\delta^1_{\rm YB}(f), \delta^1_{H}(f))$ is braided commutative, and therefore the deformation $(\hat \mu_1, \hat R_1)$ is   a sum of  
		braided commutative deformations, and therefore braided commutative. This means that in the proof of the converse of the previous part, we do not need to consider the case where we have an infinitesimal deformation that is equivalent to a braided commutative deformation, but we can directly consider only the case where we have a braided commutative deformation, as stated at the beginning of the proof. 
		
		 For the converse of the statement, suppose that $\hat V$ is a braided commutative infinitesimal deformation of $V$. From Theorem~3.8 in \cite{SZ-YBH} we know that $\hat V$  arises from a YBH $2$-cocycle $(\phi,\psi)$, and moreover there exists a YBH $1$-cochain such that $\delta^1_{\rm YBH}(f) = (\phi,\psi) $. Since the coboundary of the BC component of the BC cohomology is given by the direct sum of YB and Hochschild cohomology, it follows that $(\phi,\psi)$ is cobounded as braided commutative $2$-cocycle as well, completing the proof.
\end{proof}

\section{Braided commutativity for Hopf algebras}\label{sec:Hopf}

In this section we construct  braided commutative 2-cocycles in Hopf algebras from Hopf algebra cocycles.
First we review the construction of  a homomorphism between the second cohomology group of Hopf algebras and the second YBH cohomology of their corresponding braided algebras, following expositions of \cite{SZ-YBH}.  
Recall (\cite{ChariPressley}, Chapter 6) that the Hopf algebra second cohomology group 
of a Hopf algebra $H$ with coefficients in $H$ characterizes Hopf algebra deformations. 
Roughly speaking, a $2$-cocycle for a Hopf algebra consists of a pair of cochains $(\xi, \zeta )$ such that $\xi$ deforms the algebra structure, $\zeta$ deforms the coalgebra structure, and such deformations are compatible. 

We adopt the symbols
$(\phi, \psi)$ for YBH 2-cocycles, and  
use
$(\xi, \zeta)$ for Hopf algebra 2-cocycles,
where 
$\xi \in {\rm Hom}(H^{\otimes 2}, H)$ is a deformation 2-cocycle of a Hopf algebra multiplication $\mu$, 
and $ \zeta \in {\rm Hom}(H, H^{\otimes 2})$ is a deformation 2-cocycle of a comultiplication $\Delta$.
Both multiplications for braided algebras and Hopf algebras share the same symbol $\mu$,
as it causes little confusion.

More specifically, Hopf algebra 2-cochain groups with  coefficients in $H$
are defined as 
	$$C^2_{\rm Hf}(H,H) = C^{2,1}_{\rm Hf}(H,H) \oplus C^{2,2}_{\rm Hf}(H,H),$$ 
where 
$C^{2,1}_{\rm Hf}(H,H)={\rm Hom}(H^{\otimes 2}, H)$
and $C^{2,2}_{\rm Hf}(H,H)={\rm Hom}(H,H^{\otimes 2})$
.}

If $\mu$ and $\Delta$ denote the multiplication and comultiplication of the Hopf algebra $H$,
the 2-cocycles $\xi \in Z^{2,1}_{\rm Hf}(H,H)$ and $\zeta \in Z^{2,2}_{\rm Hf}(H,H)$ deforme
$\mu$ and $\Delta$, respectively, so that 
$\tilde \mu=\mu + \hbar \xi$ and $\tilde \Delta = \Delta + \hbar \zeta$ defines 
a Hopf algebra structure on $H \otimes {\mathbb k}[[\hbar]]/(\hbar^2)$. 
This condition requires  that $(\xi, \zeta)$ satisfy the compatibility condition 
needed to guarantee that $\tilde \mu$ and $\tilde \Delta$ are compatible,
i.e., $\tilde \Delta (x  y) = \tilde \Delta (x)\cdot \tilde \Delta (y)$ for simple tensors $x$ and $y$,
where the concatenation represents $\tilde \mu$.

A Hopf $2$-cocycle $(\xi,\zeta)$ is said to be {\it normalized}
if it satisfies the {\it normalization conditions}
\begin{eqnarray*}
		\xi(1\otimes x) &=& \xi(x\otimes 1) = 0,\\
		(\epsilon \otimes \mathbb 1)\zeta(x) &=& (\mathbb 1\otimes \epsilon)\zeta(x) = 0.
\end{eqnarray*} 
The module of normalized 2-cocycles is denoted by $\hat Z^2_{\rm Hf}(H,H)$,
and its second homology group is denoted by 
$\hat H^2_{\rm Hf}(H,H) := \hat Z^2_{\rm Hf}(H,H) / ( B^2_{\rm Hf}(H,H) \cap \hat Z^2_{\rm Hf}(H,H) )$.
These normalization conditions correspond to the unit and counit conditions for the deformed bialgebra. 
It is known (see \cite{ChariPressley}, Chapter 6) that any $2$-cocycle is cohomologous to a $2$-cocycle that satisfies 
 the normalization  conditions. 
 Thus given a deformation of a bialgebra, there exists an equivalent deformation where the unit and counit are not deformed. 
 
 In addition, any bialgebra deformation can  be extended to a Hopf algebra deformation, in the sense that there is 
 a deformation $\hat S=S + \hbar S'$ of  the antipode $S$  in such a way that the conditions $\hat \mu(\mathbb 1\otimes \hat S)\hat \Delta = \hat \mu(\hat S\otimes \mathbb 1)\hat \Delta = \eta \epsilon$ are satisfied. 
 The deformed antipode $\hat S$ is unique, because antipodes of Hopf algebras are unique.

For ease of notation, we will denote $\zeta(x)$ in Sweedler notation as for $\Delta$, the only difference being that we will use lower scripts for $\zeta$, while upper scripts for $\Delta$. Therefore, we will have $\zeta(x) = x_{(1)}\otimes x_{(2)}$. Since $(\Delta\otimes \mathbb 1)\zeta$ and $(\zeta\otimes \mathbb 1)\Delta$ are in general not the same, we introduce further a notation to distinguish the consecutive application of $\Delta$ and $\zeta$. We set 
$$(\mathbb 1 \otimes \Delta )\zeta(x) = ( \mathbb 1 \otimes \Delta)(x_{(1)}\otimes x_{(2)}) = (x_{(1)})^{(1)}\otimes (x_{(1)})^{(2)} \otimes x_{(2)} =: \prescript{}{(1)}x\otimes \prescript{}{(2)}x \otimes \prescript{}{(3)}x, $$
 and similarly we set 
 $(\mathbb 1 \otimes \zeta )\Delta(x) =: \prescript{(1)}{}x\otimes \prescript{(2)}{}x\otimes \prescript{(3)}{}x$.
This notation allows us to forget the brackets that determine whether the internal index is lower or upper -- i.e. whether we have applied $\Delta$ or $\zeta$ first.


In \cite{SZ-YBH}, a map from Hopf algebra chain complex to the Yang-Baxter Hochschild chain complex was defined as follows.
Let $H$ denote a Hopf algebra and let $R_H$ be the YB operator defined through 
the adjoint map 
as in Section~\ref{subsec:adR}. Let $(\xi, \zeta)$
denote a Hopf algebra $2$-cocycle as described above.
Then, a cochain 
$\Psi_{\zeta} (\xi) : H\otimes H \rightarrow H\otimes H$ was defined as 
\begin{eqnarray*}
	\Psi_\zeta(\xi)(x\otimes y) &=& y^{(1)}\otimes S(y^{(2)})\xi(x\otimes y^{(3)}) + y^{(1)}\otimes \xi(S(y^{(2)})\otimes xy^{(3)})\\
	&&+\prescript{}{(1)}{y} \otimes S(\prescript{}{(2)}{y})x\prescript{}{(3)}{y} + \prescript{(1)}{}{y} \otimes S(\prescript{(2)}{}{y})x\prescript{(3)}{}{y} + y^{(1)} \otimes S'(y^{(2)})xy^{(3)}. 
\end{eqnarray*}
Here we note that $S'$ is the degree 1 term in $\tilde S=S + \hbar S'$, which is uniquely determined by $(\xi, \zeta)$, so that $\Psi$ does not depend on $S'$.

It was shown in \cite{SZ-YBH} that the map $(\xi, \zeta) \mapsto (\Psi_\zeta(\xi), \xi)$
defines a map 
$\Psi: \hat Z^2_{\rm Hf}(H,H) \rightarrow Z^2_{\rm YBH}(H,H)$.

\begin{lemma}\label{lem:2cocy_Hopf}
Let the map $\Psi: \hat Z^2_{\rm Hf}(H,H) \rightarrow Z^2_{\rm YBH}(H,H)$ be defined as above,
and let $(\phi, \psi)=(\Psi_\zeta (\xi), \xi)$ be the image of $(\xi, \zeta)$.
Then $(\phi,  \psi) \in Z^2_{\rm BC}(H,H)$. 
\end{lemma}

\begin{proof}
As a consequence of Theorem~5.1 in \cite{SZ-YBH}, we only need to show the result for the braided commutative component of the differential.

Diagrammatic computations are shown in Figure~\ref{BCHopf}.
The first line depicts the first differential $\delta^1_{\rm Hf}  (\phi, \psi)$.
The second line is obtained by substitution $(\phi, \psi) = (\Psi_{\zeta}  (\xi), \xi) $.
In particular, the terms (1) through (5) represent the five terms in the definition of $\Psi_\zeta (\xi)$. 
In (5), the shaded circle represents $S'$. 

The terms $(3)$ and $(7)$ cancel. The 2-cocycle conditions are applied to 
the pairs $(1)$, $(2)$, and $(4)$, $(6)$, respectively, to obtain the pairs 
$(8)$, $(9)$, and $(10)$, $(11)$.
After applying the identity  $x^{(1)} S(x^{(2)})=\eta \epsilon$, the terms $(8)$ and $(11)$ become 
$(12)$ and $(15)$, respectively, and they vanish due to normality of $\xi$ and $\zeta$. 
The terms $(9)$ and $(10)$ are equal to $(13)$ and $(14)$, respectively, by 
associativity and coassociativity, and together with (5) they  vanish, since they follow from the condition $\hat \mu(\mathbb 1\otimes \hat S)\hat \Delta  = \eta \epsilon$.
\end{proof}

\begin{figure}[htb]
\begin{center}
\includegraphics[width=4in]{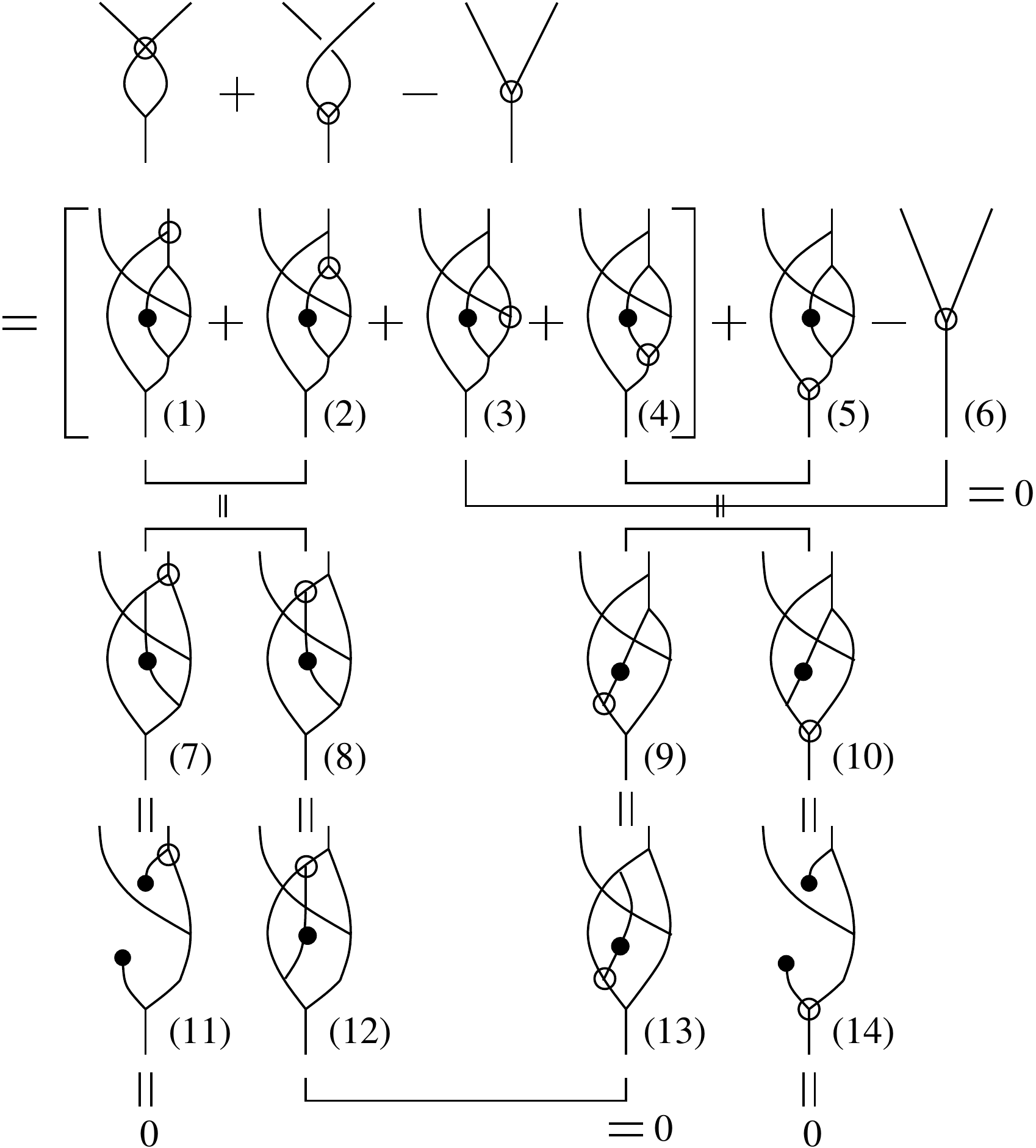}
\end{center}
\caption{}
\label{BCHopf}
\end{figure}

\begin{theorem}\label{thm:Hopf}
		 Let $H$ be a Hopf algebra. Then, the map $\Psi$ induces a well defined homomorphism of second cohomology groups $ \Psi: \hat H^2_{\rm Hf}(H,H) \rightarrow H^2_{\rm BC}(H,H)$, such that $[\Psi(\xi,\zeta)] \neq 0$ whenever $[\xi] \neq 0$ in $H^2_{\rm H}(H,H)$. 
		    
\end{theorem}
\begin{proof}
		Applying Lemma~\ref{lem:2cocy_Hopf}, we only need to prove that $\Psi$ descends to quotients. 	 
		 In \cite{SZ-YBH} it was proved that 
		 $\Psi$ induces a homomorphism $ \Psi: \hat H^2_{\rm Hf}(H,H) \rightarrow H^2_{\rm YBH}(H,H)$. 
		 Lemma~\ref{lem:2cocy_Hopf} implies that the image of $\Psi$ is in $H^2_{\rm BC}(H,H)$, and since the first differentials of BC and YBH cohomologies are the same, the map is well defined also in $H^2_{\rm BC}(H,H)$.
		 
				 	For the second statement, we need to show that if $(\Psi_\zeta (\xi), \xi)$  is cobounded in BC cohomology, then $\xi$ is cobounded in Hochschild cohomology, and therefore $(\xi, \zeta)$ is a deformation which is trivial as an algebra deformation. However, this is a direct consequence of the definitions. In fact, if $(\Psi_\zeta (\xi), \xi)$  is cobounded in BC cohomology, then in particular it is cobounded in YBH cohomology, which in turn implies that $\xi$ is cobounded in Hochschild cohomology by the results in \cite{SZ-YBH}, against the initial assumption of nontriviality. This completes the proof.  
		\end{proof}

\begin{example}
	{\rm 
			Theorem~\ref{thm:Hopf}  allows us to find examples of Hopf algebras whose BC second cohomology groups are nontrivial. For instance, there are examples of group algebras over fields of nonzero characteristic whose second Hopf cohomology group is nontrivial, therefore inducing a nontrivial second BC cohomology group via  $\Psi$, cf. Example~5.6 in \cite{SZ-YBH}. A similar reasoning can be applied to the Hopf algebra of regular functions over an algebraic group, cf. Example~5.7 in \cite{SZ-YBH}. 
	}
\end{example}

\section{Deformation $3$-cocycles for braided commutativity}\label{sec:3_cocy_BC}

In \cite{SZ-YBH}, it was proved for a braided algebra $V$ that 
if $H_{\rm YBH}^3(V,V) = 0$, then any infinitesimal deformation can be extended to a quadratic deformation.
It is our goal of this section to prove, in addition, that if $V$ is braided commutative and
$H_{\rm BC}^3(V,V) = 0$, then any degree $n$ braided commutative deformation can be extended to a degree $n+1$ braided commutative deformation. 
Toward this goal, we first show that the obstruction to extending braided commutative infinitesimal deformation to 
a braided commutative quadratic deformation lies in the third BC cohomology. We then generalize this result to arbitrary degrees. 

\begin{definition}\label{def:d3}
{\rm
The forth cochain group is defined by 
$
C^4_{\rm BC} (X,X)= C^4_{\rm YBH}\oplus{\rm Hom} (X^{\otimes 3}, X^{\otimes 2}) 
$.
The third differential 
$\delta^3_{\rm BC:YI} : C^3_{\rm BC}(X,X) \rightarrow C^4_{\rm BC}(X,X)$ is defined 
for $\alpha \in C^{3,2}_{\rm YI}(X,X)$, $\beta \in C^3_{\rm YB}(X,X)$, 
and $\alpha' \in {C^{3,2}_{\rm IY}(X,X) }$, $\gamma\in C^3_{\rm H}(X,X)$, and $\tau\in C^{2,1}_{\rm BC}(X,X)$
by 
 \begin{eqnarray*}
\delta^3_{\rm BC:YI} ( \alpha \oplus \beta \oplus \gamma ) 
 &=& 
[  
\alpha  (R \otimes \mathbb 1 ) +
 R ( \gamma \otimes \mathbb 1 ) ]  
- 
[({\mathbb 1} \otimes \mu ) \beta +
 ({\mathbb 1} \otimes  \gamma ) (R \otimes {\mathbb 1} )({\mathbb 1} \otimes  R )+ 
 \alpha ], \\ 
 \delta^3_{\rm BC:IY} ( \alpha \oplus \beta \oplus \gamma ) 
 &=&  [(\mu\otimes \mathbb 1)\beta + \alpha(\mathbb 1\otimes R) + R(\mathbb 1\otimes \gamma)] - [(\gamma\otimes\mathbb 1)(\mathbb 1\otimes R)(R\otimes \mathbb 1) + \alpha],\\
 \delta^3_{\rm BC} ( \alpha \oplus \alpha' \oplus \beta \oplus \gamma \oplus \tau) 
 &=& \delta^3_{\rm YBH} (\alpha\oplus \alpha' \oplus \beta \oplus \gamma) \oplus
 \delta^3_{\rm BC:YI} ( \alpha \oplus \beta \oplus \tau ) 
 \oplus
 \delta^3_{\rm BC:IY} ( \alpha' \oplus \beta \oplus \tau ). 
 \end{eqnarray*} 
}
\end{definition}

\begin{figure}[htb]
\begin{center}
\includegraphics[width=3in]{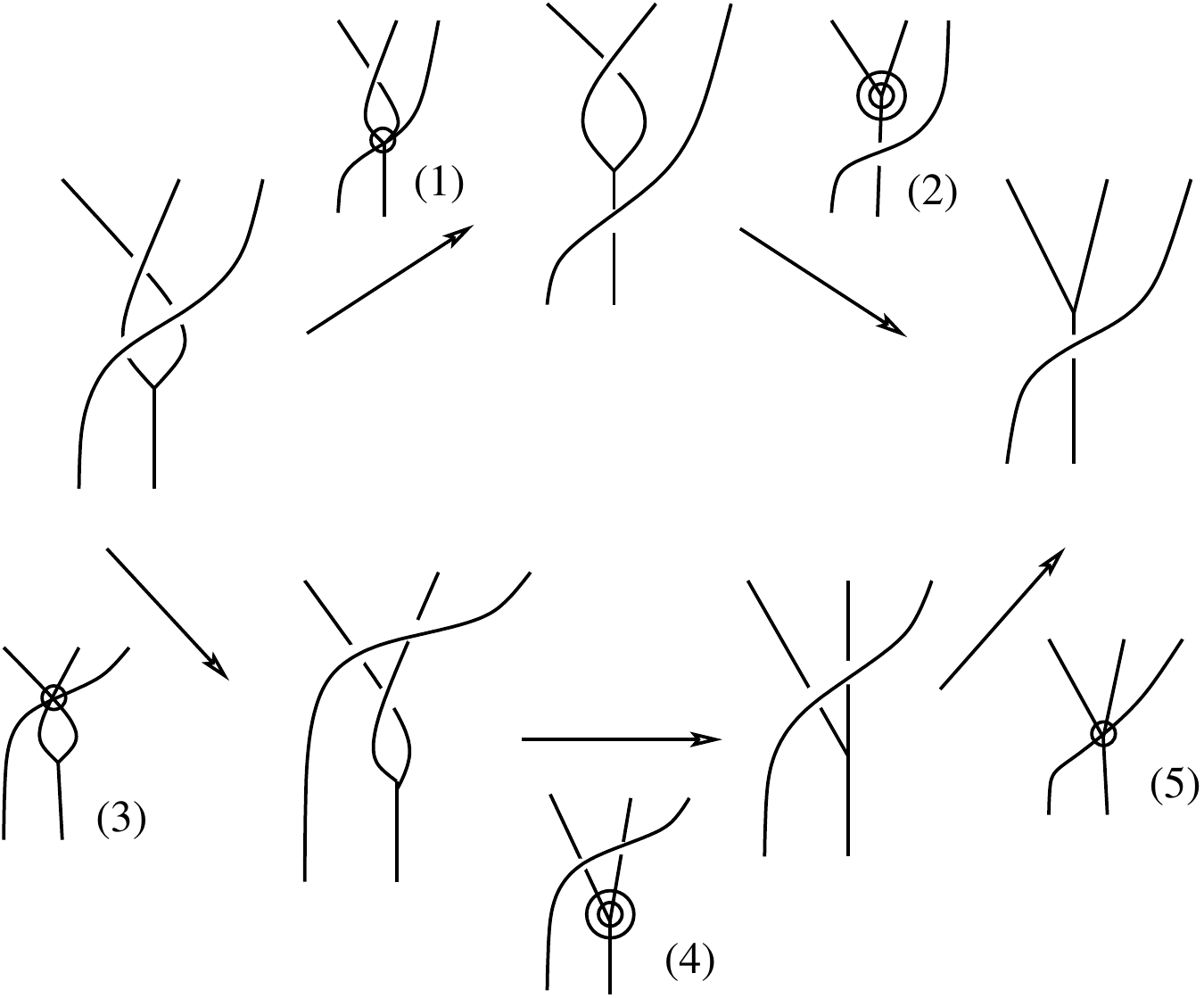}
\end{center}
\caption{}
\label{BCdiff3}
\end{figure}

The YI component of the third differential is diagrammatically represented in Figure~\ref{BCdiff3}. 
 Similar diagrammatic interpretation can be given to the IY component of $\delta^3_{\rm BC}$.
There are two sequences of graph moves from the left-most diagram to the right-most,
using the moves from Figure~\ref{BA} (D), (E) and Figure~\ref{BCeq}.
The first move labeled (1) at the upper left, for example, is the 
 YI move depicted in  Figure~\ref{BA} (E), 
and this move is represented by the circled 5-valent vertex (3 top, 2 bottom edges) 
This 5-valent vertex represents the cocycle $\alpha$, 
and the crossing above it with a string at right represents $(R \otimes \mathbb 1 ) $, 
 so that the graph at (1) indicates the first term  
$\alpha  (R \otimes \mathbb 1 ) $
 in the differential
$\delta^3_{\rm BC:YI} ( \alpha \oplus \beta \oplus \tau )$ of Definition~\ref{def:d3}. 
Similarly, the graph diagrams at (2) through (5) represent the second and the fifth terms of $\delta^3_{\rm BC:YI} $.
The last three terms  (3), (4) and (5) correspond to the lower arrows in Figure~\ref{BCdiff3} and receive
the negative signs.

We extend the braided commutative complex up to degree 4 as follows.

\begin{lemma}\label{lem:deg4}
The  following sequence forms a cochain complex: 
$$
C^1_{\rm BC} (V,V) \stackrel{\delta^1_{\rm BC}}{\xrightarrow{\hspace*{1cm}}} 
C^2_{\rm BC} (V,V) \stackrel{\delta^2_{\rm BC}}{\xrightarrow{\hspace*{1cm}}} 
C^3_{\rm BC} (V,V) \stackrel{\delta^3_{\rm BC}}{\xrightarrow{\hspace*{1cm}}} 
C^4_{\rm BC} (V,V) .
$$
\end{lemma}

\begin{proof}
In  Figure~\ref{BCdiff2diff3}, the five terms of $\delta^3_{\rm BC:IY}$ as depicted in Figure~\ref{BCdiff3} (1) -- (5)
are listed in the left-hand side.
In the right-hand side, substitutions of $\delta^2_{\rm BC}(\phi \oplus \psi)$ are depicted.
For (1), the left-hand side represents 
$\alpha  (R \otimes \mathbb 1 ) $ 
 as above, and 
the substitution of $\alpha = \delta^2_{\rm BC}(\phi \oplus \psi)$ is depicted in the right-hand side of the top line.
Since $\alpha$ is a 
 Yang-Baxter Hochschild 3-cocycle, we have $\alpha = \delta^2_{\rm YBH}(\psi)$. 
The terms of $\delta^2_{\rm YBH}(\psi)$ are substituted in 
the 5-valent vertex to obtain the diagrams of the right-hand side of (1). 
Similarly, the right-hand side of (3) with negative sign corresponds to the negative 
of the Yang-Baxter 2-cocycle condition depicted in Figure~\ref{YBdiff2}.
Thus the total of the right-hand side of (1) through (5) represents $\delta^3_{\rm BC:YI} \delta^2_{\rm BC}(\phi \oplus \psi)$, and  all labeled terms cancel in pairs with opposite signs, showing that 
$\delta^3_{\rm BC:YI} \delta^2_{\rm BC}(\phi \oplus \psi)=0$. 
Similar computations show that $\delta^3_{\rm BC:YI} \delta^2_{\rm BC}(\phi \oplus \psi)=0$, therefore implying that $\delta^3_{\rm BC} \delta^2_{\rm BC}(\phi \oplus \psi)=0$.
\end{proof}

\begin{figure}[htb]
\begin{center}
\includegraphics[width=3.7in]{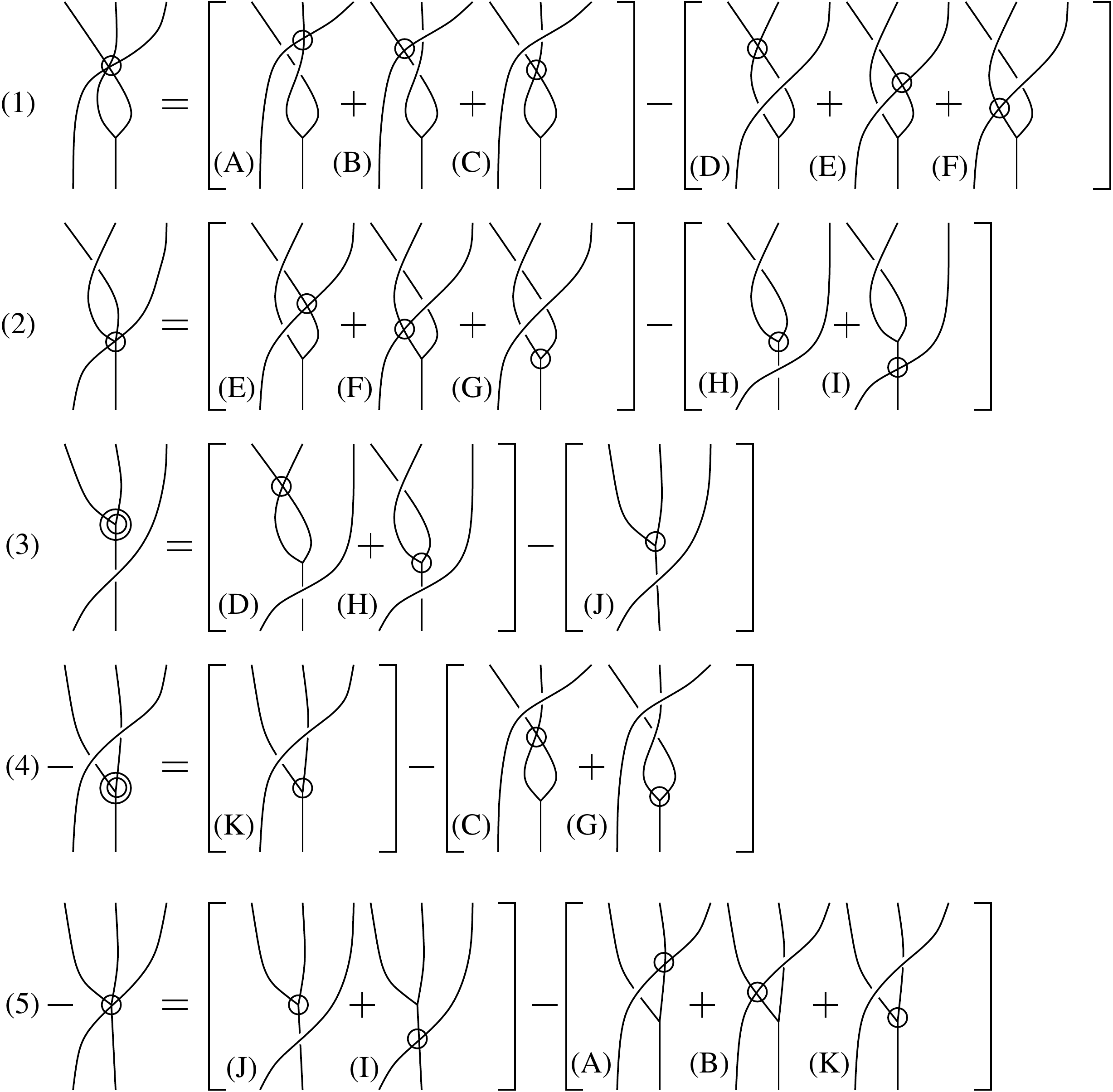}
\end{center}
\caption{}
\label{BCdiff2diff3}
\end{figure}

For the rest of the section, we show that the third cohomology group  contains the obstruction to secondary deformation. First we review such results for Yang-Baxter Hochschild cohomology.

	Let $(V,\mu)$ be an algebra over ${\mathbb k}$, and let $\psi=\psi_1$ denote a Hochschild $2$-cocycle,
	$\psi_1 \in Z^2_{\rm H}(V,V)$. 
As before we set  $(\tilde V, \tilde \mu)$, where $\tilde V= V \otimes  ({\mathbb k}[[\hbar]])/(\hbar^2)$
and $\tilde \mu = \mu + \hbar \psi_1$.
Next we consider $(\hat V, \hat \mu)$ where 
$\hat V= (V \otimes {\mathbb k}[[\hbar]])/(\hbar^3)$ and $\hat  \mu = \mu + \hbar \psi_1 + \hbar^2 \psi_2$.
	One computes the degree two term of $\hbar$ to be 
\begin{eqnarray*}
& & \hbar^2 [ \{  \mu (\psi_2 \otimes \mathbb{1} ) +  \psi_2 (\mu \otimes \mathbb{1} ) 
-  \mu (\mathbb{1}  \otimes \psi_2) -  \psi_2 (\mathbb{1}  \otimes \mu) ) \} 
+ \{ \psi_1 (\psi_1 \otimes \mathbb{1} ) - \psi_1 (\mathbb{1}  \otimes \psi_1) ) \} ]
  \\
& & = 
\hbar^2 [ \delta^1 (\psi_2) + \{  \psi_1 (\psi_1 \otimes \mathbb{1} ) - \psi_1 (\mathbb{1}  \otimes \psi_1) \} ] .
\end{eqnarray*}
Set $\Theta_2 (\psi_1) = \psi_1 (\psi_1 \otimes \mathbb{1} ) - \psi_1 (\mathbb{1}  \otimes \psi_1)$.
Then this computation  shows the following well known lemma.

\begin{lemma} \label{lem:Hdeform2}
 The above defined $(\hat V, \hat \mu)$ is an algebra such that 
 $\hat \mu \equiv_{(\hbar^2)} \mu$
 if and only if 	$\delta^2 (\psi_2) +\Theta_2 (\psi_1) =0$. 
\end{lemma}

 In this case we say that $(\hat V, \hat \mu)$ is a 
 degree 2 deformation of $(V, \mu)$.
The following is known. 

\begin{lemma}[\cite{Gerst}] \label{lem:H3cocy}
$\delta^3_{\rm H}\Theta_2 (\psi_1) =0$. 
\end{lemma}
 
By this lemma,  $\Theta_2 (\psi_1) $ is a Hochschild 3-cocycle, representing an element 
$[\Theta_2 (\psi_1) ] \in H^3_{\rm H}(V,V)$.
If  $\delta^2 (\psi_2) +\Theta_2 (\psi_1) =0$ as in Lemma~\ref{lem:Hdeform2},
then $\Theta_2 (\psi_1) $ is a coboundary, $[\Theta_2 (\psi_1) ] =0$, and the degree 1 deformation $(\tilde V, \tilde \mu)$
of $(V, \mu)$ further deforms to $(\hat V, \hat \mu)$.
In particular, if $H^3(V,V)=0$, then $(V, \mu)$  deforms to $(\hat V, \hat \mu)$.
Thus $H^3(V,V)$ can be regarded as obstruction to secondary deformation.

In \cite{SZ-YBH},
the following computation is made for the degree $n$ terms. 
Let $(V,R,\mu)$ be a braided algebra. Then we will use the notation $R_n = \sum_{i=0}^n \phi_i\hbar^i $ and $\mu_n = \sum_{i=0}^n \psi_i\hbar^i$,
where $\phi_0 := R$ and $\psi_0 := \mu$. For all $r\geq 2$ we also set
$$
\Gamma_r := \{(i,j,k)\ |\ i+j+k = r,\ i,j,k\neq r \}.
$$
Then, we define $\Xi^{IY}_r := \sum_{(i,j,k)\in \Gamma_r} (\psi_i\otimes \mathbb 1)(\mathbb 1\otimes \phi_j)(\phi_k\otimes 1)$ and $\Omega^{IY}_r := \sum_{(p,q,0)\in \Gamma_r} \phi_p(\mathbb 1\otimes \psi_q)$.  Similarly we also set $\Xi^{YI}_r := \sum_{(i,j,k)\in \Gamma_r} (\mathbb 1\otimes \psi_i)(\phi_j\otimes \mathbb 1)(\mathbb 1\otimes \phi_k)$ and $\Omega^{YI}_r := \sum_{(p,q,0)\in \Gamma_r} \phi_p(\psi_q\otimes \mathbb 1)$. Lastly, we define $\Theta_r = \sum_{(i,j,k)\in \Gamma_r} (\phi_i\otimes \mathbb 1)(\mathbb 1\otimes \phi_j)(\phi_k\otimes \mathbb 1) - (\mathbb 1\otimes \phi_i)(\phi_j\otimes \mathbb 1)(\mathbb 1 \otimes \phi_k)$ and $\Lambda_r = \sum_{(p,q,0)\in \Gamma_r} \mu_p(\mu_q\otimes \mathbb 1) - \mu_p(\mathbb 1\otimes \mu_q)$. 

The following  lemma was proved by calculations  in \cite{SZ-YBH}.

\begin{lemma}[\cite{SZ-YBH}] \label{lem:higher}
	Let $(V,R,\mu)$ be a braided algebra, and let $(\tilde V, R_n, \mu_n)$ be a deformation of degree $n$. Then $(V, R_{n+1}, \mu_{n+1})$ is a deformation of $(V,R,\mu)$ of order $n+1$ if and only if 
	\begin{center}
		$
	\begin{cases}
		\delta^2_{\rm YB}\phi_{n+1} + \Theta_{n+1} = 0\\
		\delta^2_{\rm IY}(\phi_{n+1} \oplus \psi_{n+1}) + \Xi^{IY}_{n+1} - \Omega^{IY}_{n+1} = 0\\
		\delta^2_{\rm YI}(\phi_{n+1} \oplus \psi_{n+1}) + \Xi^{YI}_{n+1} - \Omega^{YI}_{n+1} = 0\\
		\delta^2_{\rm H}\psi_{n+1} + \Lambda_{n+1} =0.
	\end{cases}
	$
	\end{center}
\end{lemma}

In a manner similar to Lemma~\ref{lem:Hdeform2} of the 
Hochschild case, the following is proved for Yang-Baxter Hochschild case.

\begin{lemma}[\cite{SZ-YBH}]  \label{lem:3cocy}
The obstruction to extending an infinitesimal deformation of braided algebras to a quadratic deformation lies in the third YBH cohomology group. More concisely, we have 
$$\delta_{\rm YBH}^3 (\Theta_2 ) =
\delta_{\rm YBH}^3 (\Xi^{\rm YI}_2 - \Omega^{\rm YI}_2 )=
		\delta_{\rm YBH}^3 ( \Xi^{\rm IY}_2 - \Omega^{\rm IY}_2 ) =
		\delta_{\rm YBH}^3 ( \Lambda_2 )=0. $$
\end{lemma}

Using these lemmas, the following was obtained.

\begin{corollary}[\cite{SZ-YBH}] \label{cor:quadratic}
	If $H_{\rm YBH}^3(V,V) = 0$, then any infinitesimal deformation can be extended to a quadratic deformation.
\end{corollary}

We now consider the effect of the braided commutativity constraint as an obstruction to higher order deformations. We define the term 
\begin{eqnarray}
		\Upsilon_r = \sum_{(i,j,0)\in \Gamma_r} \psi_i \phi_j.
\end{eqnarray}
Then, we have the following.

\begin{lemma}\label{lem:higher_BC}
	Let $(V,R,\mu)$ be a braided algebra, and let $(\tilde V, R_n, \mu_n)$ be a braided commutative deformation of degree $n$. Then $(V, R_{n+1}, \mu_{n+1})$ is a braided commutative deformation of $(V,R,\mu)$ of order $n+1$ if and only if 
	\begin{center}
		$
		\begin{cases}
		\delta^2_{\rm YB}\phi_{n+1} + \Theta_{n+1} = 0\\
		\delta^2_{\rm IY}(\phi_{n+1} \oplus\psi_{n+1}) + \Xi^{IY}_{n+1} - \Omega^{IY}_{n+1} = 0\\
		\delta^2_{\rm  YI}(\phi_{n+1} \oplus\psi_{n+1}) + \Xi^{YI}_{n+1} - \Omega^{YI}_{n+1} = 0\\
		\delta^2_{\rm H}\psi_{n+1} + \Lambda_{n+1} =0\\
		(\delta^2_{\rm BC} - \delta^2_{\rm YBH})(\phi_{n+1} \oplus \psi_{n+1}) + \Upsilon_{n+1} = 0.
		\end{cases}
		$
	\end{center}
\end{lemma}
\begin{proof}
		Applying Lemma~\ref{lem:higher}, it follows that the only constraint that we need to check is the last equation:
		\begin{eqnarray}\label{eqn:BC_higher}
			\delta^2_{\rm BC} (\phi_n \oplus \psi_n) + \Upsilon_{n+1} = 0.
		\end{eqnarray} 
		Braided commutativity on $(V, R_{n+1}, \mu_{n+1})$ gives the equation
		\begin{eqnarray*}
				R_{n+1}\mu_{n+1} - \mu_{n+1} = 0,
		\end{eqnarray*}
		which holds up to degree $n$ by hypothesis. On degree $n+1$ we obtain
		\begin{eqnarray}\label{eqn:BC_higher_expanded}
				\sum_{i+j=n+1} \psi_i  \phi_j - \psi_{n+1} = 0.
		\end{eqnarray}
		Using the definition of $\delta^2_{BC}$, we see that 
		$(\delta^2_{BC} - \delta^2_{\rm YBH})(\phi_{n+1} \oplus \psi_{n+1}) = \psi_{n+1} \phi_0 + \psi_0 \phi_{n+1} - \psi_{n+1}$. 
		Therefore, Equation~\eqref{eqn:BC_higher_expanded} gives precisely Equation~\eqref{eqn:BC_higher}. This means that a YBH deformation of degree $n+1$ is braided commutative if and only if Equation~\eqref{eqn:BC_higher} holds, therefore completing the proof. 
\end{proof}

For the the second order deformation of braided commutativity, one computes,
modulo $\hbar^3$, 
\begin{eqnarray*}
\tilde R  \tilde \mu - \tilde \mu &=&
(R + \hbar \phi_1   + \hbar^2 \phi_2) ( \mu   + \hbar \psi_1 + \hbar^2 \psi_2)
- ( \mu + \hbar \psi+ \hbar^2 \psi_2) \\
&=&
[R \mu - \mu] + \hbar [ R \psi_1 + \phi_1 \mu - \psi_1 ] 
+ \hbar^2 [ ( R \psi_2 + \phi_2 \mu - \psi_2 ) 
+ \phi_1 \psi_1 ] .
\end{eqnarray*}

Then for quadratic deformations, we have the following.

\begin{figure}[htb]
	\begin{center}
		\includegraphics[width=5in]{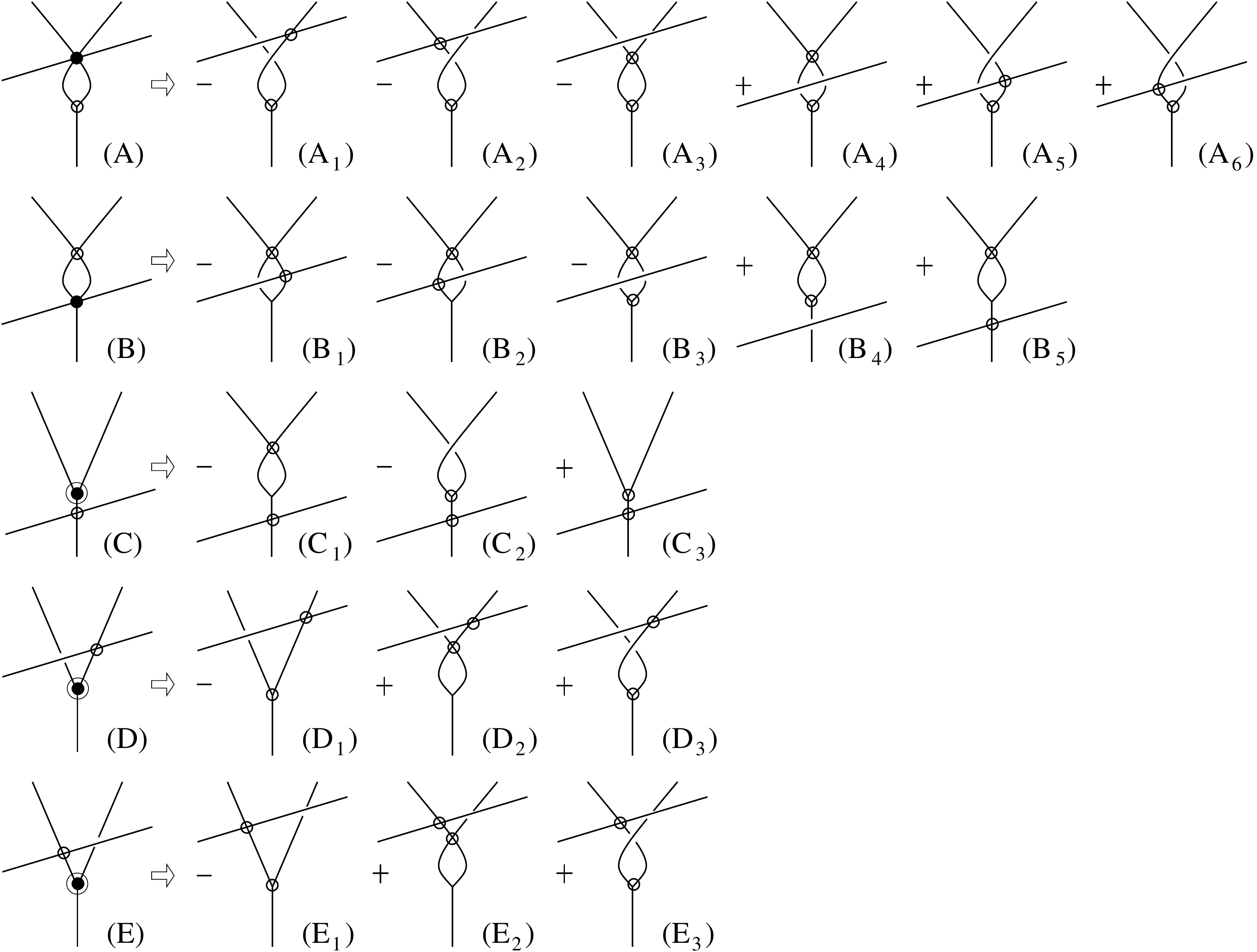}
	\end{center}
	\caption{}
	\label{BC2cocylist}
\end{figure}

\begin{theorem}\label{thm:higher_BC}
For $(\psi, \phi) \in Z^2_{\rm YBH} (X,X)$, we have 
$$\delta^3_{\rm BC:YI}(\Xi^{YI}_{2} - \Omega^{YI}_{2} - \phi \psi)=0 \quad {\rm and } \quad 
\delta^3_{\rm BC:IY}(\Xi^{IY}_{2} - \Omega^{IY}_{2} - \phi \psi)=0 . $$
Therefore, the obstruction to quadratic deformations lies in the third braided commutative cohomology group.
\end{theorem}

\begin{proof}
In Figure~\ref{BC2cocylist}, the 2-cocycle conditions are represented diagrammatically
with addtional 2-cocycle placed, and each term is labeled. 
In these figures the 3-cocycle is represented by black vertices to distinguish them from
2-cocycles in white circles.
Each row is equated to zero as 2-cocycle condition.
In Figure~\ref{Diff30}, the term $\delta^3_{\rm BC:YI}(\Xi^{IY}_{2} - \Omega^{IY}_{2} - \phi \psi)$
is represented.

In each of square the frames, on the left-hand side, a tree diagram that appears in the 3-cocycle condition in Figure~\ref{BCdiff3} is depicted. On the right-hand side, all possible ways to place a pair of white vertices (representing 2-cocycles) are listed. There are two types of these diagrams with two white vertices; those that appear in 
$\delta^3_{\rm BC:YI}(\Xi^{IY}_{2} - \Omega^{IY}_{2} - \phi \psi)$ and those that do not. 
Those that do not appear are listed in pairs with opposite signs, below dotted horizontal lines in a square frame, if there are any.
Hence the sum of the maps represented by these tree diagrams is equal to $\delta^3_{\rm BC:YI}(\Xi^{IY}_{2} - \Omega^{IY}_{2} - \phi \psi)$.

The labels in each term in Figure~\ref{Diff30} correspond to those in Figure~\ref{BC2cocylist}, and the total sum vanishes.
Those labeled by $Z_i$ cancel with opposite signs.
The IY case is similar.
This shows that $\delta^3_{\rm BC}(\Xi^{IY}_{2} - \Omega^{IY}_{2} - \phi \psi)=0$.

The obstruction to quadratic deformation for a braided algebra lies in $H^3_{\rm YBH}(X,X)$, as an application of Lemma~\ref{lem:3cocy}. From the previous part of the proof of this theorem, $\delta^3_{\rm BC:YI}$ and $\delta^3_{\rm BC:IY}$ map $(\Xi_2 - \Omega_2)\oplus \Theta_2\oplus\Upsilon_2$ (i.e. the braided commutative part of the obstruction) to zero. Since $\delta^3_{\rm BC} = \delta^3_{\rm YBH} \oplus \delta^3_{\rm BC:YI} \oplus \delta^3_{\rm BC:IY}$, and making use of the fact that all the $\delta^2$ terms appearing in Lemma~\ref{lem:higher_BC} are mapped to zero by $\delta^3_{\rm BC}$, it follows that the obstruction to extending an infinitesimal braided commutative deformation to a quadratic braided commutative deformation lies in $H^3_{\rm BC}(X,X)$.
\end{proof}

Thus we obtain the following  useful criterion for the extension of infinitesimal deformations to quadratic deformations.

\begin{corollary}
	If $H_{\rm BC}^3(V,V) = 0$, then any infinitesimal deformation can be extended to a quadratic deformation.
\end{corollary}

\begin{proof}
			From Theorem~\ref{thm:higher_BC} we know that the obstruction to quadratic braided commutative deformations lies in $H^3_{\rm BC}(X,X)$. When $H^3_{\rm BC}(X,X) = 0$, any obstruction vanishes and any infinitesimal deformation can be integrated to a quadratic deformation.
\end{proof}

\begin{figure}[htb]
	\begin{center}
		\includegraphics[width=5.5in]{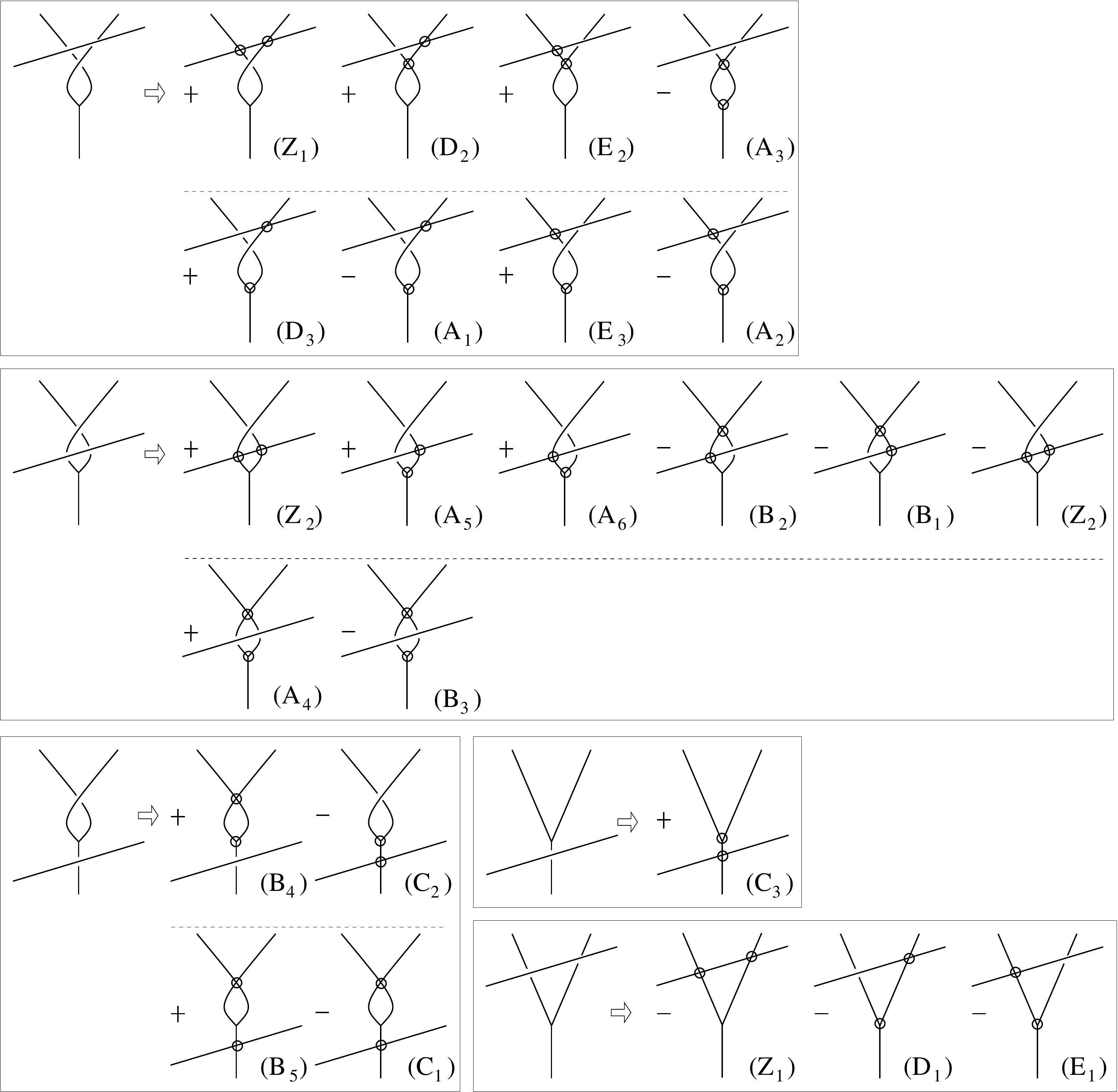}
	\end{center}
	\caption{}
	\label{Diff30}
\end{figure}

	\begin{figure}[htb]
					\begin{center}
						\includegraphics[width=3in]{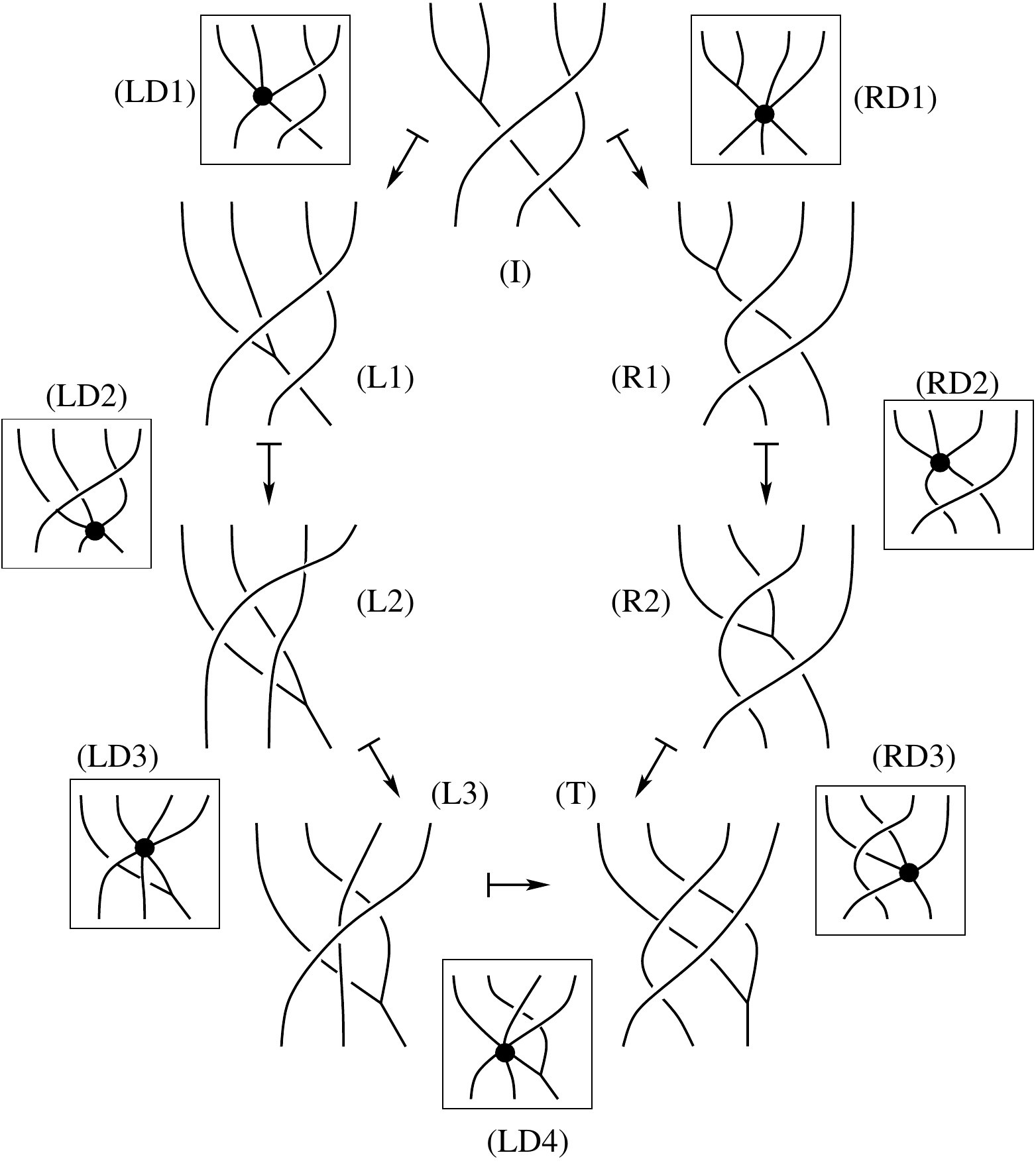}
					\end{center}
					\caption{}
					\label{YII}
				\end{figure}
				
	Next we extend the preceding result to obstructions of higher dimensions by utilizing graph diagrams 
	 in the context of 
		 the calculus of deviations by Markl and Stasheff \cite{deviations}. 
			In the following, we distinguish the result for YBH cohomology, and BC cohomology, as they are both interesting in their own right. 
		
		\begin{theorem}\label{thm:higher_YBH_deg_n}
			Let $\bar \psi = \sum_{i=0}^n \hbar^i \psi_i$ and $\bar \phi =  \sum_{i=0}^n \hbar^i \phi_i$ be a degree $n$ deformation of $\mu = \psi_0$ and $R = \phi_0$. Let $\psi_{n+1}$, $\phi_{n+1}$, $\Theta_{n+1}$, $\Xi^{\bullet}_{n+1}$, $\Omega^{\bullet}_{n+1}$, $\Lambda^{\bullet}_{n+1}$ be as in Lemma~\ref{lem:higher_BC}. Then, we have 
			\begin{eqnarray*}
					\delta^3_{\rm YBH}(\Theta_{n+1}\oplus(\Xi^{\rm IY}_{n+1}-\Omega^{\rm IY}_{n+1})\oplus(\Xi^{\rm YI}_{n+1}-\Omega^{\rm YI}_{n+1})\oplus\Lambda_{n+1})  = 0.
			\end{eqnarray*}
			In other words, the obstruction to extending a degree $n$ braided algebra deformation to a degree $n+1$ deformation lies in the third YBH cohomology group. 
		\end{theorem}
		\begin{proof}
				We start by showing the result for the YBH component of the braided commutative third differential, and then consider the braided commutative constraint. The differential $\delta^3_{\rm BC}$ contains a direct summand that coincides with the Hochschild differential, for which the result is known. Similarly, the YB component of $\delta^3_{\rm BC}$ also satisfies that $\delta^3_{\rm YB}(\Theta_{n+1}) = 0$. We need to verify the claim for the YI and IY conditions. Figure~\ref{YII} shows the diagrammatic procedure that was used in \cite{SZ-YBH} to introduce the third differential associated to the YI condition, which was called $\delta^{3,1}_{\rm YI}$ in \cite{SZ-YBH}. We want to show that $\delta^{3,1}_{\rm YI}(\Theta_{n+1}\oplus (\Xi^{\rm YI}_{n+1}-\Omega^{\rm YI}_{n+1})) = 0$.
			
				Let us consider the first arrow on the top left, which allows to pass from (I) to (L1). We replace the braiding by the deformed operator $\hat \phi = \bar \phi + \hbar^{n+1}\phi_{n+1}$, and the multiplication by the deformed operator $\hat \psi = \bar \psi + \hbar^{n+1}\psi_{n+1}$. The configuration (I) corresponds to the homomorphism
				\begin{eqnarray}\label{eqn:I_cond}
						f_1 := (\mathbb 1 \otimes \hat \phi)(\hat \phi \otimes \mathbb 1)(\hat \psi\otimes \hat \phi),
				\end{eqnarray} 
				while configuration (L1) corresponds to
				\begin{eqnarray}\label{eqn:L1_cond}
						f_2 := (\mathbb 1 \otimes  \hat \phi)(\mathbb 1\otimes \hat \psi\otimes \mathbb 1)(\hat \phi \otimes \mathbb 1^{\otimes 2})(\mathbb 1\otimes \hat \phi \otimes \mathbb 1)(\mathbb 1^{\otimes 2} \otimes \hat \phi). 
				\end{eqnarray}
				In both equations, we assume that each string corresponds to a copy of $V_\hbar := \mathbb k[[\hbar]] \otimes V$.
				Since Equations~\eqref{eqn:I_cond} and \eqref{eqn:L1_cond} differ by an application of the YI condition, which holds up to degree $n$ by assumption, we have that the diagram
				\begin{eqnarray*}
				\begin{tikzcd}
						V_\hbar \otimes V_\hbar \otimes V_\hbar \otimes V_\hbar\arrow[rrr,"\mathbb 1^{\otimes 4}"]\arrow[dd,"f_1"] & & & V_\hbar \otimes V_\hbar \otimes V_\hbar\otimes V_\hbar\arrow[dd,"f_2"]\\
						&&&\\
						V_\hbar \otimes V_\hbar \otimes V_\hbar\arrow[rrr,"\mathbb 1^{\otimes 3}"] & & & V_\hbar \otimes V_\hbar \otimes V_\hbar
				\end{tikzcd}
				\end{eqnarray*}
				commutes up to degree $n+1$ in $\hbar$, i.e. modulo the ideal $(\hbar^{n+1})$. In the language of Markl and Stasheff, we have a deviation $\hbar^{n+1}\alpha = f_1 - f_2$. By Lemma~\ref{lem:higher}, we have 
				\begin{eqnarray}
						\alpha = (\mathbb 1\otimes \phi_0)((\Xi^{\rm YI}_{n+1}-\Omega^{\rm YI}_{n+1})\otimes \mathbb 1)(\mathbb 1^{\otimes 2}\otimes \phi_0),
				\end{eqnarray} 
				which is called the deviation. Observe that $\alpha$ is diagrammatically represented by the diagram on the top-left arrow in Figure~\ref{YII}, where the black dot represents $\Xi^{\rm YI}_{n+1}-\Omega^{\rm YI}_{n+1}$. Graphically, following the convention of \cite{deviations}, we have
				\begin{eqnarray}\label{diag:deviation_YI}
					\begin{tikzcd}
						V_\hbar \otimes V_\hbar \otimes V_\hbar \otimes V_\hbar\arrow[rrr,"\mathbb 1^{\otimes 4}"]\arrow[dd,"f_1"] & & {}\arrow[dr,"\alpha",swap] & V_\hbar \otimes V_\hbar \otimes V_\hbar\otimes V_\hbar\arrow[dd,"f_2"]\\
						&&&{}\\
						V_\hbar \otimes V_\hbar \otimes V_\hbar\arrow[rrr,"\mathbb 1^{\otimes 3}"] & & & V_\hbar \otimes V_\hbar \otimes V_\hbar.
					\end{tikzcd}
				\end{eqnarray}
				Proceeding similarly along the perimeter of Figure~\ref{YII} we obtain several diagrams of this type with the deviations obtained from anticlockwise arrows called $\alpha_1, \alpha_2, \alpha_3, \alpha_4$. We call the deviations that run clockwise $\beta_1, \beta_2, \beta_3$. Here we follow the orientation convention of \cite{deviations}, and the $\beta_i$ are therefore considered with negative signs in the argument below. For each deviation, the black dot vertex with three inputs and three outputs represents $\Theta_{n+1}$, while the one with three inputs and two outputs represents $\Xi^{\rm YI}_{n+1}-\Omega^{\rm YI}_{n+1}$. By patching together the diagrams like~\eqref{diag:deviation_YI} along their common edges, we obtain a closed surface homeomorphic to a sphere, as in Figure~\ref{fig:Markl-Stasheff}. 
				\begin{figure}[htb]
					\begin{center}
						\includegraphics[width=2in]{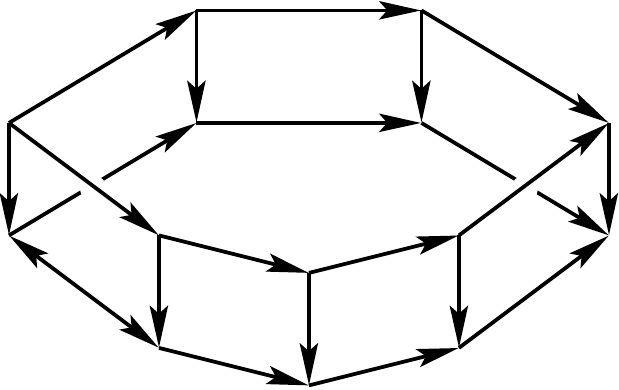}
					\end{center}
					\caption{}
					\label{fig:Markl-Stasheff}
				\end{figure}
				Therefore, the standard argument of Markl-Stasheff in \cite{deviations} gives that $\alpha_1 + \alpha_2 + \alpha_3 + \alpha_4 - \beta_1 - \beta_2 - \beta_3 = 0$. However, we notice that each of $\alpha_i$ and $\beta_j$ corresponds to one of the summands of $\delta^{3,1}_{\rm YI}$. We have therefore proved that $\delta^{3,1}_{\rm YI}(\Theta_{n+1}\oplus (\Xi^{\rm YI}_{n+1}-\Omega^{\rm YI}_{n+1})) = 0$. A simlar procedure, by means of Figure~13 and Figure~14 in \cite{SZ-YBH}, shows that also $\delta^{3,2}_{\rm YI}((\Xi^{\rm YI}_{n+1}-\Omega^{\rm YI}_{n+1})\oplus (\Xi^{\rm IY}_{n+1}-\Omega^{\rm IY}_{n+1})) = 0$ and $\delta^{3,3}((\Xi^{\rm YI}_{n+1}-\Omega^{\rm YI}_{n+1})\oplus \Lambda_{n+1}) = 0$. Therefore, $\delta^3_{\rm YI}(\Theta_{n+1}\oplus(\Xi^{\rm YI}_{n+1}-\Omega^{\rm YI}_{n+1})\oplus(\Xi^{\rm IY}_{n+1}-\Omega^{\rm IY}_{n+1})\oplus\Lambda_{n+1}) = 0$. Similarly, we obtain that  $\delta^3_{\rm IY}(\Theta_{n+1}\oplus(\Xi^{\rm IY}_{n+1}-\Omega^{\rm IY}_{n+1})\oplus(\Xi^{\rm YI}_{n+1}-\Omega^{\rm YI}_{n+1})\oplus\Lambda_{n+1}) = 0$. This shows that 
				\begin{eqnarray*}
						\delta^3_{\rm YBH}(\Theta_{n+1}\oplus(\Xi^{\rm IY}_{n+1}-\Omega^{\rm IY}_{n+1})\oplus(\Xi^{\rm YI}_{n+1}-\Omega^{\rm YI}_{n+1})\oplus\Lambda_{n+1}) = 0,
				\end{eqnarray*}
				which completes the proof.
		\end{proof}
	
		\begin{corollary}
			If $H^3_{\rm YBH}(V,V) = 0$, then any braided infinitesimal deformation of $V$ can be extended to a full deformation of arbitrary degree. 
		\end{corollary}
	
		\begin{theorem}\label{thm:higher_BC_deg_n}
					Let $\bar \psi = \sum_{i=0}^n \hbar^i \psi_i$ and $\bar \phi =  \sum_{i=0}^n \hbar^i \phi_i$ be a degree $n$ deformation of $\mu = \psi_0$ and $R = \phi_0$. Let $\psi_{n+1}$, $\phi_{n+1}$, $\Theta_{n+1}$, $\Xi^{\bullet}_{n+1}$, $\Omega^{\bullet}_{n+1}$, $\Lambda^{\bullet}_{n+1}$ and $\Upsilon_{n+1}$ be as in Lemma~\ref{lem:higher_BC}. Then, we have 
					\begin{eqnarray*}
						\delta^3_{\rm BC}(\Theta_{n+1}\oplus(\Xi^{\rm IY}_{n+1}-\Omega^{\rm IY}_{n+1})\oplus(\Xi^{\rm YI}_{n+1}-\Omega^{\rm YI}_{n+1})\oplus\Lambda_{n+1} \oplus \Upsilon_{n+1})  = 0.
					\end{eqnarray*}
					In other words, the obstruction to extending a degree $n$ braided commutative algebra deformation to a degree $n+1$ deformation lies in the third BC cohomology group.
		\end{theorem}
		\begin{proof}
				Since the YBH part was already handled in Theorem~\ref{thm:higher_YBH_deg_n}, to complete the proof we just need to show that 
				\begin{eqnarray}\label{BCeq1}
					\delta^3_{\rm BC:YI}(\Theta_{n+1}\oplus(\Xi^{\rm IY}_{n+1}-\Omega^{\rm IY}_{n+1})\oplus(\Xi^{\rm YI}_{n+1}-\Omega^{\rm YI}_{n+1})\oplus\Lambda_{n+1}\oplus \Upsilon_{n+1}) = 0,
				\end{eqnarray}
				and
				\begin{eqnarray}\label{BCeq2}
					\delta^3_{\rm BC:IY}(\Theta_{n+1}\oplus(\Xi^{\rm IY}_{n+1}-\Omega^{\rm IY}_{n+1})\oplus(\Xi^{\rm YI}_{n+1}-\Omega^{\rm YI}_{n+1})\oplus\Lambda_{n+1}\oplus \Upsilon_{n+1}) = 0.
				\end{eqnarray}
				This is done following the same procedure as before. To show the first equality, we consider Figure~\ref{BCdiff3}, and using Lemma~\ref{lem:higher_BC} we compute the deviation for each of the arrows, whic are replaced by the corresponding commutative diagrams up to terms of degree $n+1$ in $\hbar$. By glueing all diagrams together, the arguments of Markl-Stasheff here give that all the deviations sum up to zero. Since each deviation corresponds to one of the summands defining $\delta^3_{\rm BC:YI}$, we obtain Equation~\ref{BCeq1}
				and 
				as required. By considering the diagrams defining $\delta^3_{\rm BC:IY}$ we also obtain
				Equation~\ref{BCeq2}, and 
				the proof is complete.
		\end{proof}
	
		\begin{corollary}
			If $H^3_{\rm BC}(V,V) = 0$, then any braided commutative infinitesimal deformation of $V$ can be extended to a full deformation of arbitrary degree. 
		\end{corollary}

\section{Toward Higher Dimensions}\label{sec:complex}

In this section, we explore possible approaches to extending the cohomology theory developed in this article to higher dimensions. 

\subsection{Extending Yang-Baxter cohomology to higher dimensions}

We start by giving a cochain complex for Yang-Baxter cohomology that coincides in low-dimensions with the one discussed above. 
Differential maps in these low  dimensions are defined in \cite{SZ-YBH}, and this subsection gives a generalization to higher dimensions. 

\begin{definition} {\rm 
Let ${\mathbb k}$ be a unital ring and $V$ a ${\mathbb k}$-module. 
Let $R: V^{\otimes 2} \rightarrow V^{\otimes 2}$ be a Yang-Baxter operator. 
We set $\sigma_{n,i} = \mathbb 1^{\otimes (i-1)}\otimes R \otimes \mathbb 1^{\otimes(n-i)}$. 
Define 
	$d_{{\rm YB},i}^n : {\rm Hom}(V^{\otimes n}, V^{\otimes n}) \rightarrow {\rm Hom}(V^{\otimes (n+1)}, V^{\otimes (n+1)})$
	for $\phi \in {\rm Hom}(V^{\otimes n}, V^{\otimes n})$
	by 
\begin{eqnarray*}
		d_{{\rm YB},i}^n(\phi) &=&\sigma_{n+1,i-1}\cdots\sigma_{n+1,2}\sigma_{n+1,1}(\mathbb 1\otimes \phi)\sigma_{n+1,1}\cdots\sigma_{n+1,n-i}\sigma_{n+1,n-i+1}\\
		&& - \sigma_{n+1,i}\cdots\sigma_{n+1,n-1}\sigma_{n+1,n}(\phi \otimes \mathbb 1)\sigma_{n+1,n}\cdots\sigma_{n+1,n-i+3}\sigma_{n+1,n-i+2}.
\end{eqnarray*}
Then the 
differentials for YB cohomology are defined by $d^n_{\rm YB} = \sum_{i=1}^{n+1} (-1)^{i+1}d_{{\rm YB},i}^n$.

}
\end{definition}

\begin{figure}[htb]
	\begin{center}
		\includegraphics[width=3in]{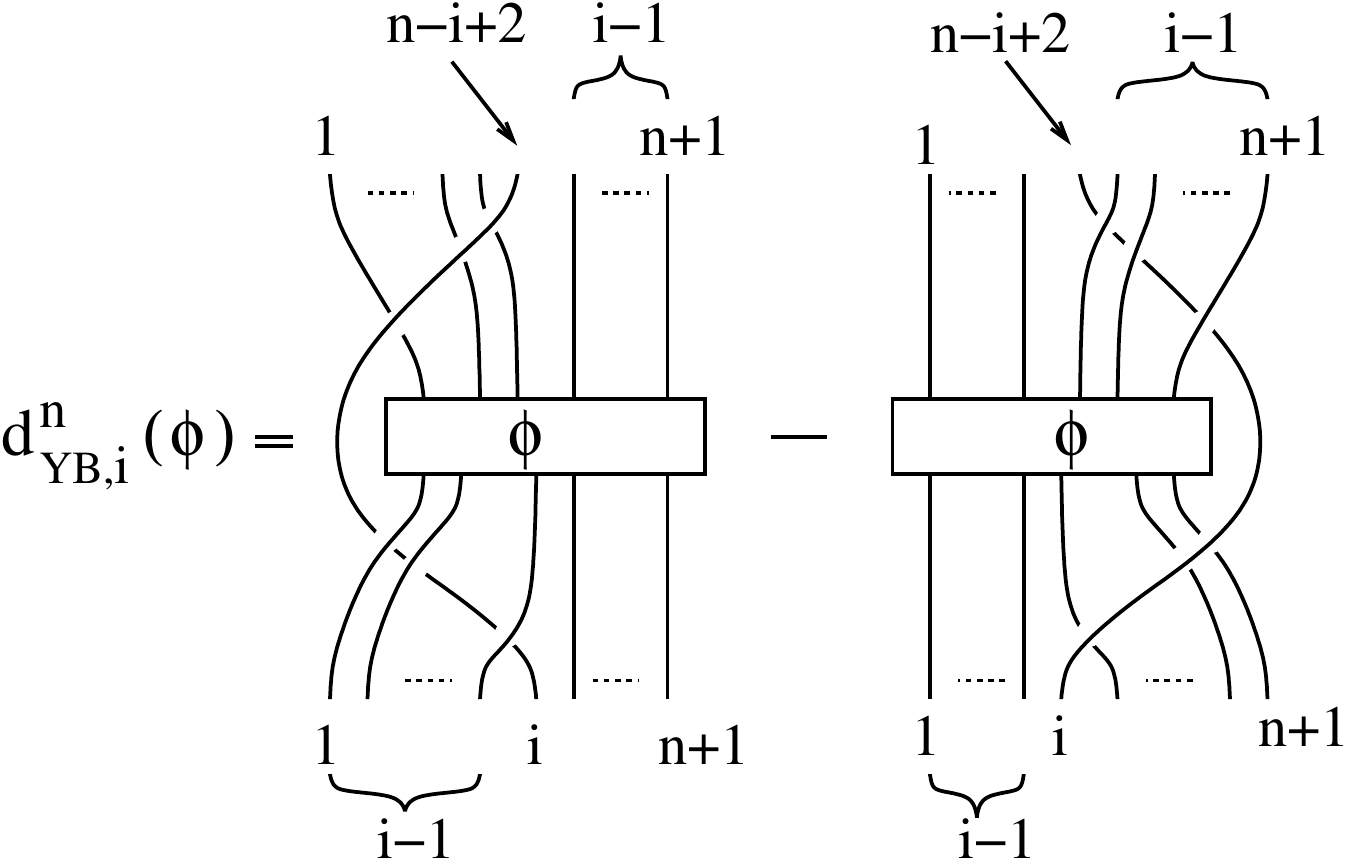}
	\end{center}
	\caption{ 
	}
	\label{fig:partial_YB}
\end{figure}

The partial differentials $d_{{\rm YB},i}^n$ are graphically depicted as in Figure~\ref{fig:partial_YB}.
We note that the positions of the  top and bottom crossings are not the same. 
This is different from the definition used, for instance, in \cite{Eisermann1}. 
An advantage of this definition
 is that no use of the inverse of $R$ is made, meaning that our cohomology can be also applied to pre-YB operators, i.e. operators that satisfy the YBE but that are not necessarily invertible.
 
\begin{figure}[htb]
	\begin{center}
		\includegraphics[width=5in]{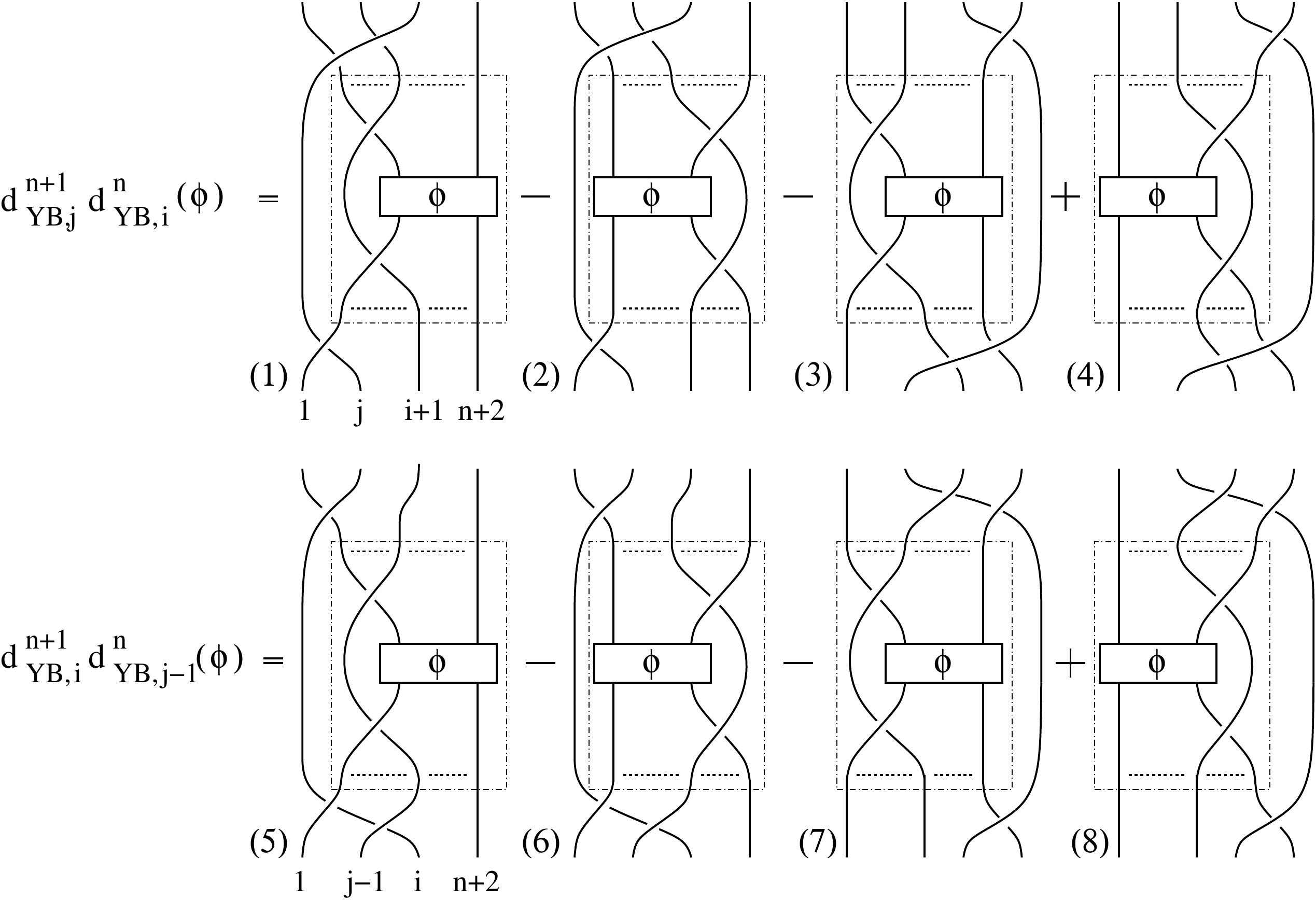}
	\end{center}
	\caption{
	}
	\label{fig:YBdjdi}
\end{figure}

 \begin{theorem}
 The differentials 
 $(  {\rm Hom}(V^{\otimes n}, V^{\otimes n}), d^n_{ {\rm YB} } )$ form a cochain complex.
 \end{theorem}
 \begin{proof}
   To prove that $d^n_{\rm YB} = \sum_{i=1}^n (-1)^{i+1}d_{{\rm YB},i}^n$ gives a cochain complex, by standard arguments 
   that it suffices to show that whenever $i<j$ the equality $d_{{\rm YB},j}^{n+1}d_{{\rm YB},i}^n = d_{{\rm YB},i}^{n+1}d_{{\rm YB},j-1}^n$ holds. This is more easily obtained using the diagrammatics of Figure~\ref{fig:partial_YB}. The computation is shown in Figure~\ref{fig:YBdjdi}.
   Specifically, the terms (1) and (5), (2) and (7), (3) and (6), (4) and (8) each cancel. 
   The first and the last canceling pairs use the assumption that $R$ is a Yang-Baxter operator. Also, we notice that the strings coincide in both cases due to the numbering shifts induced by the crossings. 
   \end{proof}

A direct computation shows that in degrees 1 and 2 one obtains $\delta^1_{\rm YB}$ and $\delta^2_{\rm YB}$ as defined before. For the third differential one obtains the differential depicted diagrammatically in Figure~\ref{fig:YB3cocy}.

\begin{figure}[htb]
	\begin{center}
		\includegraphics[width=3.5in]{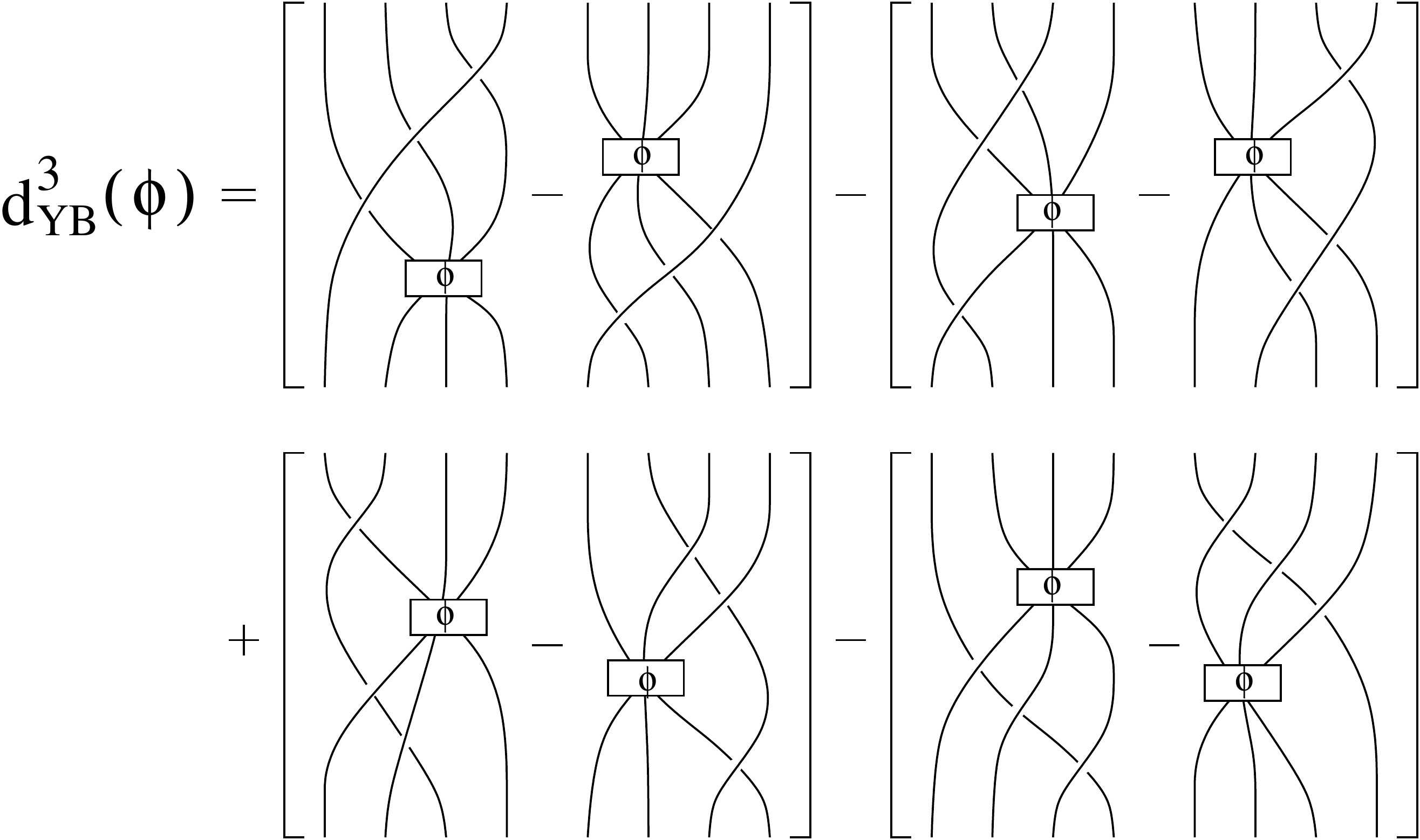}
	\end{center}
	\caption{ 
	}
	\label{fig:YB3cocy}
\end{figure}

\begin{figure}[htb]
	\begin{center}
		\includegraphics[width=3in]{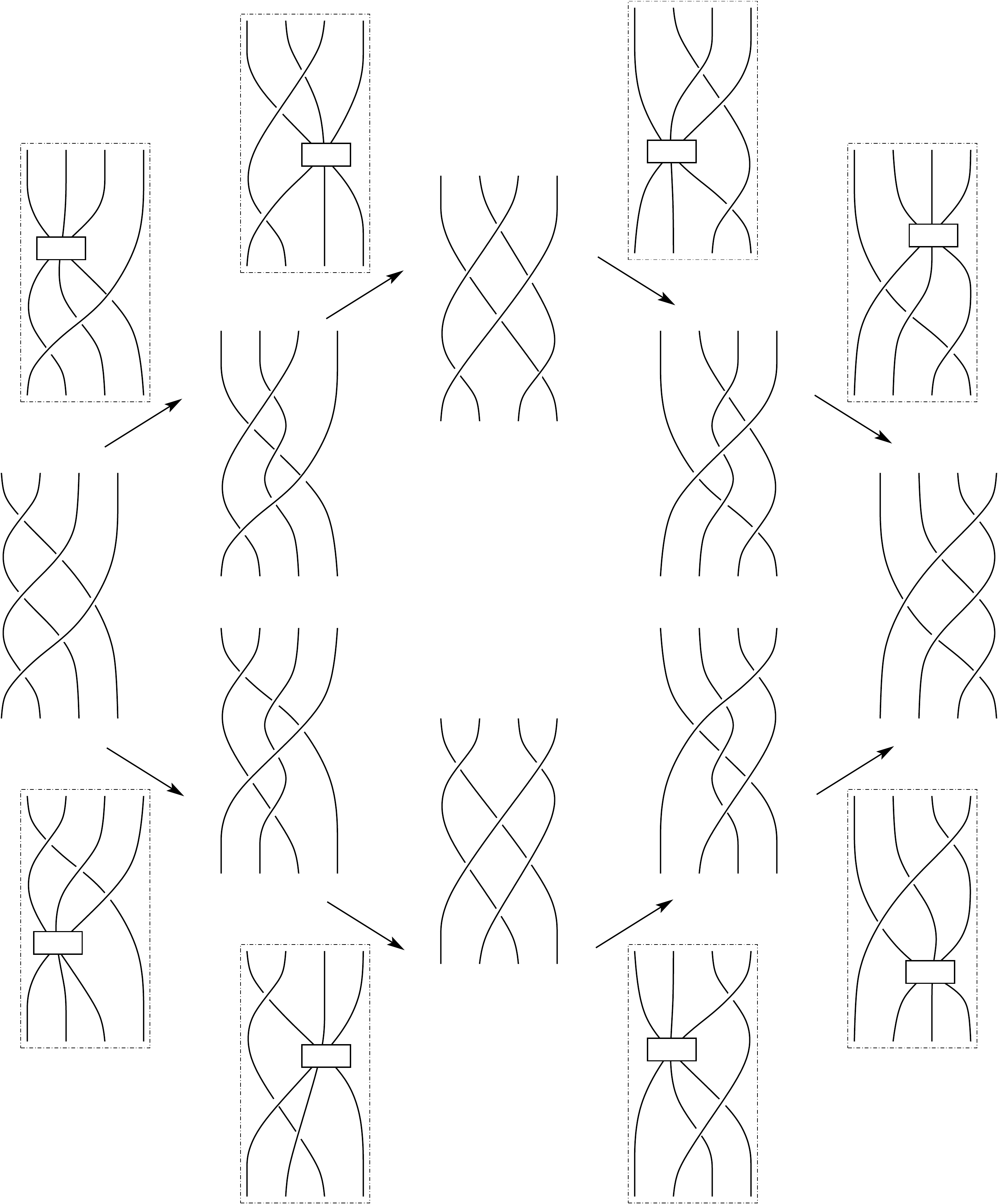}
	\end{center}
	\caption{ 
	}
	\label{fig:tetra}
\end{figure}

In Figure~\ref{fig:tetra}, at the center, two sequences (top and bottom) of Reidemeister type III moves 
are depicted, that was used in \cite{CJKLS} to define quandle cocycle invariants for knotted surfaces.
At the instance of a move, a triple point of the projection  appears. When a triple point is replaced by rectangles, we obtain
the diagrams framed and placed at arrows in the diagram. 
We observe that these diagrams appear in the YB 3-differential in Figure~\ref{fig:YB3cocy}.
This correspondence is used in the last subsection in Figure~\ref{fig:BCd41}.

\subsection{Extending Yang-Baxter Hochschild cohomology to higher dimensions}

To define a higher dimensional YBH cohomology whose low-dimensional differentials coincide with the ones studied in this article it is natural to attempt reconstructing YBH cohomology as a multicomplex \cite{Boardman,Mar-Shn,Shr}. In this situation, one has a doubly graded module $K^{s,t}$ where homomorphisms of total degree $1$ are defined, $d_r : K^{s,t} \rightarrow K^{s+r,t-r+1}$, with bidegrees $(r,1-r)$. The maps $d_r$ need their sum $\partial = \sum_r d_r : K^n \rightarrow K^{n+1}$ to satisfy $\partial^2 = 0$, where $K^n = \bigoplus_{s+t=n} K^{s,t}$ is the total complex. We pursue this approach in this subsection.

We define $K^{s,t} = {\rm Hom}(V^{\otimes(s+t)}, V^{\otimes(|t|+1)})$, whenever $s>0$, $s+t \geq |t| + 1$ and $s>t$, and put $K^{s,t} = 0$ otherwise. In this setup, then, we have the Hochschild cochain complex in bidegrees $(n,0)$, since $K^{n,0} =  {\rm Hom}(V^{\otimes n}, V)$. Moreover, we have the YB cochain complex in bidegrees $(2n-1,-n+1)$, since $K^{2n-1,-n+1} = {\rm Hom}(V^{\otimes n}, V^{\otimes n})$ for all $n\geq 1$. Note that while the Hochschild differential is of type $d_1$, i.e.~has bidegree 
$(1,0)$, the YB differential has bidegree $(2,-1)$ and therefore is of type $d_2$.

For $\psi \in K^{s,t}$, i.e. $\psi \in {\rm Hom}(V^{\otimes n}, V^{\otimes q})$ where $n = s+t$ and $q = |t| + 1 \geq 2$, we define a bidegree $(1,0)$ differential $d_1$ as follows. For $t>0$ we put $\rho = t$, 
and set
\begin{eqnarray*}
d_{1, 0}^{s,t} (\psi ) &=& (\mathbb 1^{\otimes (q-2)}\otimes \mu)\sigma_{q+1,q-2}\cdots\sigma_{q+1,1}(\mathbb 1 \otimes \psi) \\
d_{1,i}^{s,t}  (\psi ) &=& \psi(\mathbb 1^{\otimes i-1}\otimes \mu \otimes \mathbb 1^{n-i+1})\\
d_{1, \rho}^{s,t} &=& (\mathbb 1^{\otimes (q-1)}\otimes \mu)(\psi \otimes \mathbb 1)\sigma_{n+1,n}\cdots\sigma_{n+1,\rho+1}\sigma_{n+1,\rho}
\end{eqnarray*}
and define 
$$ d_1^{s,t} = d_{1,0}^{s,t} + \sum_{i=1}^{\rho -1}(-1)^i d_{1,i}^{s,t} + (-1)^{\rho} d_{1,\rho}^{s,t} . $$
In Figure~\ref{fig:YBHd1} the diagram representing the differential $d_1^{s,t} $ is depicted.

If $t\leq0$, we set $\rho = n-q+2$ and define
$$ d_1^{s,t} = (-1)^t d_{1,0}^{s,t} + \sum_{i=1}^{\rho -1}(-1)^i d_{1,i}^{s,t} + (-1)^{\rho} d_{1,\rho}^{s,t} . $$
\begin{figure}[htb]
	\begin{center}
		\includegraphics[width=4.5in]{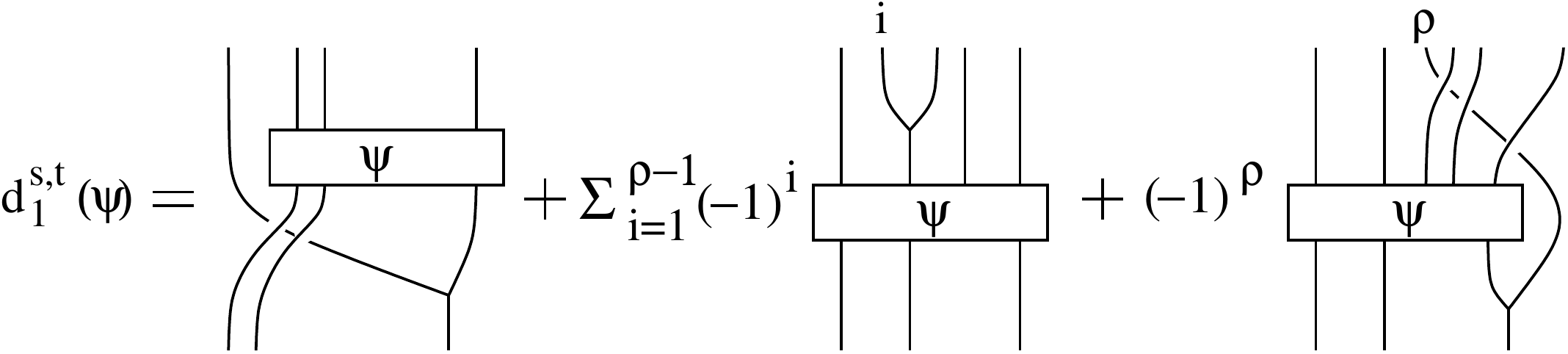}
	\end{center}
	\caption{ 
	}
	\label{fig:YBHd1}
\end{figure}

We also note that $d_1$ contains the Hochschild differential $d_{\rm H}$ 
as a special case when $t=0$. Indeed, the last sum in the definition of $d_1$ is identical to the corresponding term in Hochschild cohomology. The first two terms are the same as the two remaining terms of Hochschild cohomology. 
We note that when $t\neq 0$, cocycles $\psi$ are allowed to have more than one outputs. 

Moreover, for $\phi \in K^{s,t}$, i.e. $\phi \in {\rm Hom}(V^{\otimes n}, V^{\otimes q})$ where $n = s+t$ and $q = |t| + 1$, with the restriction that $t < 1$ and $(s,t)$ lies above the diagonal determined by $(2n-1,-n+1)$
(i.e., $n=s+t \geq q=|t|+1$), we introduce a bidegree $(2,-1)$ differential $d_2$ as follows. We set $\rho = q$ and define 
\begin{eqnarray*}
		d_2(\phi) &=& (-1)^n \sum_{i=1}^\rho(-1)^i [\sigma_{n+1,i-1}\cdots\sigma_{n+1,2}\sigma_{n+1,1}(\mathbb 1\otimes \phi)\sigma_{n+1,1}\cdots\sigma_{n+1,n-i}\sigma_{n+1,n-i+1}\\
		&& - \sigma_{n+1,i}\cdots\sigma_{n+1,n-1}\sigma_{n+1,n}(\phi \otimes \mathbb 1)\sigma_{n+1,n}\cdots\sigma_{n+1,n-i+3}\sigma_{n+1,n-i+2}.]. 
\end{eqnarray*}
If $t\geq 0$, we set $d_2=0$. 
We note that $d_2$ and $d_{\rm YB}$ are defined by the same formula. 
The differences are:  (1) the initial sign, 
(2) that the cochains are allowed to have less output factors  than inputs, and (3) that the summation for $d_{\rm YB}$
in the case of $n=q$  is up to $\rho=q+1$. 
Thus the differentials $d_1$ and $d_2$ represent generalized versions of Hochschild and YB differentials to $K^{s,t}$. 

\begin{figure}[htb]
	\begin{center}
		\includegraphics[width=6in]{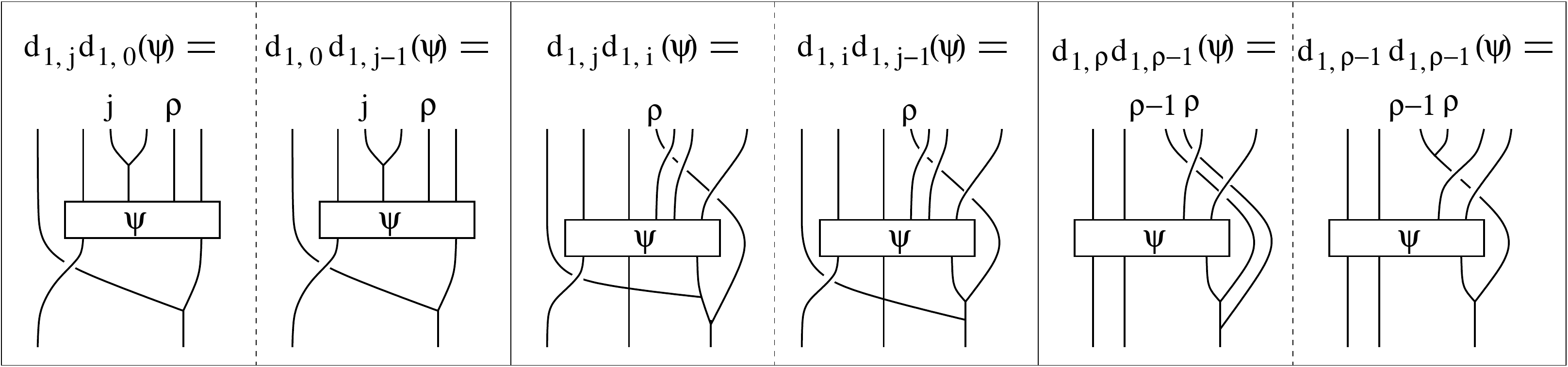}
	\end{center}
	\caption{ 
	}
	\label{fig:YBHd1d1}
\end{figure}

\begin{lemma}\label{lem:d_1_d2}
			We have $d_1^2 = d_2^2 = 0$. 
\end{lemma}
\begin{proof}
By standard arguments it is enough to show that $d_{a,j}d_{a,i} = d_{a,i}d_{a,j-1}$ for $a=1,2$ and $i<j$.
For $d_1^2=0$, the middle terms $d_{1,i}^{s,t}$ for $i=1, \ldots, \rho -1$ are similar to Hochschild differentials, 
and the proof for $i,j = 1, \ldots, \rho-1$ is analogous to the Hochschild case. 
For the other terms, the terms are depicted in Figure~\ref{fig:YBHd1d1}. 
In the figure, superscripts $s,t$ are suppressed, and  canceling terms for $d_{1,j} d_{1,i} =d_{1,i} d_{1, j-1} $ 
in the cases in question 
are depicted in each frame, with identical pairs listed with dotted lines in the middle.
For $d_2^2=0$, the proof is similar to the YB cohomology which is shown in Figure~\ref{fig:YBdjdi}.
\end{proof}

When ${\rm char} \mathbb k = 2$, i.e. when all signs in the definitions can be discarded, the multicomplex so defined indeed satisfies $\partial^2 = 0$, and therefore induces a cohomology theory in its total complex.

\begin{proposition}
		Let $\mathbb k$ a ring of characteristic $2$, and let $V$ be a braided algebra. Then, the complex $K^{\bullet,\bullet}$ with $d_1$ and $d_2$ defined above is a multicomplex, i.e. satisfies $\partial^2=0$, where $\partial = \sum_r d_r$. 
\end{proposition}
\begin{proof}
		The proof is a direct computation. First, observe that we need to check that all terms landing in a direct summand $K^{s,t}$ with $s+t=n$ via $\partial^2$ are zero. There are, by construction of $\partial$, four combinations of $d_1$ and $d_2$ that map into $K^{s,t}$. These are
		\begin{eqnarray*}
				K^{s-2,t} \xrightarrow{d_1} &K^{s-1,t}&  \xrightarrow{d_1}  K^{s,t},\\
				K^{s-3,t+1} \xrightarrow{d_2} &K^{s-1,t}&  \xrightarrow{d_1}  K^{s,t} , \\
				K^{s-3,t+1} \xrightarrow{d_1} &K^{s-2,t+1}&  \xrightarrow{d_2}  K^{s,t}, \\
				K^{s-4,t+2} \xrightarrow{d_2} &K^{s-2,t+1}&  \xrightarrow{d_2}  K^{s,t}.
		\end{eqnarray*}
		Then, we need to show that $d_1\circ d_1(\phi_1) + [d_1\circ d_2+d_2\circ d_1](\phi_2) + d_2\circ d_2(\phi_3) = 0$. By Lemma~\ref{lem:d_1_d2} we have that $d_1^2 = d_2^2 = 0$.
We 
show that $d_1\circ d_2 + d_2\circ d_1 = 0$. First, we observe that when $t\geq0$ both terms vanish since $d_2(K^{s,t}) = 0$ when $t>0$.  Hence  
we may assume that $t<0$ and therefore $t+1\leq 0$. For this case, we distinguish two subcases. The first one is when $d_2$ is not $d_{\rm YB}$, and the second one is when $d_2 = d_{\rm YB}$. This means that we distinguish the cases when $(s-3,t+1) \neq (2a-1,-a+1)$ for some $a$, and when $(s-3,t+1) = (2a-1,-a+1)$ for some $a$. 

In the first subcase, in Figures~\ref{fig:YBHd1d2} and \ref{fig:YBHd2d1} the terms of $d_1 d_2 (\phi)$ and $d_2 d_1 (\phi)$ are respectively depicted. 
The summation bounds are as follows.
The differential $d_1\circ d_2$ maps $K^{s-3,t+1}$ into $K^{s,t}$, and therefore $d_2$ is defined on ${\rm Hom}(V^{\otimes(s+t-2)},V^{\otimes (-t)})$, while $d_1$ is defined on ${\rm Hom}(V^{\otimes(s+t-1)},V^{\otimes (-t+1)})$. The sum over $i$, which refers to $d_1$ runs from $1$ to $\rho_1 = s+t-1-(-t+1)+2 = s+2t$. The sum on $j$, which refers to $d_2$ runs from $1$ to $\rho_2 = q$. However,  as in Figure~\ref{fig:partial_YB} this means that $d_2$ involves strings starting from $n-q+2 = s+t-2-(-t)+2 = s+2t$. Therefore, $i$ and $j$ are on the left and right of $\rho = s+2t$, respectively. For the composition $d_2\circ d_1$, similar calculations give sums up to $\rho_1 = s+2t$ and $\rho_2 = s+2t+1$. Since $d_1\circ d_2$ has a shift in the tensorands to which $d_2$ is applied, the summation bounds agree in both cases.  
Using the diagrammatics in Figures~\ref{fig:YBHd1d2} and \ref{fig:YBHd2d1}, along with the summation bounds obtained, one sees that $d_1\circ d_2 + d_2\circ d_1 = 0$. Observe that we used a single $\rho$ for $d_2\circ d_1$ as well, to indicate the position of $s+2t$, since in both cases the summation positions agree, once we consider the labeling shifts. 

When $d_2 = d_{\rm YB}$, a new subtlety arises, which involves the fact that the bounds for defining $d_{\rm YB}$ include an extra term in the sum. Notice that the differential involved are $d_1\circ d_{\rm YB}$ and $d_2\circ d_1$, since $K^{s-2,t+1}$ is not of type $K^{2a-1,-a+1}$. 
The new terms give rise, upon considering the composition $d_1\circ d_{\rm YB}$, to four extra terms, which are seen by direct inspection to cancel out.
This completes the proof.
\begin{figure}[htb]
	\begin{center}
		\includegraphics[width=6in]{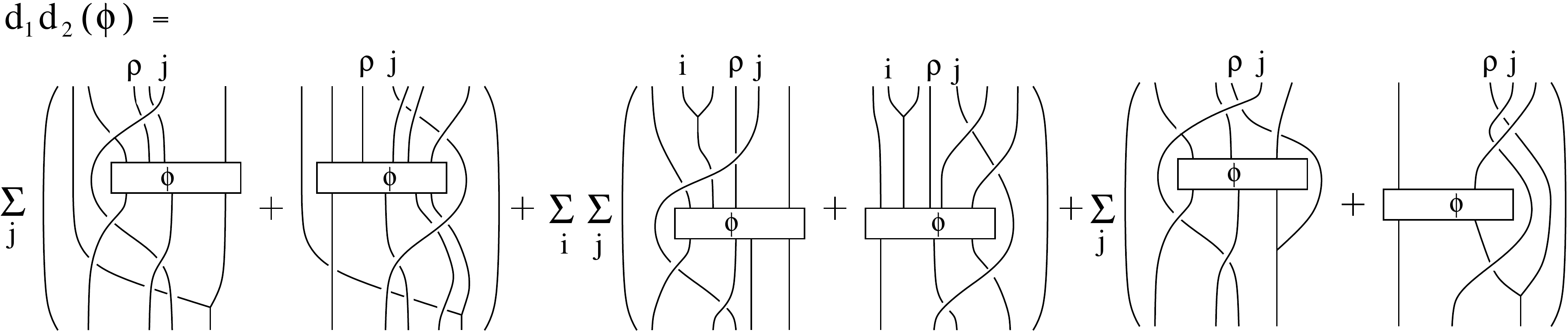}
	\end{center}
	\caption{  
	}
	\label{fig:YBHd1d2}
\end{figure}
\begin{figure}[htb]
	\begin{center}
		\includegraphics[width=6in]{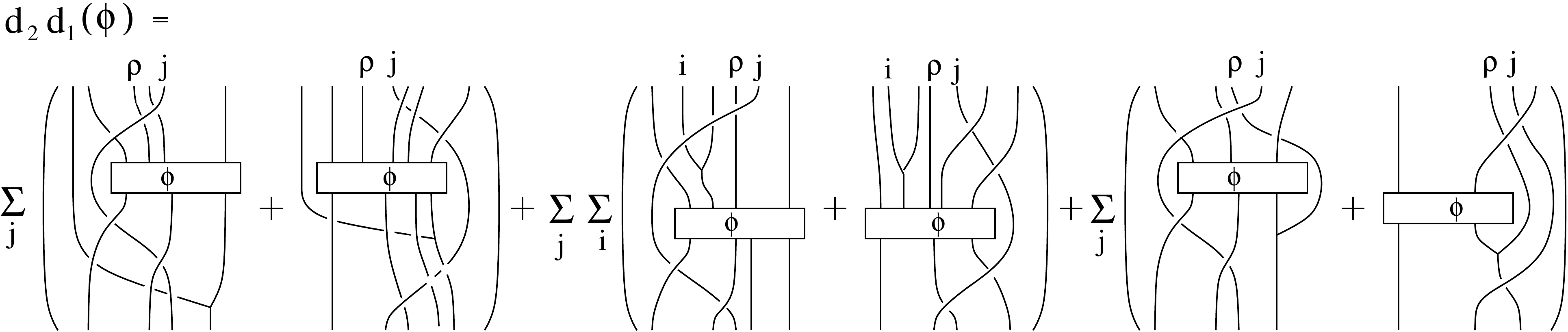}
	\end{center}
	\caption{ 
	}
	\label{fig:YBHd2d1}
\end{figure}
\end{proof}

It  is still unclear whether the associated total complex actually satisfies $\partial^2 = 0$ for all degrees when the characteristic of $\mathbb k$ is not $2$. 
 The main issue
  is the definition of the signs involved. 
 The cancellation scheme that gives $\partial^2 = 0$ in characteristic $2$ does not seem to apply because some signs appear concordant, therefore preventing cancellation. 

It is desirable that adjustments in signs in our differential to extend coefficients to integers.
Other possible approaches include the following.
In \cite{Mar-Shn,Shr}, the fact that all the terms of the multicomplex give a differential is proved using topological results on the associahedron.
 In our case, the permutoassociahedron considered in \cite{Kap}  may be relevant, but 
any relation  to the cohomology studied in this article is not clear.
Another approach may be  the iterated cones used in \cite{Yetter}.

 In the remaining part of the subsection, 
we recover the same differentials as in the YI part of the YBH cohomology complex \cite{SZ-YBH}
in low degrees. 
 There is another part called IY part defined in \cite{SZ-YBH}
and there is a counterpart for our complex. Due to  symmetry of the theory, however, we focus on the YI part in the following. 

At degree $1$, we have 
$$K^1 := \bigoplus_{s+t=1} K^{s,t} = K^{1,0} = {\rm Hom}(V,V).$$
 At deegree $2$, we have 
 $$K^2 := \bigoplus_{s+t=2} K^{s,t} = K^{2,0}\oplus K^{3,-1} = {\rm Hom}(V^{\otimes 2},V) \oplus {\rm Hom}(V^{\otimes 2},V^{\otimes 2}).$$
 Hence both $K^1$ and $K^2$ coincide with the YI cochain groups. The differential $\partial : K^1 \rightarrow K^2$ is the sum of the only two nonzero terms, namely $d_{\rm H} : K^{1,0} \rightarrow K^{2,0}$ and $d_{\rm YB}: K^{1,0} \rightarrow K^{3,-1}$. This coincides with $\delta^1_{\rm YI}$. 

Let us now consider the terms in degree $3$ of the associated total complex. We have 
\begin{eqnarray*}
K^3 & := &  \bigoplus_{s+t=3} K^{s,t} = K^{2,1} \oplus K^{3,0} \oplus K^{4,-1} \oplus K^{5,-2} \\
&=& {\rm Hom}(V^{\otimes 3},V^{\otimes 2}) \oplus {\rm Hom}(V^{\otimes 3},V^{\otimes 1}) \oplus {\rm Hom}(V^{\otimes 3},V^{\otimes 2})\oplus {\rm Hom}(V^{\otimes 3},V^{\otimes 3}).
\end{eqnarray*}
 We now compute $\partial: K^2 \rightarrow K^3$ by computing each component and summing them together in their respective direct summands. 
 We have the Hochschild differential $ d^{2,0}_1 =  d_{\rm H} : K^{2,0} \rightarrow K^{3,0}$, and the YB differential $d^{3,-1}_2 =  d_{\rm YB} : K^{3,-1} \rightarrow K^{5,-2}$. Additionally, we have 
 $d_1  = d^{3,-1}_1    : K^{3,-1} \rightarrow K^{4,-1}$ and $d_2 = d^{2,0}_2   : K^{2,0} \rightarrow K^{4,-1}$. Using the definitions, with $\alpha\in K^{3,-1}$ and $\beta\in K^{2,0}$ we find
 \begin{eqnarray*}
 		d_1(\alpha) &=& (\mathbb 1\otimes \mu)(R\otimes \mathbb 1)(\mathbb 1\otimes \alpha) 
		 - \alpha (\mu \otimes \mathbb 1 ) 
		+ (\mathbb 1\otimes \mu)(\alpha \otimes \mathbb 1)(\mathbb 1\otimes R),\\
 		d_2(\beta) &=& (\mathbb 1\otimes \beta)(R\otimes \mathbb 1)(\mathbb 1\otimes R) - R(\beta\otimes \mathbb 1).
 \end{eqnarray*}
The sum $d_1(\alpha) + d_2(\beta)$ gives precisely the term $\delta^2_{\rm YI}$ upon renaming $\alpha = \phi$ and $\beta = \psi$. 

Lastly, we consider the total complex at degree $4$, and the corresponding third differential. We have 
\begin{eqnarray*}
K^4 &:=& \bigoplus_{s+t=4} K^{s,t} = K^{3,1} \oplus K^{4,0} \oplus K^{5,-1} \oplus K^{6,-2} \oplus K^{7,-3} \\
&=&  {\rm Hom}(V^{\otimes 4},V^{\otimes 2}) \oplus {\rm Hom}(V^{\otimes 4},V^{\otimes 1}) \oplus {\rm Hom}(V^{\otimes 4},V^{\otimes 2}) \\
& & \hspace{1.1in} \oplus \, {\rm Hom}(V^{\otimes 4},V^{\otimes 3}) \oplus {\rm Hom}(V^{\otimes 4},V^{\otimes 4}) .
\end{eqnarray*}
 This agrees with the YBH deformation of the YI component of this article. The corresponding nontrivial differentials that add up to give the total differential $\partial$ are as follows. We have $  d^{3,0}_1  = d_{\rm H}: K^{3,0} \rightarrow K^{4,0}$ and $ d^{5,-2}_2 =  d_{\rm YB} : K^{5,-2} \rightarrow K^{7,-3}$. We note that no other differentials have $K^{4,0}$ and $K^{7,-3}$ as target, where we use the fact that $d_2 = d_{\rm YB}$ (as pointed out above) when $K^{s,t}$ has the form $K^{2n-1,-n+1}$, and the additional fact that $d_2 = 0$ on $K^{s,t}$ with $t\geq 1$. 
 Therefore, we obtain the Hochschild and YB components of the third differentials of YBH. 

Let us now calculate the remaining terms, and compare them with the YI component of $\delta^3_{\rm YBH}$. 
 In \cite{SZ-YBH}, the third YBH differential is defined as follows.
Define $\delta^3_{\rm YI} (V,V) : C^3_{\rm YI} (V,V) \rightarrow  C^4_{\rm YI} (V,V)$ by 
	$\delta^3_{\rm YI} = \delta^3_{\rm YB} \oplus \delta^{3,1}_{\rm YI}  \oplus  \delta^{3,2}_{\rm YI}
	\oplus \delta^{3,3}_{\rm YI}   \oplus \delta^3_{\rm H}$, where each direct summand of the differential is defined as follows.
For $\alpha \in  C^{3,2}_{\rm YI}(V,V)$ and $\beta \in C^{3,3}_{\rm YBH}(V,V)$, 
define 
$\delta^{3,1}_{\rm YI} : C^{3,3}_{\rm YBH}(V,V)  \oplus C^{3,2}_{\rm YI}(V,V)
\rightarrow C^{4,3}_{\rm YI}(V,V)   $
by
\begin{eqnarray*}
\delta^{3,1}_{\rm YI}(\beta \oplus \alpha  ) 
&=& 
( {\mathbb 1} \otimes R ) ( \alpha \otimes {\mathbb 1} ) ( {\mathbb 1}^{\otimes 2}  \otimes R )
+( {\mathbb 1} \otimes \alpha ) ( R \otimes {\mathbb 1}^{\otimes 2} ) (  {\mathbb 1} \otimes R \otimes {\mathbb 1} ) ( {\mathbb 1}^{\otimes 2}  \otimes  R ) 
\\ & & 
+
( {\mathbb 1}^{\otimes 2}  \otimes   \mu ) (    {\mathbb 1} \otimes R \otimes {\mathbb 1} ) 
( R \otimes {\mathbb 1}^{\otimes 2} ) ( {\mathbb 1} \otimes \beta) 
+ ( {\mathbb 1}^{\otimes 2}  \otimes   \mu ) ( \beta \otimes  {\mathbb 1} )
( {\mathbb 1}^{\otimes 2}  \otimes   R ) (   {\mathbb 1} \otimes R \otimes {\mathbb 1} ) \\
& & - 
\beta ( \mu \otimes  {\mathbb 1}^{\otimes 2} ) 
-
  ({\mathbb 1} \otimes R )( R \otimes     {\mathbb 1} ) (\alpha \otimes {\mathbb 1} ) - (R\otimes \mathbb 1)(\mathbb 1\otimes \alpha)(R\otimes \mathbb 1^{\otimes 2})(\mathbb 1\otimes R\otimes \mathbb 1) .
 \end{eqnarray*}
Define 
 $C^{3,2}_{\rm YI}(V,V) \oplus C^{3,2}_{\rm IY}(V,V)
\rightarrow C^{4,2}_{\rm IY}(V,V)
$
for $\alpha \in C^{3,2}_{\rm YI}(V,V)$ and $\alpha ' \in C^{3,2}_{\rm IY}(V,V)$ 
by
\begin{eqnarray*}
\delta^{3,2}_{\rm YI}( \alpha \oplus \alpha ') 
&=& 
\alpha ' ( \mu \otimes     {\mathbb 1}^{\otimes 2} ) 
+ 
( \mu \otimes  {\mathbb 1} )  (   {\mathbb 1} \otimes R ) (\alpha  \otimes     {\mathbb 1} )
+ 
(  {\mathbb 1} \otimes \alpha ) ( \mu \otimes     {\mathbb 1} )
(   R \otimes {\mathbb 1}^{\otimes 2} ) 
(   {\mathbb 1} \otimes R \otimes {\mathbb 1} ) \\
& & 
- \alpha (   {\mathbb 1}^{\otimes 2} \otimes     \mu ) 
- 
(  {\mathbb 1} \otimes \mu ) ( R \otimes  {\mathbb 1} ) (  {\mathbb 1} \otimes \alpha '   ) 
-
( \alpha '  \otimes {\mathbb 1} ) (  {\mathbb 1} \otimes \mu )
(  {\mathbb 1}^{\otimes 2}  \otimes R )
(    {\mathbb 1} \otimes R \otimes {\mathbb 1} ) . 
 \end{eqnarray*}
Define further
 $\delta^{3,3}_{\rm YI} : C^{3,2}_{\rm YBH}(V,V)  \oplus C^{3,1}_{\rm YBH}(V,V)
\rightarrow C^{4,2}_{\rm YBH}(V,V)   $
for $\alpha \in C^{3,2}_{\rm YBH}(V,V) $ and $\gamma \in  C^{3,1}_{\rm YBH}(V,V)$
by
\begin{eqnarray*}
\delta^{3,3}_{\rm YI}(\alpha \oplus \gamma) 
&=& R (\gamma \otimes {\mathbb 1}) 
+\alpha  ( {\mathbb 1} \otimes \mu \otimes  {\mathbb 1}  )  + ( {\mathbb 1} \otimes  \mu )(R \otimes  {\mathbb 1}) ( {\mathbb 1} \otimes \alpha ) \\
& & 
-\alpha (\mu \otimes  {\mathbb 1}^{\otimes 2} ) 
- ({\mathbb 1} \otimes \mu)  (\alpha \otimes {\mathbb 1})  ( {\mathbb 1}^{\otimes 2} \otimes R)
- ( {\mathbb 1} \otimes \gamma) (R \otimes  {\mathbb 1}^{\otimes 2} ) (  {\mathbb 1} \otimes R \otimes {\mathbb 1}) ( {\mathbb 1}^{\otimes 2} \otimes R) 
.
\end{eqnarray*}
The differential $\delta^{3,1}_{\rm YI}$ corresponds to the diagram in Figure~\ref{YII}, where dotted vertices 
assigned to edges represent cochains.

The terms with image in $K^{5,-1}$ are $d_1 : K^{4,-1} \rightarrow K^{5,-1}$ and $d_2 : K^{3,0} \rightarrow K^{5,-1}$. By direct inspection one sees that this gives the equation for $\delta^{3,3}_{\rm YI}$. 

 The terms in $K^{6,-2}$ 
 are $d_1 : K^{5,-2} \rightarrow K^{6,-2}$ and $d_2 : K^{4,-1} \rightarrow K^{6,-2}$. 
 They give the term $\delta^{3,1}_{\rm YI}$ of the third differential of the YI component of YBH cohomology. 
 
  The only differential with image in $K^{3,1}$ is $d_1 : K^{2,1} \rightarrow K^{3,1}$. This gives
\begin{eqnarray*}
		d_1(\alpha) &=& \alpha(\mu \otimes\mathbb 1^{\otimes 2}) 
		 - (\mathbb 1\otimes \mu)(R\otimes \mathbb 1)(\mathbb 1\otimes \alpha) - (\mathbb 1\otimes \mu)(\alpha \otimes \mathbb 1)(\mathbb 1^{\otimes 2}\otimes R)(\mathbb 1\otimes R \otimes \mathbb 1).
\end{eqnarray*}
This is not the same as $\delta^{3,2}_{\rm YI}$.
 However, we notice that once we include the symmetric component due to the IY part of YBH cohomology, we obtain $\delta^{3,2}_{\rm YI}$. 
 Hence the addition of the two YI and IY components reconstruct $\delta^{3,2}_{\rm YI}$ together.

\subsection{Including braided commutativity to the YBH multicomplex in low dimensions}

We follow an approach in \cite{Harris} in which  commutative algebra cohomology was defined by 
adding cochains related to commutativity.

We  add a condition on the multicomplex directly in total degrees $2$ and $3$ in such a way that $2$ and $3$-cocycles satisfy braided commutativity. We introduce a term of total degree $2$, which we denote by $K^2_{\rm BC} := {\rm Hom}(V^{\otimes 2}, V)$ and add it to $K^2 = K^{2,0} \oplus K^{3,-1}$.
These constraints are given by restricting to those pairs $\psi\oplus \phi$ such that 
\begin{eqnarray*}
		\lambda^2_{\rm BC}(\psi\oplus \psi) :=  \mu  \phi+ \psi  R - \psi
\end{eqnarray*}
vanishes, where $\lambda^2_{\rm BC} : K^2 \rightarrow K^2_{\rm BC}$. Moreover, as already computed we have that  
$$K^3 := \bigoplus_{s+t=3} K^{s,t} = K^{2,1} \oplus K^{3,0} \oplus K^{4,-1} \oplus K^{5,-2}$$
gives the summands of $C^3_{\rm YBH}$. We add a new term $K^3_{\rm BC} := {\rm Hom}(V^{\otimes 3}, V^{\otimes 2})$ to $K^3$.
We then restrict to those cochains such that for $\alpha \in K^{4,-1}$, $\alpha'\in K^{2,1}$, and $\beta\in K^{5,-1}$ 
we impose 
\begin{eqnarray*}
		\lambda^3_{\rm BC}(\alpha \oplus \alpha'\oplus  \beta \oplus \tau) = \delta^3_{\rm BC:YI}(\alpha \oplus \beta \oplus \tau)\oplus \delta^3_{\rm BC:YI}(\alpha' \oplus \beta \oplus \tau)=0 . 
\end{eqnarray*}

Lemma~\ref{lem:deg4} can substantially be reformulated in this context by saying that restricting the multicomplex to those cochains that satisfy the braided commutativity constraints $\lambda^2_{\rm BC}=0$
and $\lambda^3_{\rm BC}=0$ induces a total complex. The corresponding cohomology in low-degrees gives the braided commutative cohomology up to third differentials.  

However, the braided commutative constraint in higher degrees is still not understood. A geometric/diagrammatic approach to determining at in level $n=4$ is given in Section~\ref{sec:BC_n_4}, and 
the complexity of the problem seems to get considerably higher.

\subsection{Diagrammatic approach}\label{sec:BC_n_4}

In this section we provide one of candidate 4-differentials $\delta^{4,1}_{BC}$ and 
show that $\delta^{4,1}_{BC} \delta^3 =0$  using diagrammatic methods.

Let $C^{4,1}_{\rm BC}:= C^{3}_{\rm BC}(X,X) \oplus  C^{3}_{\rm YI}(X,X) \oplus  C^3_{\rm YB}(X,X)$,
we define $\delta^{4,1}_{\rm BC}: C^3_{\rm BC} (X,X) \rightarrow C^{4,1}_{\rm BC}$, 
for $\Phi \in C^{3}_{\rm BC}(X,X)$, $\Psi \in C^{3}_{\rm YI}(X,X)$ and $\Sigma \in C^3_{\rm YB}(X,X)$, 
by the formula
$$
\delta^{4,1}_{\rm BC}(\Phi \otimes \Psi \otimes \Sigma) 
=
 \Phi_1 + \Phi_2 - \Phi_3 - \Phi_4 - \Psi_1 - \Psi_2 +   \Sigma'
$$
where 
\begin{eqnarray*}
\Phi_1 &=&  (R \otimes {\mathbb 1})( {\mathbb 1} \otimes  R) (  \Phi   \otimes    {\mathbb 1} )  , \\
\Phi_2 &=& (R \otimes {\mathbb 1} ) (  {\mathbb 1} \otimes  \Phi     ) ( R \otimes {\mathbb 1} ) ( {\mathbb 1} \otimes R \otimes {\mathbb 1} ),  \\
\Phi_3 &=& ( {\mathbb 1} \otimes R ) ( \Phi \otimes  {\mathbb 1}) ( {\mathbb 1}^{\otimes 2} \otimes R) , \\
\Phi_4 &=& ( {\mathbb 1} \otimes \Phi  ) ( R \otimes {\mathbb 1}^{\otimes 2} ) ( {\mathbb 1} \otimes R \otimes {\mathbb 1} ) ( {\mathbb 1}^{\otimes 2} \otimes R) , \\
\Psi_1 &=&   \Psi , \\
\Psi_2 &=& \Psi ( R \otimes {\mathbb 1}^{\otimes 2} ) , \\
\Sigma' &=& ( {\mathbb 1}^{\otimes 2} \otimes \mu ) \Sigma  .
\end{eqnarray*}

We use the notation  $(  \Phi   \otimes    {\mathbb 1} )$ in $\Phi_1$ to indicate every monomial tensor term in $\Phi$ 
receives extra factor $\otimes    {\mathbb 1}$. For example the first term of $\Phi$ is 
$\alpha (R \otimes  {\mathbb 1})$, so that the first term of $(  \Phi   \otimes    {\mathbb 1} )$ is 
$( \alpha  \otimes  {\mathbb 1}) (R \otimes  {\mathbb 1}^{\otimes 2} )$.

\begin{proposition}
$\delta^{4,1}_{\rm BC} \delta^3_{\rm BC} =0 $.
\end{proposition}

\begin{proof}
We substitute $\delta^3_{\rm BC} (\alpha \otimes \beta \otimes \gamma)$ in  $\Phi \otimes \Psi \otimes \Sigma$ in  the defining formula of 
$\delta^{4,1}_{\rm BC} $ and compute that all terms cancel.

\begin{eqnarray*}
\quad \Phi_1 &=&  (R \otimes {\mathbb 1}) ( {\mathbb 1} \otimes  R)   \{  \delta^3_{\rm BC:YI} ( \alpha \oplus \beta \oplus \gamma ) \otimes  {\mathbb 1} \} \\
&=& (R \otimes {\mathbb 1})  ( {\mathbb 1} \otimes  R) \{ (  [ \alpha  (R \otimes \mathbb 1 ) +
 R ( \gamma \otimes \mathbb 1 ) ]  
- 
[({\mathbb 1} \otimes \mu ) \beta +
 ({\mathbb 1} \otimes  \gamma ) (R \otimes {\mathbb 1} )({\mathbb 1} \otimes  R )+ 
 \alpha ] ) \otimes {\mathbb 1} \} , \\ 
 \quad \Phi_2 &=& (R \otimes {\mathbb 1} ) ( {\mathbb 1} \otimes   \delta^3_{\rm BC:YI} ( \alpha \oplus \beta \oplus \gamma ) ) 
 ( R \otimes {\mathbb 1}^{\otimes 2} ) ( {\mathbb 1} \otimes R \otimes {\mathbb 1} ) \\
&=&
  (R \otimes {\mathbb 1} )   \{  {\mathbb 1} \otimes (  [ \alpha  (R \otimes \mathbb 1 ) +
 R ( \gamma \otimes \mathbb 1 ) ]  
 \\ & & 
- 
[({\mathbb 1} \otimes \mu ) \beta +
 ({\mathbb 1} \otimes  \gamma ) (R \otimes {\mathbb 1} )({\mathbb 1} \otimes  R )+ 
 \alpha ] ) \}  
( R \otimes {\mathbb 1}^{\otimes 2} ) ( {\mathbb 1} \otimes R \otimes {\mathbb 1} ) , \\ 
  \quad \Phi_3 &=& 
   ( {\mathbb 1} \otimes R ) (  \delta^3_{\rm BC:YI} ( \alpha \oplus \beta \oplus \gamma )   \otimes  {\mathbb 1} ) 
   ( {\mathbb 1}^{\otimes 2} \otimes R) \\
  &=&
  ( {\mathbb 1} \otimes R )
      \{  (  [ \alpha  (R \otimes \mathbb 1 ) +
 R ( \gamma \otimes \mathbb 1 ) ]  
 \\ & & 
- 
[({\mathbb 1} \otimes \mu ) \beta +
 ({\mathbb 1} \otimes  \gamma ) (R \otimes {\mathbb 1} )({\mathbb 1} \otimes  R )+ 
 \alpha ]  )  \otimes  {\mathbb 1} \}  
 ( {\mathbb 1}^{\otimes 2} \otimes R) \\ 
  \quad \Phi_4 &=&     
(  {\mathbb 1} \otimes  \delta^3_{\rm BC:YI} ( \alpha \oplus \beta \oplus \gamma )  )  ( R \otimes {\mathbb 1}^{\otimes 2} ) ( {\mathbb 1} \otimes R \otimes {\mathbb 1} ) ( {\mathbb 1}^{\otimes 2} \otimes R) \\
 &=&
   \{   {\mathbb 1} \otimes ( [ \alpha  (R \otimes \mathbb 1 ) +
 R ( \gamma \otimes \mathbb 1 ) ]  
 \\ & & 
- 
[({\mathbb 1} \otimes \mu ) \beta +
 ({\mathbb 1} \otimes  \gamma ) (R \otimes {\mathbb 1} )({\mathbb 1} \otimes  R )+ 
 \alpha ] ) \}  
 ( R \otimes {\mathbb 1}^{\otimes 2} ) ( {\mathbb 1} \otimes R \otimes {\mathbb 1} ) ( {\mathbb 1}^{\otimes 2} \otimes R) \\ 
\end{eqnarray*} 
\begin{eqnarray*} 
   \quad \Psi_1 &=&   \delta^{3,1}_{\rm YI}(\beta \oplus \alpha  ) \\
   &=&
   ( {\mathbb 1} \otimes R ) ( \alpha \otimes {\mathbb 1} ) ( {\mathbb 1}^{\otimes 2}  \otimes R )
+( {\mathbb 1} \otimes \alpha ) ( R \otimes {\mathbb 1}^{\otimes 2} ) (  {\mathbb 1} \otimes R \otimes {\mathbb 1} ) ( {\mathbb 1}^{\otimes 2}  \otimes  R ) 
\\ & & 
+
( {\mathbb 1}^{\otimes 2}  \otimes   \mu ) (    {\mathbb 1} \otimes R \otimes {\mathbb 1} ) 
( R \otimes {\mathbb 1}^{\otimes 2} ) ( {\mathbb 1} \otimes \beta) 
+ ( {\mathbb 1}^{\otimes 2}  \otimes   \mu ) ( \beta \otimes  {\mathbb 1} )
( {\mathbb 1}^{\otimes 2}  \otimes   R ) (   {\mathbb 1} \otimes R \otimes {\mathbb 1} ) \\
& & - 
\beta ( \mu \otimes  {\mathbb 1}^{\otimes 2} ) 
-
( R \otimes     {\mathbb 1} )  ({\mathbb 1} \otimes R )(\alpha \otimes {\mathbb 1} ) 
- (R\otimes \mathbb 1)(\mathbb 1\otimes \alpha)(R\otimes \mathbb 1^{\otimes 2})(\mathbb 1\otimes R\otimes \mathbb 1) \\  
 \quad  \Psi_2 &=&  \delta^{3,1}_{\rm YI}(\beta \oplus \alpha  ) ( R \otimes {\mathbb 1}^{\otimes 2} ) \\
  &=&
  \{
     ( {\mathbb 1} \otimes R ) ( \alpha \otimes {\mathbb 1} ) ( {\mathbb 1}^{\otimes 2}  \otimes R )
+( {\mathbb 1} \otimes \alpha ) ( R \otimes {\mathbb 1}^{\otimes 2} ) (  {\mathbb 1} \otimes R \otimes {\mathbb 1} ) ( {\mathbb 1}^{\otimes 2}  \otimes  R ) 
\\ & & 
+
( {\mathbb 1}^{\otimes 2}  \otimes   \mu ) (    {\mathbb 1} \otimes R \otimes {\mathbb 1} ) 
( R \otimes {\mathbb 1}^{\otimes 2} ) ( {\mathbb 1} \otimes \beta) 
+ ( {\mathbb 1}^{\otimes 2}  \otimes   \mu ) ( \beta \otimes  {\mathbb 1} )
( {\mathbb 1}^{\otimes 2}  \otimes   R ) (   {\mathbb 1} \otimes R \otimes {\mathbb 1} ) \\
& & - 
\beta ( \mu \otimes  {\mathbb 1}^{\otimes 2} ) 
-
 ( R \otimes     {\mathbb 1} )  ({\mathbb 1} \otimes R )(\alpha \otimes {\mathbb 1} )
  \\ & & 
  - (R\otimes \mathbb 1)(\mathbb 1\otimes \alpha)(R\otimes \mathbb 1^{\otimes 2})(\mathbb 1\otimes R\otimes \mathbb 1 )
  \}
  ( R \otimes {\mathbb 1}^{\otimes 2} ) \\ 
\quad \Sigma' &=&  ( {\mathbb 1}^{\otimes 2} \otimes \mu ) 
 \delta^3_{\rm YB}(\beta) \\
 &=&
  ( {\mathbb 1}^{\otimes 2} \otimes \mu ) 
  \{ 
 ( R \otimes  {\mathbb 1}^{\otimes 2}  )   ( {\mathbb 1} \otimes R \otimes {\mathbb 1} )  ( {\mathbb 1}^{\otimes 2} \otimes R )  (  \beta \otimes {\mathbb 1} ) 
 \\ & & 
  +
  ( R \otimes  {\mathbb 1}^{\otimes 2}  ) ( {\mathbb 1} \otimes \beta )  ( R \otimes  {\mathbb 1}^{\otimes 2}  )  ( {\mathbb 1} \otimes R \otimes {\mathbb 1} ) 
   \\ 
 &  & +
 (  \beta \otimes {\mathbb 1} )  ( {\mathbb 1}^{\otimes 2} \otimes R )    ( {\mathbb 1} \otimes R \otimes {\mathbb 1} ) 
  +
   ( {\mathbb 1}^{\otimes 2} \otimes R )  ( {\mathbb 1} \otimes R \otimes {\mathbb 1} )   ( R \otimes  {\mathbb 1}^{\otimes 2}  )   ( {\mathbb 1} \otimes \beta )  
   \\
   & & - 
   (  \beta \otimes {\mathbb 1} )  ( {\mathbb 1}^{\otimes 2} \otimes R )  ( {\mathbb 1} \otimes R \otimes {\mathbb 1} )   ( R \otimes  {\mathbb 1}^{\otimes 2}  )  
   -
  ( {\mathbb 1} \otimes R \otimes {\mathbb 1} )   ( R \otimes  {\mathbb 1}^{\otimes 2}  )   ( {\mathbb 1} \otimes \beta )    ( R \otimes  {\mathbb 1}^{\otimes 2}  ) 
  \\
  & &  - 
    ( {\mathbb 1} \otimes R \otimes {\mathbb 1} )    ( {\mathbb 1}^{\otimes 2} \otimes R ) (  \beta \otimes {\mathbb 1} )     ( {\mathbb 1}^{\otimes 2} \otimes R )   
    -
   ( {\mathbb 1} \otimes \beta )     ( R \otimes  {\mathbb 1}^{\otimes 2}  )
   ( {\mathbb 1} \otimes R \otimes {\mathbb 1} )    ( {\mathbb 1}^{\otimes 2} \otimes R )  
  \}
\end{eqnarray*}

To specify canceling pairs of terms, we denote each term by $\Phi_i$, $\Psi_j$, $\Sigma'$ and the order os terms in each.
For example, $(\Phi_1, 3)$ denote the first term that appear in the preceding list os terms at the third position of $\Phi_1$,
that is $- (R \otimes {\mathbb 1})( {\mathbb 1} \otimes  R) {\mathbb 1} \otimes \mu ) \beta $. Note that parenthesis is expanded
and the sign is kept. We also note that $\Phi_2$ has a negative sign in the differential, so that all its terms have reversed sign.
Then the following pairs are canceled because they are the labels assigned to the same edges by two polygons:

$$
\begin{array}{llll}
[ (\Phi_1, 1), (\Psi_2, 6) ] , &  [ ( \Phi_1, 4), (\Phi_2, 2) ]  , & [ ( \Phi_1 , 5), - (\Psi_1, 6) ] , &
[ ( \Phi_2, 3), (\Sigma', 2) ] , \\ \relax
[ (\Phi_2, 5), - (\Psi_1, 7) ] , &
[ - (\Phi_3, 1), -  (\Psi_2, 1) ] , & [ - (\Phi_3, 4), - (\Phi_4, 2) ] ,   &   [ - (\Phi_3, 5), - (\Psi_1, 1) ] , \\ \relax 
 [ - (\Phi_4, 3) ,  ( \Sigma', 1, 8 ) ] ,  &
  [ -  (\Phi_4, 5), - (\Psi_1, 2) )  ] . & & 
\end{array}
$$


There is another family of pairs of canceled edges, that come from squares in the diagram.
Two  pairs of parallel edges of a square is labeled by canceling terms modulo 2.
For example, one of the two parallel edges in the square marked by a shaded $\square$ at the left top of the diagram 
is $[ - (\Psi_2, 7), (\Phi_2,1) ]$, where
\begin{eqnarray*}
- ( \Psi_2, 7) &=&      (R\otimes \mathbb 1)(\mathbb 1\otimes \alpha)(R\otimes \mathbb 1^{\otimes 2})(\mathbb 1 \otimes R\otimes \mathbb 1 )
  ( R \otimes {\mathbb 1}^{\otimes 2} )  \\
(\Phi_2,1) &=& (R \otimes {\mathbb 1} )     ( {\mathbb 1} \otimes    \alpha  ) ( {\mathbb 1} \otimes  R \otimes \mathbb 1 )  ( R \otimes {\mathbb 1}^{\otimes 2} ) ( {\mathbb 1} \otimes R \otimes {\mathbb 1} ) 
\end{eqnarray*}
and they are the same after applying YBE. Hence with $\Z_2$ coefficient they cancel.
The other parallel edges of this square gives the  pair $[ (\Phi_1, 3) ( \Sigma', 1) ] $. 
The other pairs cancel modulo 2  in a similar manner, after applying one of the equations assumed.
Such pairs that appear in pairs as parallel edges of the bottom two squares are:
$$
[ (\Psi_1, 5), (\Psi_2, 5) ] ,  [ ( \Phi_1, 2) , \Phi_3, 2) ] ,  [ ( \Phi_3,3), ( \Sigma' , 7)], [ (\Psi_2, 2), ( \Phi_4, 1) ] . 
$$
There are two adjacent squares at the top right corner of the figure, 
and the adjacent edge cancel, so that we obtain the following pairs for these two squares:
$$ 
[ ( \Psi_1, 4), (\Sigma', 3)],  [ ( \Psi_1, 3), (\Sigma', 4)], [ ( \Phi_2, 4), ( \Phi_4, 4) ] .
$$
This completes all pairs, and the statement follows.
\end{proof}

A diagram representing this differential is shown in Figure~\ref{fig:BCd41}. 
Since 3-differentials are represented by graph polygons as in Figures~\ref{BCdiff3} and \ref{YII},
these polyhedra also represent generating 3-cochains.
Then it is expected that 4-differentials are represented by polytopes formed by faces consisting of these polyhedra representing 3-cochains, and the summation of these 3-cochcins according to union of faces represent a 4-cocycle condition.
The figure shows a projection of such a polytope in the plane.
There are four pentagons representing $\Phi$ that corresponds to Figures~\ref{fig:BCd41}, two heptagons for $\Psi$ corresponding to Figure~\ref{YII}
 and an octagon for $\Sigma$ corresponding to Figure~\ref{fig:tetra}.

\begin{figure}[htb]
	\begin{center}
		\includegraphics[width=5.5in]{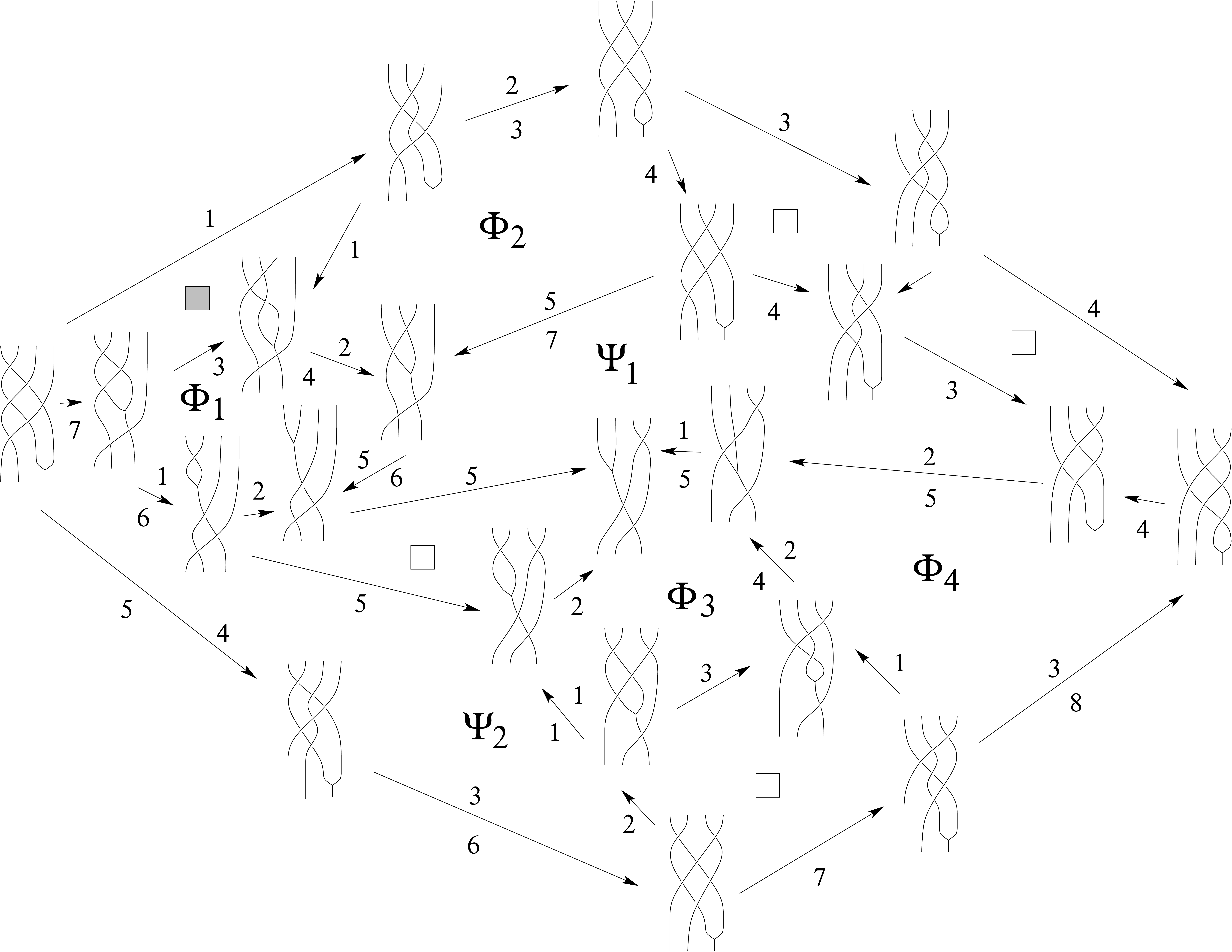}
	\end{center}
	\caption{}
	\label{fig:BCd41}
\end{figure}

In the figure, for each polygon that represent $\Phi$ or $\Psi$, the numbers are indicated to edges inside the polyhedron.
These numbers represent the order of terms in the defining formulas of $\Phi_i$ and $\Psi_j$. 
Each edge is shared by two polyhedra, and the corresponding terms cancel.
For example, the edge of $\Phi_1$ labeled 1 is shared by the edge 6 of $\Psi_2$, indicating these two terms cancel. 
The square regions marked by $\square$ represent equalities of distant maps commute.
These squares require a characteristic 2 ground ring, and it remains to be seen whether fourth differentials can be defined 
over integers.

There are other 4-differential that can be obtained in a similar manner.
Finding  formulas of all 4-differentials in such a way that $\delta^2=0$ is satisfied
using diagrams appears equally complex. Although the diagrammatic approach seems helpful to obtain some insights,
 developing other methods for higher dimensions is desirable.

\end{document}